\documentclass[a4paper, reqno, 12pt]{amsart}
\usepackage{tikz-cd}
\usepackage[all]{xy}
\usepackage{amssymb}
\usepackage{amsmath,amscd,color}
\usepackage[all]{xy}
\usepackage{enumerate}
\usepackage{mathrsfs}
\usepackage{tensor}
\usepackage{stackrel}
\usepackage{hyperref}
\usepackage{mathrsfs}
\usepackage{hyperref}
\usepackage{mathtools}
\usepackage{enumerate}
\usepackage[margin=1.2 in]{geometry}
\usepackage{longtable}
\usepackage{booktabs}

\definecolor{tocolor}{rgb}{.1,.1,.5}
\definecolor{urlcolor}{rgb}{.2,.2,.6}
\definecolor{linkcolor}{rgb}{.1,.1,.6}
\definecolor{citecolor}{rgb}{.6,.2,.1}
\hypersetup{ colorlinks=true, urlcolor=urlcolor,
  linkcolor=linkcolor, citecolor=citecolor}
\definecolor{darkgreen}{rgb}{0.0, 0.5, 0.0}

\providecommand{\U}[1]{\protect\rule{.1in}{.1in}}
\newtheorem{theorem}{Theorem}[section]
\newtheorem*{theorem*}{Theorem}
\newtheorem*{claim*}{Claim}
\newtheorem{theoremM}{Theorem}

\newtheorem{corollary}[theorem]{Corollary}

\newtheorem{proposition}[theorem]{Proposition}
\newtheorem*{proposition*}{Proposition}
\newtheorem{lemma}[theorem]{Lemma}

\theoremstyle{definition}
\newtheorem{definition}[theorem]{Definition}

\newtheorem{example}[theorem]{Example}
\newtheorem{remark}[theorem]{Remark}
\numberwithin{equation}{section}

\newcommand{\mf}[1]{\mathfrak{#1}}

\newcommand{\del}{\partial}
\newcommand{\toto}{\rightrightarrows}

\newcommand{\RR}{\mathbb{R}}                    
\newcommand{\NN}{\mathbb{N}}                    
\newcommand{\ZZ}{\mathbb{Z}}                    

\renewcommand{\SS}{\mathbb{S}}                    

\newcommand{\LL}{\mathcal{L}}                   
\newcommand{\GG}{\mathcal{G}}                   
\newcommand{\tto}{\rightrightarrows}            
\renewcommand{\:}{\colon}                       
\renewcommand{\d}{{\mathrm{d}}}                   
\newcommand{\D}{{\mathrm{D}}}                   
\newcommand{\DD}{\mathcal{D}}                   
\newcommand{\TT}{\mathcal{T}}                   
\newcommand{\FF}{\mathcal{F}}                   
\newcommand{\bas}{\text{\textnormal{bas}}}      
\newcommand{\para}{\parallel}

\newcommand{\Hom}{\operatorname{Hom}}       
\newcommand{\End}{\operatorname{End}}       
\newcommand{\MC}{\operatorname{MC}}         
\newcommand{\Jac}{\operatorname{Jac}}       
\newcommand{\id}{\operatorname{id}}         
\newcommand{\im}{\operatorname{im}}         
\newcommand{\Der}{\operatorname{Der}}       
\renewcommand{\MC}{\operatorname{MC}}         
\newcommand{\UB}{\Breve{B}}                        
\newcommand{\UD}{\Breve{\mathrm{D}}}                        

\newcommand{\Tor}{\Theta}              

\newcommand{\Curv}{\Omega}                  
\newcommand{\Cc}{\mathcal{C}}              

\newcommand{\rk}{\operatorname{rank}}         
\newcommand{\dom}{\operatorname{dom}}       
\DeclareMathOperator{\pr}{pr}               
\DeclareMathOperator{\germ}{germ}           
\newcommand{\red}{\mathrm{red}}             
\newcommand{\pullback}[2]{{}_{#1}\kern-\scriptspace{\times}_{#2}}                   
\newcommand{\sys}{\mathrm{sys}}             

\newcommand{\onabla}{{\overline{\nabla}}}   

\DeclareMathOperator{\GL}{GL}               

\newcommand{\al}{\alpha}
\newcommand{\be}{\beta}

\begin{document}
\title{Relative algebroids and Cartan realization problems}

\author{Rui Loja Fernandes}
\address{Department of Mathematics, University of Illinois at Urbana-Champaign, 1409 W. Green Street, Urbana, IL 61801 USA}
\email{ruiloja@illinois.edu}

\author{Wilmer Smilde}
\address{Department of Mathematics,
Cornell University, 212 Garden Avenue, Ithaca, NY 14853 USA}
\email{wsmilde@cornell.edu}

\date{\today}

\begin{abstract}
    We introduce the notion of a relative algebroid to address existence and classification problems for geometric structures governed by partial differential equations. This framework unifies the theories of ordinary Lie algebroids and (formal) partial differential equations. We study prolongation and formal integrability of a relative algebroid, leading to a profinite Lie algebroid that governs the classification problem. At the core of this new framework is the notion of a relative derivation, whose theory we develop in detail. Throughout, we illustrate the theory with numerous examples.
\end{abstract}

\thanks{This work was partially supported by NSF grant DMS-2303586 and UIUC Campus Research Board Award RB25014.}

\maketitle

\setcounter{tocdepth}{1}
\tableofcontents

\section*{Introduction}

The role of algebroids in certain classification problems was first explicitly recognized by Bryant in his work on the classification of Bochner-Kähler metrics \cite{Bryant2001}. Struchiner and the first author built on this insight, establishing precise connections between existence and classification problems and the integrability of the underlying Lie algebroid \cite{FernandesStruchiner2014, FernandesStruchiner2019, FernandesStruchiner2021}, particularly in relation with the \emph{global} solutions to the classification problem. However, their work is restricted to cases where the local classification is finite-dimensional in nature, which excludes most examples. In this paper, we initiate the study of classification problems without this restriction. For this, we introduce the new concept of \emph{relative Lie algebroid}, which provides a powerful tool for understanding general existence and classification problems for geometric structures. Moreover, the framework of relative algebroids, besides explaining known examples of algebroid-like structures, provides a systematic and coordinate-invariant language for understanding geometric structures defined by invariant PDEs.

The notion of a relative algebroid has roots in the work of Bryant, particularly in his unpublished notes on Lie Theory and Exterior Differential Systems \cite{Bryant2014}. There, he observed through various examples that many existence problems can be recast in a particular form —referred to here as \emph{Bryant's equations}— which resembles an algebroid but is not quite one \cite{Bryant1998-1999, Bryant2014}. Although Bryant did not identify this notion explicitly, he showed that techniques going back to Cartan can be applied to solve these equations \emph{locally} in many interesting examples. Our ultimate aim is to understand the \emph{global} (or complete, maximal) solutions to a classification problem. The framework of relative algebroids enables the use of integration techniques from Lie theory to solve global aspects of generic geometric problems.

The algebraic structure underlying the notion of a relative algebroid is what we call a \emph{relative derivation} (also called an \emph{augmented derivation} on \href{http://ncatlab.org/nlab/show/derivation}{nLab}). A significant part of the paper is dedicated to developing the theory of such relative derivations in the context of differential geometry. To our knowledge, these objects have not yet been systematically studied in the literature, although they appear to be present in many other contexts, such as deformation theory and control theory.

In the rest of this introduction, we will first explain what Bryant's equations are, where they come from and how they are used in existence and classification problems. Then, we will briefly describe the contents of this paper.

\subsection*{Bryant's equations and two examples}
Let $(M,g)$ be a Riemannian manifold. Recall that its orthogonal frame bundle 
    \[
        \pi\: P \to M.
    \]
carries a coframe 
    \[
        (\theta, \omega)\: TP \to \RR^n\oplus \mf{o}(n).
    \]
whose components are the tautological form (also called the solder form) and the Levi-Civita connection 1-form. Together they satisfy Cartan's structure equations \begin{equation}\label{eq:IntroductionCartanStructureEquationsRiemannian}
            \begin{cases} 
                \d \theta = - \omega \wedge \theta,\\
                \d \omega = R(\theta\wedge\theta) - \omega\wedge \omega,
            \end{cases} 
    \end{equation}
where $R\: P\to \Hom(\wedge^2\RR^n, \mf{o}(n))$ is a map inducing the Riemann curvature tensor when passing to the associate bundles.

The structure equations (\ref{eq:IntroductionCartanStructureEquationsRiemannian}) completely characterize the orthonormal frame bundles of Riemannian manifolds, and allow one to relate existence and classification problems for Riemannian structures to those of coframes, as observed by Cartan himself. In a classification problem one typically has restrictions on the curvature $R$ in the form of an algebraic or a differential equation (or both). Let us recall one simple example where one already sees the appearance of a Lie algebroid.

\begin{example}[Space forms]
For the classification of Riemannian manifolds $(M, g)$ with constant sectional curvature, a.k.a.~\textbf{space forms}, the Riemann curvature $R\: P \to \Hom(\wedge^2\RR^n, \mf{o}(n))$ takes the special form:
    \[
        R(K)(u, v)w = K\left( \langle w, u\rangle v - \langle w, v\rangle u \right),
    \]
where $K$ is a constant (the scalar curvature), $u, v, w\in \RR^n$ and $\langle\cdot, \cdot\rangle$ is the Euclidean inner product on $\RR^n$. So one considers Cartan's structure equations with this specific shape of the tensor together with the equation that $K$ is constant: 
    \begin{equation}\label{eq:SpaceFormStructureEquations}
        \begin{cases}
            \d \theta = -\omega \wedge \theta\\
            \d \omega = R(K)(\theta\wedge\theta) - \omega \wedge \omega\\
            \d K = 0.
        \end{cases}
    \end{equation}
These are Bryant's equations for this classification problem. 

Equations \eqref{eq:SpaceFormStructureEquations} already show the appearance of a Lie algebroid behind this classification problem of space forms. We ``forget" about the underlying manifold these equations were derived from and look at the system ``\emph{as is}''. More precisely, we consider the trivial vector bundle $A = \underline{\RR^n\oplus \mf{o}(n)}\to \RR$, interpret $K$ as a coordinate function on the base, and view $\theta$ and $\omega$ as generating sections of the vector bundle $A^*$. Equations \eqref{eq:SpaceFormStructureEquations} then define a linear operator 
\[ 
\D\:\Omega^\bullet(A)\to \Omega^{\bullet+1}(A), \quad \text{where }\Omega^k(A):=\Gamma(\wedge^kA^*).
\]
This is a graded derivation satisfying $\D^2=0$, hence it defines a Lie algebroid structure on $A$. In \cite{FernandesStruchiner2019} it is explained how integration of this Lie algebroid leads to the well-known classification of space forms.
\end{example}

In the previous example the ``moduli space'' of solutions is one-dimensional: there is only one invariant, the scalar curvature. For such classification problems of finite type, a complete theory was developed in \cite{FernandesStruchiner2014, FernandesStruchiner2019, FernandesStruchiner2021}. However, in practice, such problems are rarely encountered. Our next example illustrates a more typical occurrence.

\begin{example}[Surfaces with $|\nabla K | = 1$, {\cite[\S 5.1]{Bryant2014}}]\label{ex:introductionSurfaces|nablaK|=1} 
    Consider the problem of classifying Riemann surfaces $(M, g)$ whose scalar curvature $K$ satisfies the differential equation $|\nabla K| =1$.
    
    Let $(P, \theta, \omega)$ be the orthonormal frame bundle of such a surface and identify $\RR^2 \oplus \mf{o}(2)\cong \RR^3$. Writing $(\theta, \omega) = (\theta^1, \theta^2, \theta^3)$. The structure equations become
    \[
    \begin{cases}
        \d \theta^1 = - \theta^3 \wedge \theta^2,\\
        \d \theta^2 = \theta^3 \wedge\theta^1, \\
        \d \theta^3 = K \theta^1 \wedge\theta^2,
    \end{cases}
    \]
    where $K$ is the Gauss curvature. Since $K$ is an $O(2)$-invariant function on $P$, its derivative can be written in the form
    \[
        \d K = K_1 \theta^1 + K_2 \theta^2.
    \]
    The equation $|\nabla K|=1$ becomes $K_1^2 + K_2^2 = 1$, so the components $K_i$ are related through a single coordinate $\varphi$ on the circle. The structure equations augmented by these conditions on the curvature take the form:
    \begin{equation}\label{eq:introSurfacesStructureEquations}
        \begin{cases}
            \d \theta^1 = - \theta^3 \wedge \theta^2,\\
            \d \theta^2 = \theta^3 \wedge\theta^1, \\
            \d \theta^3 = K \theta^1 \wedge\theta^2,\\
            \d K = \cos (\varphi) \theta^1 + \sin (\varphi) \theta^2.
        \end{cases}
    \end{equation}
    These are Bryant's equations for this classification problem.
    
    As in the previous example, the next step is to ``forget" about the manifold $P$ and look at the equations ``\emph{as they are}''. So, as before, we consider the trivial vector bundle $A = \underline{\RR^3} \to \RR$, where $\RR$ has coordinate $K$, and think of the $\theta^i$'s as generating sections of the dual vector bundle. But now we run into a problem: equations \eqref{eq:introSurfacesStructureEquations} define a degree 1 linear operator $\D$ whose values depend on the extra variable $\varphi$. To solve this issue we consider the vector bundle $B=\underline{\RR^3} \to \SS^1 \times \RR$, where $\SS^1\times \RR$ has coordinates $(\varphi, K)$, so that now
    \[
        \D\:\Omega^\bullet(A)\to \Omega^{\bullet+1}(B)
    \]
    
    If $p\: \SS^1\times \RR\to \RR$ is the projection, so that $B = p^*A$, it makes sense to call $\D$ a \emph{derivation relative to $p$}. This does not define a Lie algebroid anymore since, for instance, the equation $\D^2=0$ doesn't make sense at this point. So we cannot directly use integration techniques for Lie algebroids to obtain solutions.

    One could try to look for a derivation $\D_B$ on $B$ that extends $\D$ and squares to zero by adding the extra equation $\D_B^2 K =0$. This implies
    \begin{align*}
        0 = \D_B^2 K &= \D_B(\cos (\varphi)\theta^1 + \sin (\varphi) \theta^2)\\
         &= -\sin(\varphi) \D_B \varphi\wedge\theta^1 -\cos(\varphi)\theta^3\wedge\theta^2 + \cos(\varphi)\D_B \varphi \wedge \theta^2 + \sin(\varphi)\theta^3\wedge\theta^1.
    \end{align*}
    It follows that, for this equation to hold, we must have
    \begin{equation}
        \D_B \varphi = \theta^3 + c_0( - \sin(\varphi)\theta^1 + \cos(\varphi)\theta^2),
    \end{equation}
    where $c_0$ is a new independent variable. So, again $\D_B$ is not an actual derivation but a derivation relative to the projection $\RR\times \SS^1 \times \RR\to \SS^1\times \RR$, $(c_0, \varphi, K)\mapsto (\varphi, K)$. 

    This process never stops, and for this reason there is no finite-dimensional Lie algebroid governing this classification problem. We will return to this example in Section \ref{sec:postlude}.
\end{example}

These, and many other examples, led Bryant to observe that many classification problems amount to solving a problem involving data consisting of a principal bundle $P$ with a coframe $\theta\:TP\to \RR^n$, together with functions $(a^\mu,b^\varrho)\:P\to\RR^s\times\RR^r$, satisfying \textbf{Bryant's equations}:
\begin{equation}\label{eq:BryantsEquations}
    \begin{cases}
        \d \theta^i = -\frac{1}{2}c^i_{jk}(a)\, \theta^j\wedge\theta^k,\\
        \d a^\mu = F^\mu_i(a, b)\, \theta^i,
    \end{cases}
\end{equation}
for some given functions  $c^i_{jk}\:\RR^s\to \RR$ and $F^\mu_i\:\RR^s\times \RR^r\to\RR$.

As in the two problems above, ``forgetting'' about the underlying bundle $P$, treating $(a,b)$ as independent coordinates on $\RR^s\times \RR^r$ and $\theta$ as sections of the bundle dual to the trivial vector bundle $A:=\underline{\RR^n}\to \RR^r\times \RR^s$, Bryant's equations define a degree 1, linear operator
   \[
        \D\:\Omega^\bullet(A)\to \Omega^{\bullet+1}(B),\quad \text{with } B=p^*A,
    \]
where $p\:\RR^s\times \RR^r\to \RR^s$ is the projection.

\subsection*{PDEs and derivations} Degree 1 differential operators similar to the ones in the previous example are well-known in the formal theory of PDEs. For example, there one considers the so-called \emph{total exterior differential}. In its simplest form, one has a function in jet space, say $f=f(x,y,u)$, where $x,y$ are independent variables and $u$ is a dependent variable, and then its total differential is given by
\[ \D f:=(\D_xf)\, \d x+(\D_y f)\,\d y=(f_x+f_u u_x) \d x+(f_y+f_u u_y)\d y. \]
Similarly, for a 1-form $\alpha=a(x,y,u)\d x+b(x,y,u)\d y$ its total differential is defined by
\[
    \D\alpha:=\big(\D_x b-\D_y a\big)\d x\wedge \d y=\big(b_x+b_u u_x-a_y-a_uu_y\big)\d x\wedge \d y.
\]
If $q\:\RR^3\to \RR^2$ is the submersion $(x,y,u)\mapsto (x,y)$ and we denote by $J^1q$ its 1st jet bundle with coordinates $(x,y,u,u_x,u_y)$, then one can think of $\D$ as a degree 1, linear operator
    \[
        \D\:\Omega^\bullet(A)\to \Omega^{\bullet+1}(B),\quad \text{where } A:=q^*T\RR^2,\ B:=p^*A.
    \]

\subsection*{Relative derivations, relative algebroids and the contents of this paper}

The common geometric structure underlying the previous examples is captured by a \textbf{relative derivation}. More precisely, given a vector bundle $A\to N$ and a map $p\: M\to N$, a degree 1 derivation relative to $p$ is a linear operator
\[
    \D\: \Omega^\bullet(A)\to \Omega^{\bullet+1}(p^*A),
\]  
satisfying
\[
    \D(\alpha\wedge\beta) = (\D\alpha)\wedge p^*\beta + (-1)^{|\alpha|}p^*\alpha\wedge \D\beta,  
\]
for homogeneous elements $\alpha, \beta\in \Omega^\bullet(A)$. The resemblance with Lie algebroids now becomes clear: a Lie algebroid is a degree 1 derivation $\D$ on a vector bundle $A\to N$ relative to the identity map, which additionally satisfies $\D^2 = 0$. A general relative derivation $\D$ cannot be squared since domain and codomain are distinct. Hence, the equation $\D^2=0$ does not seem to make sense. In this paper, we will develop the necessary theory to make sense of this in a natural way. 

In Section \ref{sec:relativeDerivations}, we develop the theory of \textbf{relative degree $k$ derivations}, for arbitrary $k$. Just like ordinary degree $k$ derivations which can be viewed, dually, as $k$-nary brackets, relative derivations can be viewed, dually, as \textbf{relative $k$-brackets}. Besides analyzing the structure of relative multiderivations and relative multibrackets, we will extend several notions from the theory of (formal) PDEs to this setting. These include
\begin{itemize}
    \item a notion of a \emph{tableau (bundle) of derivations},
    \item a \emph{Spencer complex} for tableaux of derivations,
    \item \emph{involutivity} and \emph{Cartan characters} for relative derivations.
\end{itemize}
The computations and examples in \cite{Bryant2014} suggest that Bryant may have already been aware of such notions. We intend to place his computations within a natural and robust framework.

In Section \ref{sec:RelativeAlgebroids}, we introduce and study the central concept in this paper, namely the notion of \textbf{relative algebroid}. This is a triple $(A, p, \D)$, where $A\to N$ is a vector bundle and $\D$ is a degree 1 derivation relative to a map $p\:M\to N$. As we already pointed out, these are the geometric objects underlying Bryant's equations. Similar to the theory developed in \cite{Bryant2001,FernandesStruchiner2014,FernandesStruchiner2019}, solutions to Bryant's equations are translated to the notion of \textbf{realization} of a relative algebroid, so one speaks of \emph{Cartan's realization problem}. In the language of derivations, a realization allows to realize a relative derivation as a manifold with the de Rham differential. More precisely, given a relative algebroid $(A, p, \D)$, a realization is a manifold $P$ together with a bundle map $\theta\: TP \to p^*A$ that is fiberwise an isomorphism and is also a morphism of relative algebroids, i.e., satisfies
\[ \d \circ \theta^* \circ p^* = \theta^* \circ \D. \]

In the case of finite-type problems, realizations naturally appear as fibers of an integrating Lie groupoid, making it possible to apply the power of Lie theory to solve classification problems (see \cite{FernandesStruchiner2014, FernandesStruchiner2019}). In the general case, however, the situation is much more complex. As we will see, the theory of relative derivations unifies Lie algebroids and PDEs. 

We will show first that the formal theory of the existence of realizations parallels that of partial differential equations, as developed by Spencer and his school, and later formalized by Goldschmidt in \cite{Goldschmidt1967}. Specifically, we will show that first-order and second-order obstructions to the existence of realizations are captured by Spencer cohomology classes, called the \textbf{torsion class} and the \textbf{curvature class}. 

In Section \ref{sec:prolongation}, we will see that the vanishing of the torsion class leads to the notion of prolongation, which seeks to complete the derivation of a relative algebroid to an operator that squares to zero. More concretely, a \textbf{prolongation} of a relative algebroid $(A, p, \D)$ is another relative algebroid $(B, p_1, \D^{(1)})$, with projection $p_1\:M^{(1)}\to M$. Here, $B=p^*A$ and the derivation is a map $\D^{(1)}\: \Omega^\bullet(B)\to \Omega^{\bullet + 1} (B^{(1)})$, where $B^{(1)}:= p_1^*B$, such that $\D^{(1)}\circ \D = 0$. Schematically, this is described by the diagram
\[
\xymatrix@C=40pt{
    B^{(1)}=p_1^*B \ar[r] \ar[d] & B=p^*A \ar[d] \ar[r] \ar@{-->}@/^1pc/[l]^{\D^{(1)}} & A \ar[d] \ar@{-->}@/^1pc/[l]^{\D} \\
    M^{(1)} \ar[r]_{p_1} & M  \ar[r]_{p} & N}
\]

The existence of the prolongation is contingent on the vanishing of the torsion class. Higher prolongations are obtained by iterating the first prolongations and are indexed by $k$. A relative algebroid is called \textbf{$k$-integrable} if all prolongations up to order $k$ exist, and is called \textbf{formally integrable} if it is $k$-integrable for all $k\ge 1$. Foundational results from the theory of formal PDEs, such as Goldschmidt's formal integrability theorem \cite[Thm. 8.1]{Goldschmidt1967}, have natural generalizations to the theory of relative algebroids, namely we will prove the following analogue (or rather, extension) of that result for relative algebroids.

\begin{theoremM}[Theorem \ref{thm:GoldschmidtsFormalIntegrabilityCriterion}]
    Let $(A,p,\D)$ be a relative algebroid with tableau of derivations $\tau$. Suppose that: 
    \begin{enumerate}[(i)]
        \item $(A,p,\D)$ is 1-integrable;
        \item the Spencer cohomology groups $H^{k, 2}(\tau)$ vanish for all $k\geq 0$.
    \end{enumerate}
    Then $(A,p,\D)$ is formally integrable.
\end{theoremM}

If, for some $k$, the map $p_k\: M^{(k)}\to M^{(k-1)}$ is a diffeomorphism, the derivation $\D^{(k)}$ actually defines the structure of a Lie algebroid, and we are in the \textbf{finite-type} case. However, in general -- as in Example \ref{ex:introductionSurfaces|nablaK|=1} -- this process does not stop, and the full prolongation tower defines a \textbf{profinite Lie algebroid.} For this reason, we call a relative algebroid a \textbf{relative \emph{Lie} algebroid} if all prolongations exist. Finding solutions to the realization problem in this more general case is much more challenging. The only general statement we can make is in the \emph{analytic} setting, where an existence result, essentially due to Cartan and K\"ahler -- see Bryant \cite[Thm.~3 and Thm.~4]{Bryant2014} -- holds. We state it as follows:

\begin{theoremM}[Theorem \ref{thm:CartanBryantExistence}]
   Let $(A,p,\D)$ be an analytic relative Lie algebroid. For each $k$ and each $x\in M^{(k)}$, there exists a realization through $x$.
\end{theoremM}

In Section \ref{sec:constructions}, we consider several important constructions with relative algebroids. First, given some vector bundle $A\to N$, we construct a \textbf{universal relative algebroid} $(A,p_1,\UD)$, which is formally integrable, with the property that every algebroid $(A,p,\D)$ relative to a submersion $p\:M\to N$ is a pullback of the universal one via a classifying map. Then we discuss the operation of \textbf{restriction} to subspaces and how this operation interacts with prolongations and realizations. Finally, we introduce the notion of \textbf{systatic foliation} of a relative algebroid, which captures the directions in which the tableau map is zero. We show that these directions are essentially ``redundant" from the perspective of the realization problem, since the original almost relative algebroid descends to a reduced almost relative algebroid, which has essentially the same realizations and the same integrability properties.

In Section \ref{sec:PDE}, our discussion comes to a full circle by showing that any partial differential equation can be recast as a relative algebroid in such a way that the formal theory of prolongations of PDEs coincides with the prolongation theory for relative algebroids. For this, we first interpret the Cartan distribution on a jet space as a relative connection, and then we show that any relative connection has an associated relative algebroid. The derivation corresponding to the Cartan distribution is nothing more than the horizontal differential in the first row of the variational bicomplex. Then, given a $k$-th order PDE $E$ viewed as a submanifold in the jet space $J^kq$ of a submersion $q\:N\to X$, by pulling back the Cartan derivation one obtains the associated relative algebroid. We will prove the following result:

\begin{theoremM}[Theorem \ref{thm:PDEsAndRelativeAlgebroids}]
    Let $E\subset J^kq$ be a PDE. Then, germs of solutions to $E$ are in one-to-one correspondence with germs of realizations of the associated relative algebroid, modulo diffeomorphisms. Moreover:  
    \begin{enumerate}[(i)]
        \item $E$ is a 1-integrable PDE if and only if the associated relative algebroid is 1-integrable.
        \item If $E$ is a 1-integrable PDE, then the relative algebroid corresponding to its first prolongation $E^{(1)}\subset J^{k+1}q$ is the prolongation of the relative algebroid associated with $E$.
    \end{enumerate}
    In particular, a PDE is formally integrable if and only if its associated relative algebroid is.
\end{theoremM}

In Section \ref{sec:postlude}, we return to Example \ref{ex:introductionSurfaces|nablaK|=1} and discuss its solutions (i.e., realizations) in light of the theory developed in the previous sections.  

The reader will notice that this paper lays only the foundations of the theory. There are many promising directions to explore, often involving connections to existing literature and related fields of mathematics. We conclude the paper with an outlook on future work in Section \ref{sec:outlook}.

\subsection*{Acknowledgements} We would like to thank Luca Accornero, Francesco Cattafi, Marius Crainic, Ivan Struchiner, and Luca Vitagliano for many discussions that helped us shape the ideas presented here. We especially thank Ori Yudilevich, who several years ago coauthored with one of us some foundational work that led to this paper but whose professional path has since taken him elsewhere.
\newpage

\section*{List of notations}
\begin{longtable}{p{3.2cm}p{10cm}}
\toprule
Notation & Meaning \\
\midrule
$V\to N,\ W\to M$ & Vector bundles.\\
$(\varphi,p):V\to W$ & Vector bundle morphism over $p:M\to N$.\\
$\Omega^\bullet(W)$ & $W$-forms, i.e., sections of $\wedge^\bullet W^*$. \\
$(M, \FF)$ & Foliated manifold.\\
$C^\infty_\bas(U)$ & Sheaf of basic functions of $(M,\FF)$.\\
$\overline{\nabla}$ & $\FF$-connection.\\
$(W, \onabla)$ & Flat foliated vector bundle over $(M, \FF)$.\\
$\Gamma_{(W, \overline{\nabla})}(U)$ & Sheaf of $\overline{\nabla}$-parallel sections.\\
$\Omega^\bullet_{(W,\overline{\nabla})}$ & Sheaf of $\onabla$-parallel $W$-forms.\\
$\mathcal D$ & Derivation or derivation relative to $\varphi$ or to $\FF$.\\
$\sigma(\mathcal D)$ & Symbol of $\D$.\\
$[v_0,\cdots,v_k]$ & $k$-bracket or relative $k$-bracket.\\
$\rho$ & Anchor of $[v_0,\cdots,v_k]$.\\
$\Der^k(V)$ & Space of $k$-derivations of $V\to N$.\\
$\Der^k(\varphi)$ & Space of $k$-derivations relative to $\varphi$.\\
$\DD_V^k$ & Bundle of $k$-derivations on $V$.\\
$\DD^k_\varphi$ & Bundle of $k$-derivations relative to $\varphi$.\\
$\DD^k_{(W, \overline{\nabla})}$ & Bundle of $k$-derivations relative to $\FF$.\\
$\TT,\ \tau$ & Tableau and tableau map of $k$-derivations.\\
$\TT^{(k)},\ \tau^{(k)}$ & $k$-th prolongation of the tableau.\\
$\delta,\ \delta_\tau$ & Spencer differential.\\
$\mathcal H^{p,q}(\tau)$ & Spencer cohomology group.\\
$A\to N$ & Underlying vector bundle of a relative algebroid.\\
$(A,p,\D)$ & Algebroid relative to $p:M\to N$.\\
$(B, \overline{\nabla},\D)$ & Algebroid relative to a foliation $\FF$.\\ 
$(P, r, \theta)$ & Realization $\theta:TP\to B$ of a relative algebroid.\\ 
$T, \Theta$ & Torsion and Intrinsic torsion.\\
$\kappa,\ \Omega$ & Curvature and intrinsic curvature.\\
$p_1:L\to M$ & Space of completions.\\
$\mathcal D^{(k)}$ & $k$-th prolongation of $\D$.\\
$M^{(k)}$ & $k$-th prolongation space.\\
$(B^{(k-1)}, p_k, \D^{(k)})$ & $k$-th prolongation of $(B, \onabla, \D)$.\\
$\UD$ & Universal relative derivation.\\
$c_{\D}$ & Classifying map of $\D$.\\
$\FF_{\text{sys}}$ & Systatic foliation.\\
$q_k\:J^kq\to X$ & $k$-th order jet bundle of $q:N\to X$.\\
$C$ & Relative connection. \\
$\mathcal{C},\ \D_{\mathcal C}$ & Cartan distribution and Cartan derivation.\\
$(q_{k-1}^*TX, p_k, \D_E)$ & Relative algebroid of $k$-th order PDE $E\subseteq J^kq$.\\
\bottomrule
\end{longtable}
\newpage

\section{Relative derivations}\label{sec:relativeDerivations}

In this section, we develop the theory of relative derivations. While after this section, we will only encounter relative derivations of degree 1 and 2, we consider here relative derivations of arbitrary degree. We believe that higher-degree relative derivations, like their ordinary counterparts, will prove useful in other applications.

\subsection{Derivations and brackets}

We recall here some basic facts about $k$-derivations and $k$-brackets. For more details and proofs we refer the reader to \cite{CrainicMoerdijk08}.

Let $V\to N$ be a vector bundle and set $\Omega^\bullet(V) := \Gamma(\wedge^\bullet V^*)$. A \textbf{$k$-derivation} on $V$ is a graded derivation of $\Omega^\bullet (V)$ of degree $k$, i.e., a linear map $\D\: \Omega^\bullet (V) \to \Omega^{\bullet + k} (V)$ satisfying
\[
	\D(\alpha \wedge \beta) = \left(\D\alpha \right) \wedge \beta + (-1)^{|\alpha| k} \alpha \wedge \left( \D\beta\right).
\]
for homogeneous $\alpha, \beta \in \Omega^\bullet(V)$.
The \textbf{symbol} of a $k$-derivation $\D$ is the bundle map $\sigma(\D) \: T^*N \to \wedge^k V^*$ defined by
\[
	\sigma(\D)(\d f) = \D(f), \mbox{ for $f\in C^\infty(N)$}.
\]
Any derivation is determined by its symbol and its action on a generating set (over $C^\infty(N)$) of $\Omega^1(V)$.

A dual point of view to multi-derivations is via multi-brackets. A $k$-bracket on a vector bundle $V\to N$ is a skew-symmetric, $\RR$-multilinear, map
\[
    [\cdot, \dots, \cdot]\: \underbrace{\Gamma(V) \times\cdots \times \Gamma(V)}_{\text{$(k+1)$-times}}\to \Gamma(V)
\]
together with an anchor $\rho\: \wedge^kV\to TN$ satisfying the Leibniz rule:
\[
    [v_0, \cdots, fv_k] = f[v_0, \cdots, v_k] + \rho(v_0, \dots, v_{k-1})(f) v_k, 
\]
for $v_0, \dots, v_k\in \Gamma(V), f\in C^\infty(N)$. 

The notions of $k$-brackets and $k$-derivations are in duality through the Koszul formula (see \cite[\S 2.5]{CrainicMoerdijk08}):
\begin{align*}
    \langle \sigma(\D)(\d f), v_1\wedge\dots\wedge v_k\rangle &=\langle \d f, \rho(v_1, \dots, v_k)\rangle \\
    \D \alpha(v_0, \dots, v_k) &= \sum_{i=0}^k (-1)^i\rho(v_0, \dots, \widehat{v_i}, \dots, v_k) (\alpha(v_i)) - \alpha([v_0,\dots, v_k]),
\end{align*}
for $f \in C^\infty(N)$, $\alpha\in \Omega^1(V)$ and $v_0, \dots, v_k\in \Gamma(V)$.

\begin{proposition}
	The space of $k$-derivations on $V$, denoted $\Der^k(V)$, is the space of sections of a vector bundle $\DD_V^k\to N$, and the symbol induces a short exact sequence
    \[
        \xymatrix{
    	0 \ar[r] & \Hom(V^*, \wedge^{k+1} V^*) \ar[r] & \DD_V^k \ar[r]^---{\sigma} & \Hom(T^*N, \wedge^k V^*) \ar[r] & 0.}
    \]
\end{proposition}

Dually, using the canonical identification $\Hom(W^*, V^*)\cong \Hom(V, W)$, one has 
\begin{equation}\label{eq:symbolExactSequence}
\xymatrix{
	0 \ar[r] & \Hom(\wedge^{k+1} V, V) \ar[r] & \DD_V^k \ar[r]^---{\sigma} & \Hom(\wedge^k V, TN) \ar[r] & 0,}
\end{equation}
which is the sequence that we will use in practice.

A choice of connection $\nabla$ on $V$ determines a splitting $\DD_V^k\to \Hom(\wedge^{k+1} V, V)$ of this sequence. Namely, for any $k$-bracket, the expression
\begin{equation}
    \label{eq:splitting:connection:0}
    [v_0,\dots,v_k]_\nabla:= [v_0,\dots,v_k]+(-1)^{k+1}\sum_{i=0}^k (-1)^i\nabla_{\rho(v_0, \dots, \widehat{v_i}, \dots, v_k)}v_i,
\end{equation} 
is $C^\infty(M)$-multilinear.

\begin{example}
    A derivation $\D\: \Omega^\bullet(V)\to \Omega^\bullet(V)$ of degree 0 corresponds to a linear vector field on $V$. The flow $\varphi_t$ of this linear vector field is a 1-parameter family of vector bundle maps that solves the ODE
    \begin{equation*}
        \begin{cases}
            \frac{\d}{\d t} \varphi_t^* \alpha = \varphi^*_t (\D \alpha),\\
            \varphi_0 = \id_V.
        \end{cases}
    \end{equation*}
\end{example}

\begin{example}
	A Lie algebroid structure on a vector bundle $A\to M$ is 1-derivation $\D\: \Omega^\bullet(A)\to \Omega^{\bullet + 1}(A)$ such that $\D^2 = 0$.
\end{example}

\subsection{Relative derivations} 
Fix two vector bundles $W\to M$ and $V\to N$ together with bundle map $(\varphi, p)\: W \to V$ covering a smooth map $p\: M\to N$.

\begin{definition}
    A \textbf{$k$-derivation relative to $\varphi$} is a map
	\[
		\D\: \Omega^\bullet (V) \to \Omega^{\bullet + k} (W)
	\]
	satisfying the Leibniz rule
	\[
		\D(\alpha\wedge\beta) = (\D\alpha)\wedge (\varphi^*\beta) +(-1)^{|\alpha| k} (\varphi^* \alpha) \wedge (\D\beta).
	\]
	The space of $k$-derivations relative to $\varphi$ is denoted by $\Der^k(\varphi)$.
\end{definition}

\begin{lemma}
	The space $\Der^k(\varphi)$ arises as the space of sections of a vector bundle $\DD^k_\varphi$ over $M$.
\end{lemma}
\begin{proof}
    The space $\Der^k(\varphi)$ is a $C^\infty(M)$-module. In local trivializations and coordinates, it is clear that this module is also locally finitely generated. The lemma follows from the Serre-Swan theorem.
\end{proof}

Note that ordinary derivations are derivations relative to the identity map, so that $\Der^k(V)=\Der^k(\id_V)$ and $\DD^k_V=\DD^k_{\id_V}$.

\begin{definition}
	The \textbf{symbol} of a $k$-derivation relative to $(\varphi,p)$ is the map $\sigma(\D)\: p^*T^* N \to \wedge^k W^*$ defined by
	\[
		\sigma(\D)_x(\d_{p(x)} f) = \D(f)_x \in \wedge^k W^*|_x, \quad \text{ for }f\in C^\infty(N), \ x\in M.
	\]
\end{definition}

The dual of the symbol, arising from the canonical identification $\Hom(W, V)\cong \Hom(V^*, W^*)$, is also denoted by $\sigma(\D)\: \wedge^k W \to p^* TN$. As for ordinary derivations, the symbol induces an exact sequence
\begin{equation}\label{eq:symbolsesrelativederivations}
	\xymatrix{
    0 \ar[r] & \Hom(\wedge^{k+1} W, p^* V) \ar[r] & \DD^k_\varphi \ar[r]^---{\sigma}& \Hom(\wedge^kW, p^*TN) \ar[r] & 0.}
\end{equation}

\begin{example}\label{ex:relativeDerivationsPulledBack}
	If $\D_1\in \Der^k(W)$  and $\D_2\in \Der^k(V)$ are $k$-derivations on $W$ and $V$, respectively, then $\varphi^*\circ \D_2$ and $\D_1\circ \varphi^*$ are both derivations relative to $\varphi$. 
\end{example}

\begin{example}\label{ex:PointwiseDerivations}
    Let $V\to N$ be a vector bundle and fix $x\in N$. Any element $\D\in (\DD^k_V)_x$ determines a $k$-derivation \textbf{relative} to the inclusion $\iota_x\: V_x \hookrightarrow V$:
    \[
        \D\: \Omega^\bullet (V)\to \wedge^{\bullet + k} V_x^*
    \]
    It can be defined as follows. Pick any section $\tilde{\D}\in \Gamma(\DD^k_V)$ with $\tilde{\D}_x = \D$ and set
    \[
        \D \alpha = \left(\iota_x^*\circ\tilde{\D}\right)\alpha.
    \]
    This gives a canonical identification between the space of $k$-derivations relative to the inclusion $\iota_x\: V_x \hookrightarrow V$ and the fiber $(\DD^k_V)_x$. 
\end{example}

The most important case relevant to Bryant's equations occurs when $V$ is a vector bundle over $N$, $W=p^*V$ is the pullback of $V$ along a map $p\:M\to N$ and $\varphi=p_*\:p^*V\to V$ is the canonical projection -- see also Remark \ref{rem:no:brackets} below. In this case, the short exact sequence \eqref{eq:symbolsesrelativederivations} becomes
\begin{equation}\label{eq:symbolsesrelativederivations:projection}
	\xymatrix{
    0 \ar[r] & p^*\Hom(\wedge^{k+1} V, V) \ar[r] & \DD^k_\varphi \ar[r]^---{\sigma} & p^*\Hom(\wedge^k V,TN) \ar[r] & 0.}
\end{equation}

The previous sequence suggests the following result.
\begin{lemma}
    Let $p\: M\to N$ be a map, $V\to N$ a vector bundle and $p_*\:p^*V\to V$ the projection. Then there is a canonical isomorphism
    \[ 
        \DD^k_{p_*}\cong p^*\DD^k_V.
    \]
\end{lemma}

\begin{proof}
    This follows from the observation that 
    \[
        \Der^k(p_*)\cong C^\infty(M)\otimes_{C^\infty(N)} \Der^k(V)
    \]
    as a $C^\infty(M)$-module.
\end{proof}

\subsection{Relative brackets} 
Relative derivations are in duality with relative brackets, a concept that we will introduce here.

\begin{definition}
    Let $V\to N$ be a vector bundle and $p\: M\to N$ any smooth map. A \textbf{$k$-bracket relative to $p$} is a skew-symmetric $\RR$-multilinear map
    \[
        [\cdot, \dots, \cdot]\: \underbrace{\Gamma(V) \times\cdots \times \Gamma(V)}_{\text{$(k+1)$-times}}\to \Gamma(p^*V)
    \]
    and a relative anchor $\rho\: \wedge^k p^*V\to p^*TN$ satisfying the Leibniz rule
    \[
        [v_0, \dots, fv_k] = (p^*f)[v_0, \dots, v_k] + \rho(p^*v_0,\dots,p^*v_{k-1})(f) p^*v_k
    \]
    for all $f\in C^\infty(N)$.
\end{definition}

\begin{remark}\label{rem:no:brackets}
    There seems to be no canonical way to define a bracket relative to an arbitrary bundle map $(\varphi, p)\: W\to V$ because there is no canonical map $\Gamma(V)\to \Gamma(W)$, unless $\varphi$ is a fiberwise isomorphism, in which case $W$ is isomorphic to $p^*V$. 
\end{remark}

Similar to the case of ordinary multiderivations, we have the following correspondence between relative multiderivations and relative multibrackets.

\begin{lemma}\label{lem:derivationsAndBrackets}
	Let $p\: M\to N$ be a map and $V\to N$ a vector bundle. There is a 1:1 correspondence between
	\begin{enumerate}[(i)]
		\item $k$-derivations relative to $p_*\: p^*V\to V$ and
		\item $k$-brackets relative to $p$.
	\end{enumerate}
\end{lemma}
\begin{proof}
    This follows from a modified version of the Koszul formula:
    \begin{align*}
        (\D f)(p^*v_1, \dots, p^*v_k) &= \rho(p^*v_1, \dots, p^*v_k)(f),\\
        (\D \alpha)(p^*v_0, \dots, p^*v_k) &= \sum_{i=0}^k(-1)^{i + k} \rho(p^*v_0, \dots, \widehat{p^*v_i}, \dots, p^*v_k)\left(\langle \alpha, v_i\rangle\right)\\
        &\phantom{=} -p^*\alpha([v_0, \dots, v_k])
    \end{align*}
    for $f\in C^\infty(N)$, $\alpha\in \Omega^1(V)$ and $v_0, \dots, v_k\in\Gamma(V)$.
\end{proof}

\begin{remark}
\label{rem:connection:splits}
    Thinking of multiderivations relative to $p$ as multibrackets relative to $p$, one sees that each choice of connection $\nabla$ for $V$ gives a splitting of the sequence \eqref{eq:symbolsesrelativederivations:projection} associated with the anchor. Namely, we have the analogue of formula \eqref{eq:splitting:connection:0}: given a relative $k$-bracket the expression
    \[
    [p^*v_0,\dots,p^*v_k]_\nabla:= [v_0,\dots,v_k]+(-1)^{k+1}p^*\Big(\sum_{i=0}^k (-1)^i\nabla_{\rho(p^*v_0, \dots, \widehat{p^*v_i}, \dots, p^*v_k)}v_i\Big),
    \]
    defines a unique element of $\Hom(\wedge^{k+1} p^*V,p^*V)=p^*\Hom(\wedge^{k+1} V, V)$.
\end{remark}

\subsection{Brackets and derivations relative to foliations}

In practice, relative derivations often appear relative to a submersion $p\: M\to N$. In some cases, however, we encounter derivations that are only \emph{locally} relative to a submersion, i.e., they are relative to a foliation $\FF$ of $M$. In order to define them properly, we need to recall first some basic notions from foliation theory (see, e.g., \cite{MoerdijkMrcun03}).

\subsubsection{Foliated flat vector bundles}
Henceforth, we will identify a foliation $\FF$ on a manifold $M$ with an involutive subbundle $\FF\subset TM$. Associated to $\FF$ one has the sheaf of \textbf{basic functions} given by
\[
    C^\infty_\bas(U) = \{ f\in C^\infty(U) : X(f) = 0 \text{ for all $X\in \Gamma(\FF\vert_U)$} \}.
\]
Given a vector bundle $W\to M$, we will write $\Omega^k(\FF; W)$ for the foliated $k$-forms with values in $W$, i.e., sections of $\wedge^k\FF^*\otimes W$. Recall that a \textbf{$\FF$-connection} on $W$ is a $\RR$-bilinear map $\overline{\nabla}\: \Gamma(\FF)\times \Gamma(W)\to \Gamma(W)$ satisfying:
\[
\overline{\nabla}_{fX}w = f \overline{\nabla}_{X},\quad 
\overline{\nabla}_X(fw) = f\overline{\nabla}_xw + X(f) w,
\] 
for $X\in \Gamma(\FF)$, $f\in C^\infty(M)$ and $w\in \Gamma(W)$. It can be interpreted as a differential operator $\overline{\nabla}\: \Gamma(W)\to \Omega^1(\FF; W)$. A section $w\in \Gamma(W)$ is \textbf{$\overline{\nabla}$-parallel} when $\overline{\nabla} w =0$. This gives rise to a sheaf of modules over $C^\infty_{\bas}$ given by
\[
    \Gamma_{(W, \overline{\nabla})}(U) = \{ w\in \Gamma_W(U) : \overline{\nabla} w  = 0 \}.
\]
Local existence of parallel sections is controlled by the curvature of $\onabla$ which is the 2-form
$R_{\overline{\nabla}}\in\Omega^2(\FF; \End(W))$ given by
\[
    R_{\overline{\nabla}}(X, Y) = \left[\overline{\nabla}_X, \overline{\nabla}_Y\right] -\overline{\nabla}_{[X, Y]}.
\]  
For a flat $\FF$-connection (i.e. $R_{\overline{\nabla}} = 0$), there exist a parallel local section through every point $w\in W$. However, non-zero global parallel sections may not exist.

\begin{definition}
    A \textbf{foliated vector bundle} $(W, \overline{\nabla})$ is a vector bundle $W$ over a manifold with foliation $(M, \FF)$ together with an $\FF$-connection $\overline{\nabla}$. The foliated vector bundle is called \textbf{flat} when the $\FF$-connection is flat.
\end{definition}
The dual $W^*$ of a foliated vector vector bundle $(W, \overline{\nabla})$ carries a canonical $\FF$-connection $\overline{\nabla}$ determined by
\[
    \LL_X\langle \alpha, w\rangle = \left\langle \overline{\nabla}_X \alpha, w\right\rangle + \left\langle \alpha, \overline{\nabla}_X w\right\rangle,
\]
for $X\in \Gamma_\FF$, $\alpha\in \Gamma_{W^*}$ and $w\in \Gamma_W$. Note that $(W^*, \overline{\nabla})$ is flat iff $(W, \overline{\nabla})$ is flat. 

Next, by a \textbf{map of foliations} $p\: (M_1, \FF_1)\to (M_2, \FF_2)$ we mean a map such that $\d p(\FF_1)\subset\FF_2$. For example, the identity $\id\:M\to M$ maps the trivial foliation to any foliation $\FF_M$. A \textbf{map of foliated vector bundles} $\varphi\: (W_1,\onabla^1) \to (W_2,\onabla^2)$ covering a map of foliations $p\: (M_1, \FF_1)\to (M_2, \FF_2)$ is a bundle map $\varphi\:W_1\to W_2$  which satisfies
\[ 
    \onabla^1_X(\varphi^*\alpha)) = \varphi^*(\onabla^2_{p_*(X)}\alpha),\quad\text{for all }X\in \Gamma_{\FF_1}, \alpha\in \Omega^1(W_2).
\]
In the case of flat foliated bundles, this condition ensures that pullback maps flat forms to flat forms $\varphi^*\:\Omega^\bullet_{(W_2,\onabla^2)}\to\Omega^\bullet_{(W_1,\onabla^1)}$.

\begin{remark}
\label{rem:fiberwise:iso}
    When $\varphi$ is a fiberwise isomorphism, we obtain also a pullback of sections which preserves flatness:
    \[ \varphi^*\:\Gamma_{(W_2,\onabla^2)}\to\Gamma_{(W_1,\onabla^1)},
    \quad (\varphi^*w)(x):=\varphi^{-1}(w(p(x))). \]
\end{remark}

There is a version of the Serre-Swan theorem for flat foliated vector bundles.

\begin{proposition}\label{prop:BasicSerreSwan}
    Let $(M, \FF)$ be a manifold with foliation. There is, up to  a natural isomorphism, a one-to-one correspondence 
    \[
        \left\{ 
            \substack{ 
                \mbox{ 
                    locally finitely generated}\\
                    \\
                \mbox{and locally free $C^\infty_{\mathrm{bas}}$-modules
                } \\
            } 
        \right\}
        \overset{1:1}{\longleftrightarrow} 
        \left\{
            \substack{
                \mbox{flat foliated vector}\\ \\
                \mbox{bundles over $(M, \FF)$}\\
            }
        \right\}
    \]
\end{proposition}

\begin{proof}
    Given $(W, \overline{\nabla})\to M$, we just saw how to construct a locally finitely generated and locally free $C^\infty_{\bas}$-module $\Gamma_{(W, \overline{\nabla})}$. For the converse, given such a module $\Gamma_{\para}$, set 
    \[
        \Gamma_W := \Gamma_{\para}\otimes_{C^\infty_{\bas}} C^\infty
    \]
    which is a locally finitely generated, locally free $C^\infty$-module. By the Serre-Swan theorem, $\Gamma_W$ is the sheaf of sections of a vector bundle $W\to M$. The flat $\FF$-connection is determined by requiring its local flat sections to satisfy
    \[
        \overline{\nabla}_X w = 0, \quad \text{ for all $X\in \Gamma_\FF$, $w\in \Gamma_{\para}$}.
    \]
    It is well-defined because $\Gamma_{\para}$ is a $C^\infty_\bas$-module. These constructions are inverses of each other. Naturality follows from naturality of the original Serre-Swan theorem, observing that the resulting maps preserve the parallel sections, and so are maps of flat foliated vector bundles.
\end{proof} 

Let $(W, \overline{\nabla})$ be flat foliated vector bundle over a foliated manifold $(M, \FF)$. The \textbf{holonomy} representation of a leaf $L$ of $\FF$ is denoted
\[
    h\:\pi_1(L, x)\to \GL(W_x) 
\]
and is obtained, as usual, by parallel transport along loops based at $x$. More generally, parallel transport along paths, gives the groupoid representation
\[ \Pi_1(\FF)\to \GL(W),\]
where $\Pi_1(\FF)\tto M$ is the Lie groupoid whose arrows are the leafwise homotopy classes of paths in $\FF$, and $\GL(W)\tto M$ is the Lie groupoid whose arrows are the linear isomorphisms between the fibers of $W$. Conversely, any such Lie groupoid representation defines a flat $\FF$-connection on $W$.

If $\tilde{\FF}\subset \FF$ is a subfoliation, a flat $\FF$-connection determines by restriction a flat $\tilde{\FF}$-connection and the holonomy representations are related by
\[
\xymatrix{
\Pi_1(\tilde{\FF})\ar[dr]^{\tilde{h}}\ar[d]_{i_*}\\
\Pi_1(\FF)\ar[r]_{h} & \GL(W)
}
\]

\begin{proposition}\label{prop:BundlesWithFlatPartialConnectionsOfSimpleFoliations}
    Let $(M,\FF)$ be a foliated manifold and $p\: M\to N$ a surjective submersion with connected fibers such that $\ker (\d p) \subset \FF$. Then there is a foliation $\FF_N$ on $N$ such that $\FF =(\d p)^{-1}(\FF_N)$ and $p\: (M, \FF)\to (N, \FF_N)$ is a map of foliated manifolds, for which there is, up to a natural isomorphism, a one-to-one correspondence between 
    \[
            \left\{
            \substack{
                \mbox{vector bundles over $N$ with} \\ \\
                \mbox{flat $\FF_N$-connection}
            } 
            \right\}
            \xleftrightarrow{1:1}
            \left\{
            \substack{ 
                \mbox{vector bundles over $M$ with}\\ \\ \mbox{flat $\FF$-connection having}\\ \\
                \mbox{trivial holonomy along $\ker(\d p)$}
            }      
            \right\}
        \]
\end{proposition}

\begin{proof}
    Since $p$ is a surjective submersion and $\ker(\d p)\subset \FF$,  it follows that 
    \[ \FF_N := \d p(\FF)\] 
    is a well-defined involutive subbundle of $TN$. Since the fibers of $p$ are connected and contained in $\FF$, one has that $\FF =(\d p)^{-1}(\FF_N)$.

    One direction of the correspondence is clear: since $p$ is a map of foliations, if $(W_N,\overline{\nabla}_N)\to (N, \FF_N)$ is a flat foliated vector bundle, then $(p^*W_N, p^*\overline{\nabla}_N)\to (M, \FF)$ is a vector bundle with flat $\FF$-connection with trivial holonomy along $\ker(\d p)$.

    For the converse, suppose we are given a vector bundle with a flat $\FF$-connection $(W, \overline{\nabla})\to M$ having trivial holonomy along $\FF_p:=\ker(\d p)$. Then the holonomy representation of $\FF_p$ factors via the submersion groupoid 
    \[
    \xymatrix{
    \Pi_1(\FF_p)\ar[dr]\ar[d]\\
    M\times_p M\ar[r] & \GL(W)}
    \]
    The resulting linear action of $M\times_p M\tto M$ on $W\to M$ is free and proper. The quotient is a vector bundle $W_N\to N$ and there is a canonical isomorphism
    \[ p^*W_N\cong W. \]
    Moreover, the submersion $p\:M\to N$ induces a surjective groupoid morphism
    \[ p_*\:\Pi_1(\FF)\to \Pi_1(\FF_N) \]
    whose kernel contains $\Pi_1(\FF_p)$. It follows that the holonomy representation descends to a representation of $\Pi_1(\FF_N)$ making the following diagram commute. 
    \[
    \xymatrix{
    \Pi_1(\FF)\ar[r]^{h}\ar[d]_{p_*} & \GL(W)\ar[d]\\
    \Pi_1(\FF_N)\ar[r]^{h_N} & \GL(W_N)}
    \]
    Hence, there is a unique flat $\FF_N$-connection $\onabla_N$ on $W_N$ whose holonomy representation is $h_N$. For such a connection one has
    \[ p^*\onabla_N=\onabla. \]
    The previous two constructions are inverse to each other so the result follows.
\end{proof}

\begin{example}
    \label{ex:relating:things}
    Let $\FF$ be any foliation on $M$ and denote by $\nu(\FF) = TM/\FF$ the normal bundle of $\FF$. This carries a flat $\FF$-connection, namely the \textbf{Bott connection}
    \[
        \overline{\nabla}^{\text{Bott}}_X \overline{Y} := \overline{[X, Y]}, \quad \mbox{ for $X\in \Gamma(\FF)$, $\overline{Y}\in\nu(\FF)$}.
    \]
    When $\FF$ is a simple foliation, so that $N=M/\FF$ is a manifold, the resulting foliation on $N$ is the trivial one: $\FF_N=0_N$. The corresponding vector bundle over $N$ is the tangent bundle $TN$ (a flat foliated vector bundle for the trivial foliation).
    
    More generally, any vector bundle $V\to N$ is a flat foliated vector bundle for the trivial foliation, so the pullback $W:=p^*V$ carries a canonical flat $\FF$-connection $\onabla$. It is the connection whose local flat sections are the sections of the form $p^*v$, with $v$ any local section of $V$.
\end{example}

\subsubsection{Derivations relative to a foliation}

Throughout the rest of this section, we fix a flat foliated vector bundle $(W, \onabla)$ over $(M, \FF)$. The sheaf of sections of $\wedge^\bullet W^* \to M$ (the ``$W$-forms'') will be denoted by $\Omega^\bullet_W$, whereas the sheaf $\Omega^\bullet_{(W,\overline{\nabla})}$ of $\onabla$-parallel $W$-forms is given by  
\[
    \Omega^\bullet_{(W,\overline{\nabla})} (U) := \{ \alpha \in \Omega^\bullet_W(U) : \overline{\nabla} \alpha = 0 \}.
\]

\begin{definition}
    A \textbf{parallel $k$-derivation} on a flat foliated vector bundle $(W, \overline{\nabla})$ is a map of sheaves
    \[
        \D\: \Omega^\bullet_{(W, \overline{\nabla})} \to \Omega^{\bullet + k}_{(W, \overline{\nabla})}
    \]
    satisfying, for any homogeneous $\alpha, \beta\in \Omega^\bullet_{(W, \overline{\nabla})}$, the Leibniz rule
    \[
        \D(\alpha\wedge\beta) = (\D\alpha)\wedge \beta + (-1)^{|\alpha|k} \alpha \wedge(\D\beta).
    \]
    The sheaf of parallel $k$-derivations is denoted by $\Der^k_{\para}(W, \overline{\nabla})$.
\end{definition}

\begin{lemma}\label{lem:ParallelConnectionsVectorBundleWithConnection}
    $\Der^k_{\para}(W, \overline{\nabla})$ is the sheaf of flat sections of a flat foliated vector bundle, denoted $(\DD^k_{(W, \overline{\nabla})}, \overline{\nabla})$.
\end{lemma}

\begin{proof}
    $\Der^k_{\para}(W, \overline{\nabla})$ is a locally finitely generated, locally free $C^\infty_{\bas}$-module. Thus, by Proposition \ref{prop:BasicSerreSwan}, it is the space of parallel sections of a flat foliated vector bundle $(\DD^k_{(W, \overline{\nabla})}, \overline{\nabla})$.
\end{proof}

We are now ready to introduce the following generalization of derivations relative to submersions to derivations relative to foliations.

\begin{definition}
    A \textbf{$k$-derivation on $(W, \overline{\nabla})$ relative to $\FF$} is a map of sheaves
    \[ \D\: \Omega^\bullet_{(W, \overline{\nabla})} \to \Omega^{\bullet + k}_W,\]
    satisfying for homogeneous $\alpha, \beta\in \Omega^\bullet_{(W, \overline{\nabla})}$ the Leibniz rule
    \[
        \D(\alpha\wedge\beta) = (\D\alpha)\wedge\beta + (-1)^{|\alpha|k} \alpha\wedge(\D\beta).
    \]
    Its \textbf{symbol} is the map $\sigma(\D)\in \Hom(\nu^*(\FF), \wedge^k W^*)$ given by
    \[
        \sigma(\D)(\d f) = \D(f), \quad \text{for $f\in C^\infty_{\bas}$.}
    \]
\end{definition}    

In practice, we will think of the symbol as a map $\sigma(\D)\:\wedge^k W \to \nu(\FF)$ via the usual canonical identification.

\begin{remark}
    The relative $k$-derivations just defined are canonically identified with the global sections of the vector bundle $\DD^k_{(W, \overline{\nabla})}$, while parallel $k$-derivations correspond to parallel sections of $\DD^k_{(W, \onabla)}$. The flat $\FF$-connection on sections of $\DD^k_{(W, \onabla)}$ is given by
    \[
        (\onabla_X \D)\alpha = \onabla_X(\D \alpha), \quad \mbox{ for $\alpha\in \Omega^\bullet_{(W, \onabla)}$}.
    \]
    Unlike global parallel $k$-derivations, which may or may not exist, relative $k$-derivations always exist in abundance.
\end{remark}

The duality between brackets and derivations relative to submersions extends to foliations.

\begin{definition}
    A \textbf{$k$-bracket on $(W, \overline{\nabla})$ relative to $\FF$} is a skew-symmetric $\RR$-multilinear map of sheaves
    \[
        [\cdot, \dots, \cdot]\: \underbrace{\Gamma_{(W, \overline{\nabla})} \times\cdots \times \Gamma_{(W, \overline{\nabla})}}_{\text{$(k+1)$-times}}\to \Gamma_W
    \]
    with a relative anchor $\rho\: \wedge^k W\to \nu(\FF)$ satisfying the Leibniz rule
    \[
        [w_0, \dots, fw_k] = f[w_0, \dots,w_k] + \rho(w_0,\dots,w_{k-1})(f) w_k
    \]
    for any local flat sections $w_i$ and basic function
    $f\in C^\infty_\bas$.
\end{definition}

In a manner entirely similar to Lemma \ref{lem:derivationsAndBrackets}, we find:

\begin{lemma}
	Let $(W, \onabla)$ be a flat foliated vector bundle over $(M, \FF)$. There is a 1:1 correspondence between
	\begin{enumerate}[(i)]
		\item $k$-derivations relative to $\FF$ and
		\item $k$-brackets relative to $\FF$.
	\end{enumerate}
\end{lemma}

Notice that in the case of a submersion $p\:M\to N$ and a vector bundle $V\to N$ -- see Example \ref{ex:relating:things} -- these definitions and results specialize to the previous notions of $k$-derivation and $k$-bracket relative to $p\:M\to N$.

\subsubsection{The structure of relative derivations}\label{sec:structureOfRelativeDerivations}
We start by giving the analogue of the short exact sequence \eqref{eq:symbolsesrelativederivations:projection}. The proof is immediate.

\begin{lemma}\label{lem:sessymbolconnection}
    Given a foliated flat bundle $(W, \onabla)\to (M, \FF)$, the symbol map induces a short exact sequence 
    \begin{equation}
    \label{eq:symbolsesrelativederivations:foliation}
        \xymatrix{0 \ar[r]& 
    \Hom(\wedge^{k+1}W, W) \ar[r]&
    \DD^k_{(W, \overline{\nabla})} \ar[r]^---{\sigma} & \Hom(\wedge^kW, \nu(\FF))\ar[r]& 0,}
    \end{equation}
    If we equip these bundles with the induced connections $\onabla$ and $\overline{\nabla}\otimes \overline{\nabla}^{\mathrm{Bott}}$ (for the last term), this is a sequence of flat foliated vector bundles.
\end{lemma}

An \textbf{extension} of the $\FF$-connection $\overline{\nabla}$ to an ordinary connection $\nabla$ is a connection on $W$ such that 
\[ \nabla_X w = \overline{\nabla}_X w,\quad \text{for all }X\in \Gamma(\FF),\ w\in \Gamma(W).\] It induces a well-defined map
\[
    \nabla\: \Gamma_{\nu(\FF)}\times \Gamma_{(W, \overline{\nabla})} \to \Gamma_W, \quad \nabla_{[X]}w:= \nabla_X w.
\]
The extension $\nabla$ is called \textbf{$\FF$-parallel} when this map takes values in $\Gamma_{(W, \overline{\nabla})}$. While extensions always exist, parallel extensions are only guaranteed to exist locally.

Extensions of $\overline{\nabla}$ yield splittings of the previous short exact sequence.

\begin{lemma}\label{lem:sessymbolconnectionsplitting}
    Let $(W, \onabla) \to (M, \FF)$ be flat foliated bundle. An extension $\nabla$ of $\overline{\nabla}$ induces a splitting of the short exact sequence \eqref{eq:symbolsesrelativederivations:foliation} so that
    \[
         \DD^k_{(W, \overline{\nabla})} \cong \Hom(\wedge^{k+1}W, W) \oplus \Hom(\wedge^kW, \nu(\FF)).
    \]
    If the extension is $\FF$-parallel, then this splitting is an isomorphism of flat foliated vector bundles.
\end{lemma}

\begin{proof}
    For any extension $\nabla$ we can define a splitting $\DD^k_{(W, \overline{\nabla})}\to \Hom(\wedge^{k+1}W, W)$ of \eqref{eq:symbolsesrelativederivations:foliation} similar to Remark \ref{rem:connection:splits}: given a relative bracket in $\DD^k_{(W, \overline{\nabla})}$ the expression
    \[
        [w_0,\dots, w_k]_{\nabla} := [w_0,\dots,w_k] +(-1)^{k+1} \sum_{i=0}^k (-1)^i \nabla_{\rho(w_0,\dots, \widehat{w_i},\dots, w_k)} w_i,
    \]
    which is defined for local flat sections $w_i\in\Gamma_{(W,\onabla)}$, extends to a unique $C^\infty$-linear map, determining a bundle map $\wedge^{k+1}W\to W$. When $\nabla$ is an $\FF$-parallel extension, this bundle map sends $\FF$-flat sections to $\FF$-flat sections, so it is map of foliated flat bundles.
\end{proof}

Any $k$-derivation $\D\in \Gamma(\DD^k_{W})$ can be restricted to $\Omega^\bullet_{(W,\overline{\nabla})}$ to yield a derivation $\Pi(\D)$ relative to $\FF$. This restriction map is $C^\infty$-linear, so it is induced from a bundle map $\Pi\: \DD^k_{W}\to \DD^k_{(W, \overline{\nabla})}$.
    
\begin{lemma}\label{lem:structureofRelativeDerivations}
    There is a short exact sequence
    \[
        \xymatrix{0 \ar[r] & \Hom(\wedge^kW, \FF) \ar[r]^---{\iota} & \DD^k_{W} \ar[r] ^{\Pi} &\DD^k_{(W, \overline{\nabla})} \ar[r] & 0}.
    \]  
    where the inclusion $\iota$ is defined at the level of sections by
    \[
        \iota(\rho)(\alpha)(w_1, \dots, w_{k+l}) = \frac{1}{k!l!} \sum_{\sigma \in S_{k+l}} (-1)^\sigma \left(\overline{\nabla}_{\rho(w_{\sigma(1)}, \dots, w_{\sigma(k)})}\alpha\right)(w_{\sigma(k+1)}, \dots, w_{\sigma(k+l)}),
    \]
    for $\alpha\in \Omega^l(W)$ and $w_1, \dots, w_{k+l}\in \Gamma(W)$. 
\end{lemma}

\begin{remark}
The map $\iota\:\Hom(\wedge^kW, \FF) \to \DD^k_{W}$ in the previous statement is the unique linear map satisfying
\[
\begin{cases}
    \iota(\rho)(f) = \rho(f),\quad &\text{ if }f\in C^\infty(M),\\
    \iota(\rho)(\alpha) = 0,\quad &\text{ if }\alpha\in \Omega^\bullet_{(W, \overline{\nabla})}.
\end{cases}
\]
\end{remark}

\begin{proof}
    Restriction induces a surjective map of the short exact sequences of the symbols, resulting in the following commutative diagram:
    \[
    \xymatrix{
     & & 0\ar[d] & 0\ar[d] \\
        & 0 \ar[d] \ar[r] & K \ar[d]^{\iota} \ar[r] & \Hom(\wedge^k W,\FF) \ar[d]\ar[r] & 0\\
    0 \ar[r] & {\Hom(\wedge^{k+1}W,W)} \ar[r] \ar[d] & \DD^k_{W} \ar[r]^---{\sigma} \ar[d]^{\Pi} & \Hom(\wedge^k W,TM) \ar[r] \ar[d] & 0 \\
    0 \ar[r] & \Hom(\wedge^{k+1}W,W) \ar[r] \ar[d] & \DD^k_{(W, \overline{\nabla})} \ar[r]^---{\sigma} \ar[d]   & \Hom(\wedge^kW,\nu(\FF)) \ar[r] \ar[d] & 0\\
    & 0 & 0 & 0  
    }
    \]
    It follows that $K=\Hom(\wedge^k W,\FF)$ and that $\iota$ is as in the statement.
\end{proof}

\subsection{Morphisms and extensions of relative derivations}
\label{sec:morphisms:extensions}

In this section, we discuss morphisms of relative derivations, extensions and what it means to ``compose" two relative degree 1 derivations.

\begin{definition}
\label{def:Phi-related:derivations}
Let $(\varphi,p)\: (W,\onabla^M) \to (V,\onabla^N)$ be a map of flat foliated vector bundles covering $p\: (M,\FF_M)\to (N, \FF_N)$. We say that a derivation $\D_W\in \Gamma(\DD^k_{(W,\onabla^M)})$ is \textbf{$(\varphi,p)$-related} to a derivation $\D_V\in \Gamma(\DD^k_{(V,\onabla^N)})$ if
\begin{equation}
    \label{eq:morphism:derivations}
    \D_W \circ \varphi^* = \varphi^* \circ \D_V.
\end{equation}
\end{definition}  
Note that \eqref{eq:morphism:derivations} implies that the anchors of $\D_W$ and $\D_V$ are related by
\[ \rho_V\circ \wedge^k\varphi=\d p\circ\rho_W. \]
The relation between the associated $k$-brackets is more complicated to express since, in general, there is no map relating sections of $W$ and sections of $V$. However, when $\varphi$ is a fiberwise isomorphism, as in Remark \ref{rem:fiberwise:iso}, one can express relation \eqref{eq:morphism:derivations} in terms of $k$-brackets as
\[ [\varphi^*v_0,\dots,\varphi^*v_k]_W=\varphi^*[v_0,\dots,v_k]_V,\quad \text{ for all }v_i\in\Gamma_{(V,\onabla^V)}. \]
Actually, when $\varphi$ is a fiberwise isomorphism, it induces a bundle map 
\[
\varphi_*\: \DD^k_{(W, \onabla^W)}\to \DD^k_{(V, \onabla^V)}
\]
as follows. Recall that an element $\D_x$ in the fiber $(\DD^k_{(W, \onabla^W)})_x$ can be regarded as a derivation $\D_x \: \Omega^\bullet_{(W, \onabla^W)}\to \big(\wedge^{\bullet + k} W^*\big)_x$ relative to the inclusion (as a map of foliated vector bundles). We define $\varphi_*(\D_x)$ by
\[
    \varphi_*(\D_x) := \varphi_* \circ \D_x \circ \varphi^*\: \Omega^\bullet_{(V, \onabla^V)} \to \left(\wedge^{\bullet + k} V^*\right)_{p(x)},
\]
where $\varphi_*= (\varphi^*)^{-1}\: \big(\wedge^{\bullet}W^*\big)_x \to \big(\wedge^{\bullet}V^*\big)_{p(x)}$. 
Then, a derivation $\D_W$ is $(\varphi,p)$-related to a derivation $\D_V$ if and only if the outside square in the following diagram commutes.
\[
\xymatrix{
\DD^k_{(W, \onabla^W)}\ar[r]^{\varphi_*}\ar[d] & \DD^k_{(V, \onabla^V)}\ar[d]\\
M\ar[r]_{p} \ar@/^1pc/[u]^{\D_W} & N \ar@/_1pc/[u]_{\D_V}
}
\]

\begin{example}[Extensions of relative derivations]
\label{ex:extension:derivations}
If $\D_W$ is $(\varphi,p)$-related to $\D$ and the base map $p$ is a submersion, 
we say that $\D_W$ is an \textbf{extension} of $\D_V$. In particular, we have the following:
\begin{enumerate}
    \item An extension of a derivation $\D\in\Gamma(\DD^k_{(V,\onabla)})$ to a derivation $\D_1\in\Gamma(\DD^k_V)$ is an ordinary derivation $\D_1$ that agrees with $\D$ on the flat forms:
    \[ \D_1\alpha=\D\alpha,\quad \text{ for all }\alpha\in\Omega^\bullet_{(V,\onabla)}. \] 
    This fits into the setting of Definition \ref{def:Phi-related:derivations} by viewing the identity map $\id\:(N,0_N)\to (N,\FF)$ as a map of foliated manifolds.
    \item More generally, given a submersion $p\:M\to N$, 
    an extension of a derivation $\D\in\Gamma(\DD^k_{(V,\onabla)})$ to a derivation $\D_1\in\Gamma(\DD^k_{p})$, is a derivation $\D_1$ on $V$ relative to $p\:M\to N$ 
    satisfying:
    \[ \D_1\alpha=p^*(\D\alpha),\quad \text{ for all }\alpha\in\Omega^\bullet_{(V,\onabla)}. \] 
\end{enumerate}
Extensions will be important in the theory of relative algebroids, developed in Section \ref{sec:RelativeAlgebroids}.
\end{example}

Given two ordinary derivations, their composition usually fails to be a derivation. An important exception occurs in the case of a 1-derivation $\D\in\Gamma(\DD^1_V)$: the square $\D^2:=\D\circ \D$ is a 2-derivation. It is easy to see that if $\D$ has symbol $\rho$ and associated 1-bracket $[\cdot,\cdot]$ then $\D^2$ has symbol
\[ \rho_{\D^2}(v_1,v_2)(f)=\rho(v_1)\big(\rho(v_2)(f)\big)-\rho(v_2)\big(\rho(v_1)(f)\big)-\rho([v_1,v_2])(f), \]
while the associated 2-bracket is the Jacobiator
\[ [v_1,v_2,v_3]_{\D^2}=[v_1,[v_2,v_3]]+[v_2,[v_3,v_1]]+[v_3,[v_1,v_2]]. \]

This generalizes to relative 1-derivations which are extensions, as in Example \ref{ex:extension:derivations}. For the statement, note that under the assumptions of that example, the foliation $\FF$ pulls back to a foliation $p^*\FF$ on $M$, and the bundle $p^*V$ inherits a flat $p^*\FF$-connection $p^*\onabla$ from $(V, \overline{\nabla})$. Pullback gives a canonical isomorphism between the space of flat forms
\[ \Omega^\bullet_{(p^*V, \overline{\nabla})} \cong p^*\Omega^\bullet_{(V, \overline{\nabla})}.
\]

\begin{lemma}
    \label{lem:square}
    If a derivation $\D_1 \in\Gamma(\DD^1_{(p^*V,\onabla)})$ extends a derivation $\D_0\in\Gamma(\DD^1_{(V,\onabla)})$, their composition 
    \[ \D_1\circ \D_0\:\Omega^\bullet_{(p^*V,\onabla)}\to \Omega^{\bullet+2}(p^*V),\]
    is a relative 2-derivation with symbol and 2-bracket given by
    \begin{equation}\label{eq:1-derivation:square}
    \begin{aligned}
       \rho_{\D_1\circ \D_0}(p^*v_1,p^*v_2)(f)&=\rho_1(p^*v_1)\big(\rho_0(v_2)(f)\big)-\rho_1(p^*v_2)\big(\rho_0(v_1)(f)\big)\\
       &\phantom{=}-\rho_1(p^*[v_1,v_2]_0)(f),\\
       [p^*v_1,p^*v_2,p^*v_3]_{\D_1\circ \D_0}&=[v_1,[v_2,v_3]_0]_1+[v_2,[v_3,v_1]_0]_1+[v_3,[v_1,v_2]_0]_1, 
    \end{aligned}
    \end{equation}
    for any $v_1,v_2,v_3\in\Gamma_{(V,\onabla)}$ and $f\in C^\infty_\bas$. 
\end{lemma}

\begin{proof}
    The fact that $\D_1\circ \D_0$ is a $2$-derivation follows from 
    \begin{align*}
    \D_1\circ \D_0(\al\wedge\be)&=\D_1(\D_0\al\wedge\be+(-1)^{|\al|}\al\wedge \D_0\be)\\
        &=\D_1(\D_0\al)\wedge p^*\be +(-1)^{|\al|+1}p^*\D_0\al\wedge\D_1\be+\\
        &\phantom{=}+(-1)^{|\al|}\D_1\al\wedge p^*\D_0\be+(-1)^{2|\al|}p^*\al\wedge \D_1(\D_0\be)\\
        &=(\D_1\circ\D_0)(\al)\wedge p^*\be +p^*\al\wedge (\D_1\circ\D_0)(\be),
    \end{align*}
    where we used the extension property $\D_1\alpha=p^*(\D_0\alpha)$ to cancel two terms. The expressions for the symbol and 2-bracket follow from straightforward computations using the formulas for the duality (see the proof of Lemma \ref{lem:derivationsAndBrackets}).
\end{proof}

\subsection{Tableaux of derivations}\label{sec:TableauxOfDerivations}

Tableaux are extremely useful gadgets in the theory of PDEs \cite{BryantChern1991, Bryant2014, Salazar2013}. They codify higher order consequences of a set of partial differential equations and, under the condition of involutivity, provide a measure of the size of the space of local solutions. 

Let $W, V$ be finite-dimensional vector spaces. Classically, a \textbf{tableau} is a subspace $\TT \subset \Hom(W, V)$ or, more generally, a linear map $\tau\: \TT\to \Hom(W, V)$. In practice, a PDE does not come with a single tableau but with a family of tableaux depending on coordinates, i.e., a vector bundle. The theory of tableaux is pointwise in nature, and thus naturally carries over to vector bundles (possibly under some extra constant rank assumption). Associated to a classical tableau, one has the notions of prolongation, the Spencer complex, involutivity and of Cartan characters. We refer to \cite{Bryant2014, BryantChern1991, Salazar2013} for more details.

It turns out that tableaux associated to classification problems as in the Introduction (or the applications of Theorem 4 in Bryant \cite{Bryant2014}) are not tableaux in the classical sense. Bryant was certainly aware of this, as is evident from his computations in \cite{Bryant2014} of the Cartan characters and prolongations. However, to the authors' knowledge, this is not formalized or mentioned anywhere in the literature.
In this section, we will extend the classical notion of tableau by formalizing the notion of a \emph{tableau of derivations}.

\subsubsection{Definition of tableau of derivations} In what follows, we fix vector bundles $W\to M$ and $V\to N$ and a vector bundle map $(\varphi,p)\: W\to V$. 

\begin{definition}
	A \textbf{tableau of $k$-derivations} relative to a vector bundle map $\varphi$ is a vector subbundle $\TT\subset \DD^k_\varphi$.
\end{definition}

\begin{remark}[Tableaux for derivations relative to foliations]\label{rk:TableauxOfDerivationsRelativeToFoliations}
     All the definitions and results that follow apply equally well to bundles $\DD^k_{(W,\overline{\nabla})}$ of derivations relative to foliations. This is due to the pointwise nature of the operations and notions related to tableaux and the fact that $\DD^k_{(W, \overline{\nabla})}$ is locally isomorphic to a space of derivations relative to a bundle map.
\end{remark}

The theory of prolongations and involutivity of tableaux of derivations, rests upon the following definition.

\begin{definition}\label{def:SpencerDifferentialModifiedTableau}
	The \textbf{Spencer differential} $\delta\: \Hom(\wedge^l W, \DD^k_\varphi)\to \DD^{k+l}_\varphi$ is the unique $C^\infty(M)$-linear map such that
	\[
		\delta(\omega\otimes \D) = \omega\wedge \D, 
	\]
	for any $\omega\in \Omega^l(W)$ and $\D\in \Gamma(\DD^k_\varphi)$.
\end{definition}

\begin{remark}\label{rk:TableauxMapSpencerDifferential}
    Sometimes, instead of a subbundle $\TT\subset \DD^k_\varphi$, we need to consider a vector bundle map $\tau\: \TT\to \DD^k_\varphi$ that is not necessarily injective. For such a bundle map, the Spencer differential $\delta_\tau\:\Hom(\wedge^lW, \TT)\to \DD^{k+l}_\varphi$ is defined by requiring
    \[
        \delta_\tau(\omega\otimes \D) = \omega \wedge \tau(\D),
    \]
    for $\omega \otimes \D \in \Gamma(\wedge^l W^* \otimes \TT)$. For a classical tableau, such objects were considered in \cite{Salazar2013} under the name ``generalized tableaux". In this paper, we will refer to the bundle map $\tau$ as a \textbf{tableau map}.
\end{remark}

An element in $\Gamma\left(\Hom(\wedge^l W, \DD^k_\varphi)\right)$ can be viewed as a $W$-form of degree $l$ with values in $\DD^k_\varphi$, so it is determined by its action on $\Omega^1(V)$. For $\xi\in \Gamma\left(\Hom(\wedge^l W, \DD^k_\varphi)\right)$, the Spencer differential $\delta\xi$ is the $(k+l)$-derivation relative to $\varphi$ acting on $\alpha\in \Omega^1(V)$ as
\begin{equation}\label{eq:SpencerDifferentialOneForms}
    \begin{split}
        \left(\delta\xi\right)&(\alpha)(w_1, \dots, w_{k+l + 1})= \\
        &= \frac{1}{(k+1)!l!}\sum_{\sigma\in S_{k+l+ 1}}(-1)^\sigma \left(\xi(\alpha)(w_{\sigma(1)}, \dots, w_{\sigma(l)})\right) (w_{\sigma(l+1)}, \dots, w_{\sigma(l+k + 1)}) 
    \end{split}
\end{equation}
where $w_1, \dots, w_{k+l+1} \in \Gamma(W)$. Its symbol acts on a function $f\in C^\infty(N)$ by
\begin{equation}\label{eq:SpencerDifferentialFunctions}
    \begin{split}
        &\left(\delta \xi \right)(f)(w_1, \dots, w_{k+l}) \\
        &= \frac{1}{k!l!}\sum_{\sigma\in S_{k+l}} (-1)^\sigma \left(\xi(f)(w_{\sigma(1)}, \dots, w_{\sigma(l)})\right)(w_{\sigma(l+1)}, \dots, w_{\sigma(l+k)}) 
    \end{split}
\end{equation}

\begin{example}\label{ex:normalTableauAsGeneralizedTableau}
    There is more than one way of interpreting a classical tableau bundle as a tableau of derivations.
    
    One direct way is through 0-derivations. Let $\varphi\: W\to V$ be any bundle map covering the identity on $M$ (e.g., it can be the zero map). From the short exact sequence (\ref{eq:symbolsesrelativederivations}), it follows that $\Hom(W, V)\subset \DD^0_{\varphi}$. Hence, a classical tableau bundle $\TT\subset \Hom(W, V)$ can be seen as a 0-tableau in $\DD^0_\varphi$. 
    
    Another interpretation of a classical tableau as a bundle of derivations is through relative 1-derivations. Assume that the bundle map $(\varphi,p)\:W\to V$ is a fiberwise isomorphism. By Lemma \ref{lem:structureofRelativeDerivations}, there is a short exact sequence
    \[
    \xymatrix{
        0\ar[r] & \Hom(W, \ker(\d p)) \ar[r] & \DD^1_W \ar[r]^{\varphi_*} & p^*\DD^1_V \ar[r] & 0}.
    \]
    A subbundle $\TT\subset \Hom(W, \ker(\d p))$ is both a tableau in the classical sense as well as a tableau of 1-derivations. 
    
    In both cases, the Spencer differential on $\Hom(W, \TT)$, interpreted as a tableau of derivations, corresponds to the classical Spencer differential for $\TT$ as a classical tableau. Therefore, both interpretations of $\TT$ as a tableau of derivations recover the usual prolongation and Spencer cohomology theory.
\end{example}

\begin{remark}
    There is a warning: while na\"{i}vely the subbundle $\TT\subset \Hom(\wedge^2W,V)$ is a classical tableau (in the sense that it is a subbundle of a $\Hom$-space), it really depends on the context whether it should be treated as such. Usually, the exterior power $\wedge^2 W$ indicates the presence of brackets. Interpreting $\Hom(\wedge^2W, V)$ as a classical tableau leads to different Cartan characters and prolongations. For instance, to compute the characters of $\TT$ as a classical tableau, a flag of $\wedge^2W$ must be used. However, as we shall see, in the derivation picture one only requires a flag of $W$.
\end{remark}

\subsubsection{Prolongations and Spencer cohomology} 
Given a tableau of derivations we define its prolongation as follows.

\begin{definition}
    The \textbf{first prolongation} of a tableau $\TT\subset \DD^k_\varphi$ is defined as
    \[
        \mathcal{T}^{(1)} = \ker \delta\big\vert_{\Hom(W, \mathcal{T})}.
    \]
\end{definition}

The prolongation of a tableau $\TT \subset \DD^k_\varphi$ is a subbundle $\TT^{(1)} \subset \Hom(W, \TT)$ whenever it has constant rank. Hence, in this case, $\TT^{(1)}$ is a classical tableau and one defines the higher prolongations of $\TT$, when they exist, recursively as
\[
    \TT^{(m)} := (\TT^{(m-1)})^{(1)}.
\]  
The Spencer differentials 
\[
    \delta\: \Hom(\wedge^l W, \TT^{(m)}) \to \Hom(\wedge^{l+1} W, \TT^{(m-1)})
\]
are defined for $m \geq 1$ as in Definition \ref{def:SpencerDifferentialModifiedTableau} by restricting to $\Hom(\wedge^lW, \TT^{(m)})$ (i.e., regarding $\TT^{(m)}$ as classical tableau for $m\geq 1$, as in Example \ref{ex:normalTableauAsGeneralizedTableau}).{ Together they form a sequence of complexes
\begin{equation}\label{eq:SpencerComplex}
    \xymatrix{
        0 \ar[r] & \TT^{(m)} \ar@{^{(}->}[r] & \Hom(W, \TT^{(m-1)}) \ar[r]^-{\delta} & \ldots \ar[r]^-{\delta} & \Hom(\wedge^{m}W, \TT) \ar[r]^-{\delta} & \DD^{k+m}_\varphi.
    }
\end{equation}
If necessary, we extend this complex with the zero differential.
\begin{lemma}
    $\delta^2=0$.
\end{lemma}
\begin{proof}
    Let $\xi\in \Hom(\wedge^lW, \TT^{(m)})$. If $m\geq 2$, then both steps in the sequence correspond to the classical Spencer differential so $\delta^2\xi=0$. When $m=1$, some care needs to be taken. The essential part is that $\xi$ takes values in $\TT^{(1)}$, which by definition is the kernel of $\delta\: \Hom(W, \TT)\to \DD^{k+1}_\varphi$. The argument goes as follows. Let $w_1, \dots, w_{k+l+1}\in W$ and $\alpha\in \Omega^1(V)$. We use the following notation:
    \begin{itemize}
        \item for a natural number $n$ we set $[n]=\{1, \dots, n\}$, remembering the order. 
        \item if $J\subset [n]$ is an ordered subset of length $L$, we set
        \[
            w_J := w_{j_1}\wedge \dots \wedge w_{j_L}, \quad \mbox{ for $j_1<\dots <j_L\in J$}.
        \]
    \end{itemize}
    Then, according to Equation \eqref{eq:SpencerDifferentialOneForms}:
    \begin{align*}
        (\delta^2\xi)(\alpha)(w_1, \dots, w_{k+l+1}) &\propto  \sum_\sigma \sum_{i=1}^{l} (-1)^i(-1)^\sigma \xi(w_{\sigma\left([l]\setminus{i}\right)})(w_{\sigma(i)})(\alpha)(w_{\sigma\left([k+l+1]\setminus[l]\right)})\\
        &\propto \sum_\sigma (-1)^\sigma \xi(w_{\sigma([l-1])})(w_{\sigma(l)})(\alpha)(w_{\sigma\left([k+l+1]\setminus[l]\right)})\\
        &\propto \sum_{\substack{J\subset [k+l+1] \\ \text{s.t. } |J|=l-1}} \sum_{\substack{\sigma \text{ s.t.} \\  \sigma([l-1])= J}} (-1)^\sigma \xi(w_{J})(w_{\sigma(l)})(\alpha)(w_{\sigma\left( [k+l+1]\setminus[l]\right)})\\
        &\propto \sum_{\substack{J\subset [k+l+1] \\ \text{s.t. } |J|=l-1}} \delta\left( \xi(w_J)\right)(\alpha)(w_{[k+l+1]\setminus J})
    \end{align*}
    which vanishes because $\xi(w_J)\in \TT^{(1)}= \ker \delta$. 
\end{proof}

}

\begin{definition}\label{def:SpencerCohomology}
    The \textbf{Spencer cohomologies} of a tableau $\TT \subset \DD^k_\varphi$ are the cohomology groups of the complexes in Equation \eqref{eq:SpencerComplex}. Concretely,\begin{align*}
        H^{m, l}(\mathcal{T}) &:= \frac{\ker \delta\: \Hom(\wedge^l W, \TT^{(m)}) \to \Hom(\wedge^{l+1} W, \TT^{(m-1)}) }{\im \delta\: \Hom(\wedge^{l-1} W, \TT^{(m+1)}) \to \Hom(\wedge^{l} W, \TT^{(m)})}, \mbox{ for $m\geq 1$},\\
        H^{0, l}(\TT)&:= \frac{\ker \delta\: \Hom(\wedge^l W, \TT)\to \DD^{k+l}_{\varphi}}{\im \delta\: \Hom(\wedge^{l-1}W, \TT^{(1)})\to \Hom(\wedge^l W, \TT)},\\
        H^{-1, l}(\TT)&:= \frac{\DD^{k+ l-1}_\varphi}{\im \delta\: \Hom(\wedge^{l-1} W, \TT)\to \DD^{k+l-1}_\varphi}, \quad \mbox{for $l\geq 1$}.
    \end{align*}
\end{definition}

\begin{remark}
    Note that $H^{m, 0}(\TT)=0$ when $m\geq 1$ and $H^{m, 1}(\TT) = 0$ for $m\geq 0$, from the definition of the prolongation. 
\end{remark}  

{
\begin{remark}
    The index $m$ in the Spencer cohomology groups of Definition \ref{def:SpencerCohomology} tracks the \emph{prolongation order} of the structure, which is typical in the theory of $G$-structures (e.g.\@ \cite{GuilleminSternberg1964}). In the classical literature on formal PDEs and Spencer cohomology, the index $m$ tracks the order of the PDE. The difference is a shift by one: if $\TT$ is a classical tableau, then $H^{m, l}(\TT)$ from Definition \ref{def:SpencerCohomology} corresponds to $H^{m+1, l}(\TT)$ in the classical literature on formal PDEs (e.g.\@ \cite{BryantChern1991}).
\end{remark}
}

\begin{definition}
    A tableau $\TT\subset \DD^k_\varphi$ is called \textbf{involutive} if all prolongations $\TT^{(m)}$, $m\geq 1$, have constant rank and
    \[
        H^{m,l}(\TT) = 0, \quad \mbox{ for all $m\geq 0, l\geq 1$}.
    \]
\end{definition}
{
\begin{remark}
    Under the assumption that $H^{m,l}(\TT)=0$ for all $m\geq 0, l\geq 1$, the condition that $\TT^{(m)}$ has constant rank for $m\geq 1$ is equivalent to the Spencer cohomology groups $H^{-1, l}(\TT)$ having constant rank for $l\geq 2$, which follows from the Spencer complex \eqref{eq:SpencerComplex}.

    Moreover, since $\TT^{(1)}$ is a classical tableau, the stabilization result by Serre (Appendix of \cite{GuilleminSternberg1964}) implies that there exists $m_0$ such that $H^{m, l}(\TT)=0$ for $m\geq m_0$ and $l\geq 1$. Thus, only finitely many conditions are needed to verify involutivity of a tableau of derivations. 
\end{remark}
}

\subsubsection{Cartan characters and Cartan's test} 

Heuristically, a system of differential equations is in \textit{involution} when there are no higher order hidden consequences of the equations. These higher order consequences appear as cohomology classes in the Spencer cohomology groups which, in general, are hard to compute. A practical way of checking involutivity is through Cartan's test. In this section, we make these notions precise for tableaux of derivations. 

We continue to assume that $(\varphi,p)\: W\to V$ is a fixed morphism of vector bundles. By a \emph{flag} of $W$ we mean a sequence of vector subbundles of $W$,
\[
    0 = W_0 \subset W_1 \subset \dots \subset W_n = W,\quad \text{ with }\rk W_i = i.
\]
Fix a flag $(W_i)$ for $W$, denote by $\iota_i\: W_i\hookrightarrow W$ the inclusion map and $\iota_i^*\: \wedge^\bullet W^*\to \wedge^\bullet W_i^*$ the induced restrictions. A relative $k$-derivation $\D\: \Omega^\bullet (V)\to \Omega^{\bullet+k}(W)$ can be post-composed with the restriction map to yield a $k$-derivation $\iota_i^*\circ \D$ relative to $\varphi\circ\iota_i$ (cf.~Example \ref{ex:relativeDerivationsPulledBack}). In other words, there is restriction map
\[
    (\iota_i)_*\: \DD^k_\varphi \to \DD^k_{\varphi\circ\iota_i}.
\]
If $\TT\subset \DD^k_\varphi$ is a tableau, we set 
\[
    \TT_i := \ker (\iota_i)_*\cap \TT.
\]
Notice that $\TT_{i}\subseteq \TT_{i-1}$ and that $\TT_0 = \TT$, $\TT_n = 0$ where $n=\rk W$.

\begin{definition}
    Let $\TT\subset \DD^k_\varphi$ be a tableau and $(W_i)$ a flag for $W$. The \textbf{Cartan characters} of $\TT$ with respect to the flag are
    \[
        s_i:= \rk \TT_{i-1} -\rk \TT_{i},\quad (i=1,\dots,n=\rk W).
    \]
\end{definition}

As for classical tableaux, the Cartan characters bound the rank of the prolongation and, moreover, provide a practical way to verify that a tableau of derivations is involutive through Cartan's test. Here we present an extension of Cartan's test to tableaux of derivations.

\begin{theorem}\label{thm:CartansTest}
    Let $\TT\subset \DD^k_\varphi$ be a tableau of derivations.
    \begin{enumerate}[(i)]
        \item \emph{(Generalized Cartan's bound).} If $\{s_i\}$ are the Cartan characters w.r.t.\,a flag $(W_i)$, the dimension of the prolongation is constrained by
        \[
            \rk \TT^{(1)} \leq s_1 + 2s_2 + \dots + n s_n.
        \]
        \item \emph{(Generalized Cartan's test).} If the Cartan characters $s_i$ are locally constant and Cartan's bound is achieved, i.e.,
        \[
            \rk \TT^{(1)} = s_1 + 2s_2 + \dots + n s_n,
        \]
        then $\TT$ is involutive.
    \end{enumerate}
\end{theorem}

A flag for which Cartan's test holds is called a \textbf{regular flag} for $\TT$. From the proof below, it will be clear that if $(W_i)$ is a regular flag for $\TT$, then it is also a regular flag for the first prolongation $\TT^{(1)}$. 

\begin{remark}
    For a classical tableau, Cartan's test provides an equivalent characterization of involutivity (see \cite{Seiler2007}, Theorem 3.4). We suspect that it is also an equivalence for tableaux of derivations, but it does not seem to follow from Cartan's test applied to the prolongation as a classical tableau. The reason is that if $\TT$ is a tableau, regular flags for the prolongation $\TT^{(1)}$ may not be regular for $\TT$ itself. This poses no obstacle to our applications: computing the Cartan characters is typically much more practical than computing the Spencer cohomology groups, making the implication in Theorem \ref{thm:CartansTest} the most relevant one.
\end{remark}


Before we give the proof of the previous theorem, we make the following observation. Since the first prolongation $\TT^{(1)}\subset \Hom(W, \TT)$ is also a tableau, it comes with spaces $(\TT^{(1)})_i = \ker \iota_i^*\big\vert_{\TT^{(1)}}$, associated to a choice of flag for $W$, where now $\iota_i^*\:\Hom(W, \TT)\to \Hom(W_i, \TT)$ is the restriction map. These are related to the prolongations of $\TT_i$ as follows.

\begin{lemma}\label{lem:prolongationAndFlags}
    Let $\TT\subset \DD^k_\varphi$ be a tableau of derivations and fix a flag $(W_i)$ for $W$. Then:
    \[ \left(\TT^{(1)}\right)_i \subseteq (\TT_i)^{(1)}, \quad (i=1,\dots,n). \]
\end{lemma}

\begin{proof}
    Let $\xi\in (\TT^{(1)})_i$, so $\xi(w) = 0$ for all $w\in W_i$. Then, for any $\alpha\in \Omega^1(V)$, $u \in W$ and $w_0, \dots, w_k\in W_i$,  applying \eqref{eq:SpencerDifferentialOneForms} one finds
    \begin{align*}
        0 &= (\delta \xi)(\alpha)(u, w_0, \dots, w_k) \\
        &= \left(\xi(u) \alpha\right)(w_0, \dots, w_k) - \sum_{j=0}^k(-1)^j\left( \xi(w_j)\alpha\right)(u, w_0,\dots, \widehat{w_i}, \dots, w_k)\\
        &= \left(\xi(u)\alpha\right)(w_0, \dots, w_k).
    \end{align*}
    Therefore, $\xi(u)\in \TT_i$ for all $u\in W$, and so $\xi\in \left(\TT_i\right)^{(1)}$.
\end{proof}

\begin{remark}
    For a tableau $\TT\subset \Hom(W, V)$ in the classical sense, the inclusion in Lemma \ref{lem:prolongationAndFlags} is an equality. However, in our more general case this may fail. Consider, for example, the case where $W = V$ is a 2-dimensional vector space, with flag $W_0\subset W_1\subset W_2$, and let 
    \[ \TT:= \DD^1_\varphi = \Hom(\wedge^2W, W).\] 
    Then $\TT_1 = \Hom(\wedge^2W, W) = \TT$, and so $(\TT_1)^{(1)} = \TT^{(1)} = \Hom(W, \Hom(\wedge^2W, W))$. On the other hand, $(\TT^{(1)})_1 = \Hom(W_1, \Hom(\wedge^2W, W))\subsetneq (\TT_1)^{(1)}$.
\end{remark}

The proof of Theorem \ref{thm:CartansTest} rests on the following lemma.

\begin{lemma}\label{lem:GeneralizedCartansBound}
    If $L := \ker \left( \delta\: \Hom(\wedge^lW, \TT)\to \DD^{k+l}_\varphi\right)$,
    then
    \begin{equation}
    \rk L \leq \sum_{i=1}^ns_i \left( \binom{n}{l} - \binom{n-i}{l}\right)
    \end{equation}
    for any flag $(W_i)$ of $W$.
\end{lemma}

\begin{proof}
    Let $(W_i)$ be a flag of $W$ and set
    \[
        L_i:= L \cap \Hom(\wedge^l W, \TT_i).
    \]
    These define a filtration $L = L_0 \supseteq \dots \supseteq L_{i-1} \supseteq L_{i} \supseteq \dots \supseteq L_n = 0$, and there are natural inclusions
    \[ L_{i-1}/L_i\hookrightarrow \Hom\left(\wedge^lW, \TT_{i-1}/\TT_i\right). \]
    If $\pi_i\: W \to W/W_i$ is the projection, then pullback gives another set of inclusions
    \[ \pi_i^*\: \Hom\left(\wedge^l(W/W_i), \TT_{i-1}/\TT_i\right)\to \Hom\left(\wedge^lW,\TT_{i-1}/\TT_i\right).\]
    \begin{claim*}
        $\left(L_{i-1}/L_i\right)\cap \im \pi_i^* = \{0\}$.
    \end{claim*}    
    To prove this claim let $\pi_i^*\xi\in L_{i-1}\cap \pi_i^*\Hom\left(\wedge^l\left(W/W_i\right), \TT_{i-1}\right)$. If $w_1, \dots w_l\in W$ and $w_{l+1}, \dots, w_{k+l+1}\in W_i$, then by \eqref{eq:SpencerDifferentialOneForms} for $\alpha\in \Omega^1(V)$ we have
    \begin{align*}
        0 &= \delta\left( \pi_i^*\xi\right)(\alpha)(w_1, \dots, w_{k+l + 1})\\
        &= \frac{1}{l!(k+ 1)!} \sum_{\sigma \in S_{k+l + 1}} (-1)^\sigma (\xi(\alpha)\left(\pi_i\left(w_{\sigma(1)}\right),\dots \pi_i\left(w_{\sigma(l)}\right)\right) \left(w_{\sigma(l+1)}, \dots, w_{\sigma(k+l)}\right)\\
        &=((\pi_i^*\xi)(\alpha)(w_1,\dots, w_l))(w_{l+1},\dots,  w_{k+l + 1}).
    \end{align*}
    It follows that $\pi_i^*\xi\in \Hom\left(\wedge^l W, \TT_{i}\right)$, and the claim follows.

    The estimate of the lemma now follows from this claim by telescoping:
    \begin{align*}
        \rk L &= \sum_{i=1}^n \rk L_{i-1} - \rk L_i \\
        &\leq\sum_{i=1}^n \rk \Hom(\wedge^lW, \TT_{i-1}/ \TT_i) - \rk \Hom(\wedge^l(W/W_i), \TT_{i-1}/\TT_i)\\
        &= \sum_{i=1}^ns_i \left( \binom{n}{l} - \binom{n-i}{l}\right). \qedhere
    \end{align*}
\end{proof}

\begin{proof}[Proof of Theorem \ref{thm:CartansTest}]
    Note that the case $l=1$ in Lemma \ref{lem:GeneralizedCartansBound} is exactly item (i) (Cartan's bound) in the theorem.

    For the proof of item (ii), suppose that $(W_i)$ is a regular flag for $\TT$, with characters $(s_i)$, so that Cartan's test on $\dim \TT^{(1)}$ is satisfied. Since the Cartan characters are assumed locally constant, so is the rank of $\TT^{(1)}$. 
    
    The proof goes in two steps:
    \begin{description}
        \item[Step 1] Prove the inequality 
            \begin{equation}
            \label{eq:inequaility:Cartan:Test:proof}
            \rk (\im \delta) \geq \sum_{i=1}^n s_i \left( \binom{n}{l} - \binom{n-i}{l} \right)
            \end{equation}
            for $l = 2$ and that Cartan's bound is achieved for $\rk \TT^{(2)}$ -- as the first prolongation of $\TT^{(1)}$ -- for the same flag $(W_i)$. Then, by Lemma \ref{lem:GeneralizedCartansBound}, it follows that
            \[ H^{0, 2}(\TT)=0,\]        
            and that -- by Cartan's test for ordinary tableaux (\cite[Theorem 3.4]{Seiler2007}) -- all prolongations $\TT^{(m)}$, $m\geq 1$, have locally constant rank and are involutive, so that 
            \[ H^{m,l}(\TT) = H^{m-1, l}(\TT^{(1)}) = 0,\quad \text{for }m\geq 1,\, l \geq 0;\]
        \item[Step 2] Prove inequality          \eqref{eq:inequaility:Cartan:Test:proof} for arbitrary $l$ so that, applying Lemma \ref{lem:GeneralizedCartansBound}, one has
        \[ H^{0,l}(\TT)=0,\quad \text{for }l\ge 0.\]  
    \end{description} 
    
    \textit{Step 1.} First, we claim that the Cartan characters $s^{(1)}_i$ of the first prolongation satisfy the bound:
    \begin{equation}
        \label{eq:characters:1st:prolongation}
        s^{(1)}_i \leq s_i + s_{i+1} + \dots + s_n.
    \end{equation}
    This follows from the Euler characteristic of the exact sequence
    \[
        \xymatrix{0 \ar[r]& \left(\TT^{(1)}\right)_i \ar[r]& \left( \TT^{(1)}\right)_{i-1} \ar[r]^{\iota_{w_i}} & \TT_{i-1} }
    \]
    where $\iota_{w_i}$ is interior contraction by a local frame  $(w_i)$ adapted to $(W_i)$, i.e., $\{w_1, \dots, w_i\}$ is a frame for $W_i$.
    
    From the bound on $s_i^{(1)}$ together with Cartan's bound for $\rk \TT^{(2)}$, we see that
    \[ \rk \TT^{(2)} \leq \sum_{i=1}^n is^{(1)}_i \leq \sum_{i=1}^n \binom{i+1}{2}s_i.\] 
    Furthermore, from the exact row of the Spencer complex
    \[
        \xymatrix{0 \ar[r]& \TT^{(2)} \ar[r]&  \Hom(W, \TT^{(1)}) \ar[r]^{\delta} & \Hom(\wedge^2W, \TT)\ar[r] & \DD^{k+2}_\varphi},
    \]
    we conclude that 
    \begin{align}
    \begin{split}\label{eq:inequalityImageDifferential}
        \rk (\im \delta) &= \rk \left(\Hom(W, \TT^{(1)})\right) - \rk \TT^{(2)}\\
        &\geq \sum_{i=0}^n s_i\left( n i - \binom{i+1}{2}\right)\\
        &= \sum_{i=1}^ns_i \left( \binom{n}{2} - \binom{n-i}{2}\right),
    \end{split}
    \end{align}
    where we used the assumption that Cartan's bound is achieved for $\rk \TT^{(1)}$. This proves the inequality \eqref{eq:inequaility:Cartan:Test:proof} for $l=2$. But then by Lemma \ref{lem:GeneralizedCartansBound}, the inequality \eqref{eq:inequalityImageDifferential} must be an equality, so that
    \[ \rk \TT^{(2)} = \sum_{i=1}^n s_i\binom{i+1}{2}. \] 
    Hence, Cartan's bound is achieved for $\TT^{(2)}$ and the proof of step 1 is concluded. Incidentally, this also shows that one must have equality in \eqref{eq:characters:1st:prolongation}.

    \textit{Step 2.} It remains to prove \eqref{eq:inequaility:Cartan:Test:proof} for $l\geq 3$. This bound comes again from the Spencer complex:
    \[
    \begin{split}
        \xymatrix{ 0 \ar[r] & \TT^{(l)} \ar[r] & \Hom(W, \TT^{(l-1)}) \ar[r] & \Hom(\wedge^2 W, \TT^{(l-2)}) \ar[r] & \dots }\qquad\\
        \qquad\xymatrix{\dots \ar[r]&\Hom(\wedge^{l-1} W, \TT^{(1)}) \ar[r]^---{\delta} & \Hom(\wedge^lW, \TT)  \ar[r] & \DD^{k+l}_\varphi }
    \end{split}
    \]
    Because $\TT^{(1)}$ is involutive, this sequence is exact up until $\im \delta$. It follows that
    \[
        \rk(\im \delta) = \sum_{k=1}^l (-1)^{k+1} \rk\left(\Hom(\wedge^{l-k}W, \TT^{(k)})\right).
    \]
    We already know that Cartan's bound is achieved for all prolongations, so $\rk\TT^{(k)}$ can be given entirely in terms of the Cartan characters of $\TT$. An induction argument using $s^{(1)}_i = s_i + \dots + s_n$ gives 
    $\rk \TT^{(k)} = \sum_{i=1}^n s_i \binom{i + k - 1}{k}$, 
    and therefore
    \[
        \rk(\im \delta) =\sum_{i=1}^n s_i \left( \sum_{k=1}^l (-1)^{k+1} \binom{n}{l-k}\binom{i+k-1}{k} \right).
    \]
    The result now follows from the combinatorial identity
    \[
        \sum_{k=1}^l (-1)^{k+1} \binom{n}{l-k}\binom{i+k-1}{k} = \binom{n}{l} - \binom{n-i}{l}.\qedhere
    \]
\end{proof}

\subsubsection{Symbol exact sequences of tableaux and involutivity}\label{subsubsec:symbolExactSequences}

We continue to assume that $(\varphi,p)\: W \to V$ is a vector bundle map.

\begin{proposition}\label{prop:tautologicalTableauIsInvolutive}
    Both $\DD^k_{\varphi}$ and $\Hom(\wedge^{k+1}W, p^*V)\subset \DD^k_\varphi$ are involutive tableaux for which every flag is regular.
\end{proposition}

In order to prove this proposition, we will use the following lemma which relates the Spencer differential and 
the short exact sequence \eqref{eq:symbolsesrelativederivations} induced by the symbol map.

\begin{lemma}\label{lemma:symbolAndSpencerDifferential}
    The following diagram of short exact sequences commutes
    \[
    \xymatrix{
       {\Hom(W, \Hom(\wedge^{k+1}W, p^*V))} \ar[d] \ar[r] & {\Hom(W, \DD^k_\varphi)} \ar[d]^{\delta} \ar[r] & {\Hom(W, \Hom(\wedge^kW, p^*TN))} \ar[d] \\
        {\Hom(\wedge^{k+2}W, p^*V)} \ar[r]                    & \DD^{k+1}_\varphi \ar[r]                            & {\Hom(\wedge^{k+1}W, p^*TN)},
    }
    \]
    where the side vertical arrows are skew-symmetrization. In particular, the Spencer differential $\delta\:\Hom(W, \DD^k_{\varphi})\to \DD^{k+1}_\varphi$ is surjective.
\end{lemma}

\begin{proof}
    This is clear from \eqref{eq:SpencerDifferentialOneForms} and \eqref{eq:SpencerDifferentialFunctions}.
\end{proof}

\begin{proof}[Proof of Proposition \ref{prop:tautologicalTableauIsInvolutive}]
    Consider first the tableau $\TT := \Hom(\wedge^{k}W,p^*V)$. In this case the Spencer differential is the skew-symmetrization map
    \[
        \delta\: \Hom(W, \Hom(\wedge^{k}W, p^*V))\to \Hom(\wedge^{k+1}W, p^*V)
    \]
    which is surjective. Setting $n = \rk W$ and $m = \rk V$, we have 
    \[ \rk \TT^{(1)} = m \left( {n \binom{n}{k} - \binom{n}{k+1}}\right)=m\,k \binom{n+1}{k+1}.\]
    Now let $(W_i)$ be a (local) flag for $W$. From the exact sequence
    \[
        \xymatrix{0 \ar[r] & \TT_i \ar[r] & \Hom(\wedge^kW, p^*V) \ar[r] & \Hom(\wedge^k W_i, p^*V)\ar[r] & 0}
    \]
    it follows that $s_i = m \left(\binom{i}{k} - \binom{i-1}{k}\right) = m \binom{i-1}{k-1}$. Hence, we find that 
    \[ \sum_{i=1}^n i s_i = mk\sum_{i=k}^n \binom{i}{k}=m\,k \binom{n+1}{k+1}=
    \rk \TT^{(1)}, \] 
    so Cartan's Test \ref{thm:CartansTest} holds and $\TT$ is involutive.

    Next, to show that $\DD^k_{\varphi}$ is also involutive we apply Lemma \ref{lemma:symbolAndSpencerDifferential}. Setting $\DD = \DD^k_{\varphi}$ and $\mathcal{S} = \Hom(\wedge^{k} W, p^*TN)$, the lemma gives a short exact sequence
    \[
    \xymatrix{0 \ar[r] & \TT^{(1)} \ar[r] &  \DD^{(1)} \ar[r] &  \mathcal{S}^{(1)} \ar[r] &  0.}
    \]  
    Moreover, restriction gives a commutative diagram
    \[
    \xymatrix{
    0 \ar[r] & {\Hom(\wedge^{k+1}W, p^*V)} \ar[d] \ar[r] & \DD^{k}_\varphi \ar[d] \ar[r]        & {\Hom(\wedge^{k}W, p^*TN)} \ar[d]\ar[r] & 0 \\
            0 \ar[r]& {\Hom(\wedge^{k+1}W_i, p^*V)} \ar[r]         & \DD^k_{\varphi\circ \iota_{W_i}} \ar[r] & {\Hom(\wedge^kW_i, p^*TN)}        \ar[r]   & 0
    }
    \]
    so we obtain that the restricted spaces also fit into a short exact sequence
    \[
    \xymatrix{0 \ar[r] & \TT_i \ar[r] &  \DD_i \ar[r] &  \mathcal{S}_i \ar[r] &  0.}
    \]
    Since Cartan's test is satisfied for both $\mathcal{S}$ and $\TT$, we conclude that it is satisfied for $\DD$, which then must also be involutive.
\end{proof}

\section{Relative algebroids}\label{sec:RelativeAlgebroids}

\subsection{Relative algebroids} 

With the formalism of relative derivations at hand, we are now ready to introduce the geometric objects underlying Cartan's realization problem and Bryant's equations (\ref{eq:BryantsEquations}).

\begin{definition}
\label{def:relative:algebroid}
    An \textbf{almost Lie algebroid relative to a submersion} or in short \textbf{algebroid relative to a submersion} $p\: M \to N$ consists of a vector bundle $A\to N$ together with a 1-derivation
    \[
        \D\: \Omega^\bullet (A)\to \Omega^{\bullet + 1} (p^*A)
    \]
    relative to $p$.  
\end{definition}

We will denote by $(A, p, \D)$ the entire structure of the algebroid relative to $p$ and often write $B$ for the pullback bundle $p^* A\to M$. The manifolds $M$ and $N$ will be specified when those are not clear from the context. We like to graphically depict an algebroid relative to $p$ as
\[
    \xymatrix{
    B \ar[d] \ar[r] & A \ar[d] \ar@{-->}@/^1pc/[l]^{\D} \\
    M \ar[r]_p  & N                   
    }
\]
where the dotted arrow is not a map but indicates a derivation that is defined on $\Omega^\bullet(A)$ and takes values in $\Omega^{\bullet+1}(B)$. We use the letter $B$ to stand for Bryant, whose equations (\ref{eq:BryantsEquations}) inspired this definition. This justifies in part the change of notation from the previous section, where the bundles $A$ and $B$ were denoted by $V$ and $W$, respectively.

According to Lemma \ref{lem:derivationsAndBrackets}, the structure of an algebroid relative to $p\: M\to N$ can also be encoded by a relative bracket and anchor
\[
    [\cdot, \cdot]\: \wedge^2\Gamma(A)\to \Gamma(B), \quad \rho\: B\to p^*TN
\]
subject to the Leibniz rule:
\[
    [a_0, fa_1] = p^*f [a_0, a_1] + \LL_{\rho(p^*a_0)}(f) p^*a_1.
\]
The corresponding derivation is determined through the Koszul formula (see the proof of Lemma \ref{lem:derivationsAndBrackets}).

\begin{example}
    In the case that $M=N$ and $p=\id$, we recover a notion of an \textbf{almost Lie algebroid}: a vector bundle $A\to N$ together with an anchor and a bracket subject to the Leibniz rule, or, dually, with a degree 1 derivation $\D_A$ on $\Omega^\bullet(A)$. Such an almost Lie algebroid is a Lie algebroid when $\D_A^2 =0$. In particular, we can consider any (almost) Lie algebroid as an algebroid relative to the identity. A special case is the tangent bundle $(TP,\d)\to P$ with the de Rham differential $\d$. 

    At this point, when dealing with an arbitrary relative algebroid, it is not clear how to make sense of ``$\D^2 =0$''. We will soon see how to handle this issue. This issue also influences our terminology: the objects in Definition \ref{def:relative:algebroid} should properly be called \emph{almost relative Lie algebroid}. This name is too long and so we refer to them simply as \emph{relative algebroids}, removing ``Lie'' from its name rather than adding ``almost''. Later, after we make sense of $\D^2=0$, we will be able to define \emph{relative Lie algebroids}.
\end{example}

\begin{example}[Relative algebroids in coordinates]\label{ex:relativeAlgebroidsInCoordinates}
    Let $(U,(x^\mu, y^\varrho))$ be a coordinate system for $M$ and $(V,x^\mu)$ a coordinate system for $N$ such that
    \[ p(x^\mu, y^\varrho)=x^\mu. \]
    Also, fix $\{e_i\}$ a frame for $A|_V$, and let $\{\theta^i\}$ be the corresponding dual coframe. Then the derivation $\D$ is determined by
    \begin{equation}\label{eq:relativeAlgebroidInCoordinates}
        \begin{cases}
            \D \theta^i = -\frac{1}{2}c^{i}_{jk}(x^\mu, y^\varrho)\, \theta^j\wedge\theta^k,\\
            \D x^\mu = F^\mu_i(x^\mu, y^\varrho)\, \theta^i.
        \end{cases}
    \end{equation}
    for functions $c^{i}_{jk}, F^\mu_i\in C^\infty(M)$. Dually, the anchor and bracket are given by
    \[ 
    \begin{cases}
        [e_i,e_j]=c_{ij}^k e_k,\\
        \rho(e_i)=F^\mu_i\partial_{x^\mu}.
    \end{cases}
    \]
    The equations above look strikingly similar to Bryant's equations \eqref{eq:BryantsEquations} except that one has $\D$ instead of $\d$ and the $c^{i}_{jk}$'s also depend on the ``free derivatives'' $y^\varrho$. 
\end{example}

A true globalization of Bryant's equations requires a globalization of the projection. The language introduced in the previous section for derivations relative to foliations precisely captures this idea. This is not a significant generalization; as we will see in Section \ref{sec:prolongation}, whenever the first prolongation of a relative algebroid exists, it is always relative to a submersion.

Nevertheless, there are natural examples of relative algebroids that are relative to a foliation rather than a submersion. These include relative algebroids underlying Pfaffian fibrations \cite{CattafiCrainicSalazar2020} and those obtained through restrictions (Proposition \ref{prop:restriction}).

\begin{definition}
    An \textbf{algebroid relative to a foliation} $(M,\FF)$ is a flat foliated vector bundle $(M,\FF,B,\overline{\nabla})$ together with a 1-derivation $\D$ relative to $\FF$:
    \[
        \D\: \Omega^{\bullet}_{(B, \overline{\nabla})}\to \Omega^{\bullet + 1}_B.
    \]
    We will write $(B, \overline{\nabla}, \D)$ to denote a relative algebroid over $(M, \FF)$. If there is ambiguity in what $(M, \FF)$ could be, we write $(B, \overline{\nabla}, \D)\to (M, \FF)$.
\end{definition}

\begin{example}
    \label{ex:relativeAlgebroidsInCoordinates:2}
    Let $(A, p, \D)$ be an algebroid relative to a submersion $p\:M\to N$. Then it can be viewed as an algebroid relative to the foliation $\FF=\ker(\d p)$. Namely, the vector bundle $B=p^*A$ carries a canonical flat $\FF$-connection $\onabla$ determined by requiring
    \[
    \overline{\nabla}p^*a = 0,\quad \text{ for all } a\in \Gamma(A),
    \]
    and since $\Omega^\bullet_{(B, \overline{\nabla})}= p^*\Omega^\bullet_A$,
    we can view $\D$ as a derivation relative to $\FF$. 
    
    Conversely, any algebroid relative to a foliation is locally an algebroid relative to a submersion. It is globally an algebroid relative to a submersion precisely when $\FF$ is simple and $\overline{\nabla}$ has no holonomy along $\FF$ (see Proposition \ref{prop:BundlesWithFlatPartialConnectionsOfSimpleFoliations}).

    Suppose we choose foliation coordinates $(U,(x^\mu, y^\varrho))$ on $(M, \FF)$, so that plaques of $\FF$ correspond to $\{y^\varrho = c^\varrho\}$. Furthermore, let $\{e_i\}$ be a local frame  of flat sections of $B|_U$, so the dual coframe $\{\theta^i\}$ consists of local flat $B$-forms in $\Omega^1_{(B, \overline{\nabla})}(U)$. Then the derivation $\D$ is still determined by the same equations \eqref{eq:relativeAlgebroidInCoordinates}. Also, we now see that we can retrieve Bryant's equations, i.e., have the $c_{ij}^k$ not depend on the free variables, exactly when there is a local coframe $(\theta^i)$ such that $\overline{\nabla}(\D\theta^i)=0$. We will see later that this property is always satisfied for any prolongation of a relative algebroid (Definition \ref{def:prolongation}).
\end{example}

Using the notion of $(\varphi, p)$-related derivations (see Definition \ref{def:Phi-related:derivations}), morphisms of relative algebroids can be defined as follows:

\begin{definition}
    A \textbf{morphism} $(\varphi,p)\:(B_1, \overline{\nabla}_1, \D_1)\to (B_2, \overline{\nabla}_2,\D_2)$ of relative algebroids consists of 
    \begin{enumerate}[(i)]
        \item a map of foliations $p\: (M_1, \FF_1)\to (M_2, \FF_2)$,
        \item a map of flat foliated vector bundles $\varphi\: (B_1,\onabla_1) \to (B_2,\onabla_2)$ covering $p$,
    \end{enumerate}
    such that $\D_1$ is $(\varphi,p)$-related to $\D_2$:
    \[
        \D_1 \circ \varphi^* = \varphi^* \circ \D_2, \quad \text{ on $\Omega^\bullet_{(B_2, \onabla_2)}$}.
    \] 
\end{definition}

\subsection{Realizations}\label{sec:realizations}
In this section, we will fix an algebroid $(B, \overline{\nabla},\D)$ relative to a foliation $(M, \FF)$. 

The tangent bundle of any manifold $P$ is a Lie algebroid with derivation $\d$, the de Rham differential, so we will denote it by $(TP,\d)$. A realization is an object that ``realizes" an algebroid as the tangent bundle of a manifold. 

\begin{definition}
    A \textbf{realization} $(P, r, \theta)$ of $(B, \overline{\nabla}, \D)$ is a morphism of relative algebroids from $(TP, \d)$ to $(B, \overline{\nabla}, \D)$ that is fiberwise an isomorphism.
\end{definition}

Explicitly, a realization consists of a manifold $P$ together with a bundle map $(\theta,r)\: TP\to B$ that is fiberwise an isomorphism and satisfies
\[ \d \circ \theta^* = \theta^* \circ \D, \quad \text{on } \Omega^\bullet_{(B, \overline{\nabla})}. \]

\begin{remark}
\label{rem:realizations:MC}
    In \cite{CrainicYudilevich2024, FernandesStruchiner2014}, realizations of Lie algebroids are defined through Maurer-Cartan forms. In our situation, the bundle map $(\theta, r)$ can be reinterpreted as a one-form $\theta\in \Omega^1(P; r^*B)$ that is fiberwise an isomorphism. Choosing any extension $\nabla$ of $\overline{\nabla}$ on $B$, Lemma \ref{lem:sessymbolconnectionsplitting} gives a splitting under which the relative derivation $\d \theta^* - \theta^*\D\in \Gamma(r^*\DD^1_{(B, \overline{\nabla})})$ decomposes into two components: $(\MC^\nabla_\theta, \Pi \circ \d r - \rho \circ \theta)$, where $\Pi\: TM \to \nu(\FF)$ is the projection and the Maurer-Cartan form is given by
    \[
        \MC_\theta^\nabla:= \d^{\nabla} \theta + \frac{1}{2}[\theta, \theta]_\nabla.
    \]
    So we conclude that in terms of anchors and brackets a bundle map $(\theta, r)\: TP \to B$ is a realization of $(B, \overline{\nabla},\D)$  if and only if it is fiberwise an isomorphism and satisfies 
    \begin{equation}\label{eq:CoordinateInvariantBryantsEquations}
        \begin{cases}
            \d^\nabla \theta= -\frac{1}{2}[\theta, \theta]_{\nabla},\\
            \Pi \circ \d r = \rho \circ \theta.
        \end{cases}
    \end{equation}
    We conclude that:
    \begin{enumerate}[(i)]
        \item If $\theta$ is anchored (i.e. $\Pi\circ \d r = \rho \circ \theta$), then the Maurer-Cartan form is independent of $\nabla$.
        \item  If $\theta$ is anchored, it is a realization if and only if its Maurer-Cartan form vanishes.
    \end{enumerate}
    
    In local coordinates (see Examples \ref{ex:relativeAlgebroidsInCoordinates} and \ref{ex:relativeAlgebroidsInCoordinates:2}), writing $\theta=(\theta^i)\:TP\to \RR^n$ and $r=(a^\mu,b^\varrho)\:P\to \RR^s\times\RR^r$, equations \eqref{eq:CoordinateInvariantBryantsEquations} become
    \[
    \begin{cases}
        \d \theta^i = -\frac{1}{2}c^i_{jk}(a,b)\, \theta^j\wedge\theta^k,\\
        \d a^\mu = F^\mu_i(a, b)\, \theta^i,
    \end{cases}
    \]
    and these are exactly Bryant's equations \eqref{eq:BryantsEquations} (see Example \ref{ex:relativeAlgebroidsInCoordinates:2} for why the $c^i_{jk}$ may depend on the free derivatives $b^\varrho$).
\end{remark}

\subsection{Tableaux}\label{sec:TableauxRelativeAlgebroids}

Let $(B, \overline{\nabla},\D)$ be an algebroid relative to a foliation $(M,\FF)$. Recall from Lemma \ref{lem:ParallelConnectionsVectorBundleWithConnection}  that  there is a canonical flat $\FF$-connection on the vector bundle $\DD^1_{(B, \overline{\nabla})}$ also denoted by $\overline{\nabla}$. 

\begin{definition}\label{def:tableauMap}
    The \textbf{tableau map} of $(B, \overline{\nabla},\D)$ is the bundle map
    \[
        \tau\: \FF \to \DD^1_{(B, \overline{\nabla})}, \quad X\mapsto \overline{\nabla}_X \D
    \]
    The relative algebroid is called \textbf{standard} or \textbf{non-degenerate} when its tableau map $\tau$ is fiberwise injective.
\end{definition}

The tableau map measures the dependence of the relative algebroid structure on the directions of $\FF$. {It is a tableau map of derivations relative to foliations. By Remark \ref{rk:TableauxOfDerivationsRelativeToFoliations} and Remark \ref{rk:TableauxMapSpencerDifferential}, the theory of involutivity and Spencer cohomology carries over to this setting.
}

Composing a tableau map with the symbol map
\[
\xymatrix{
\FF \ar[r]^{\tau} \ar[dr]_---{\sigma(\tau)} & \DD^1_{(B, \overline{\nabla})}\ar[d]^{\sigma}\\
& 
\Hom(B, \nu(\FF))}
\]
we see that our tableau map covers a classical tableau map
\[
    \sigma(\tau)\: \FF\to \Hom(B, \nu(\FF)), \quad X\mapsto \overline{\nabla}_X\rho
\]
where $\rho = \sigma(\D)$ is the anchor of the relative algebroid. We call $\sigma(\tau)$ the \textbf{symbol tableau} of the relative algebroid. Explicitly, it is given by
\begin{equation}
    \label{eq:symbol:tableau}
    \sigma(\tau)(X)(b) = \overline{\nabla}^{\mathrm{Bott}}_X \rho(b) - \rho\left( \overline{\nabla}_X b\right), \quad \text{ for $X\in \FF$ and $b\in \Gamma(B)$}.
\end{equation}

On the other hand, using Lemma \ref{lem:derivationsAndBrackets}, one finds that the bracket $[\cdot,\cdot]_{\tau(X)}$ associated to derivation $\tau(X)=\onabla_X \D$ is given by
\begin{equation}
    \label{eq:bracket:tableau}
    [b_1,b_2]_{\tau(X)}=\onabla_X([b_1,b_2]),\quad\text{ for }b_1,b_2\in\Gamma_{(B,\onabla)},
\end{equation}
where $[\cdot,\cdot]$ is the bracket associated to $\D$.

\subsection{Torsion} \label{sec:Torsion}
We will now discuss the first order obstructions to the existence of realizations of a relative algebroid. In the following discussion we fix a relative algebroid $(B, \overline{\nabla},\D)$ over $(M, \FF)$, and we denote by $[\cdot,\cdot]$ and $\rho$ the corresponding bracket and anchor. Recall from Lemma \ref{lem:structureofRelativeDerivations} that there is an exact sequence
\[
\xymatrix{
    0 \ar[r] & \Hom(B, \FF) \ar[r] & \DD^1_{B} \ar[r]^---{\Pi} & \DD^1_{(B,\overline{\nabla})} \ar[r] & 0.}
\]
A \textbf{pointwise lift} or \textbf{completion} of $\D$ at $m\in M$ is an element $\tilde{\D}_m\in (\DD^1_B)_m$ 
such that $\Pi(\tilde{\D}_m) = \D_m$. We let 
\[
    L:=\{ \tilde{\D}_m \in \DD^1_B  \ | \ \Pi(\tilde{\D}_m) = \D_m \text{ for some $m\in M$ } \} = \Pi^{-1}(\D)
\]
be the space of completions of $\D$, with projection $p_1\:L \to M$, $\tilde{\D}_m\mapsto m$. It is an affine bundle modeled on $\Hom(B, \FF)$. The foliation $\FF$ pulls back to a foliation $p_1^*\FF$ on $L$, and the bundle $p_1^*B$ inherits a flat $p_1^*\FF$-connection $p_1^*\overline{\nabla}$ from $(B, \overline{\nabla})$. 

\begin{remark}
\label{rem:pointwise:lift:bracket}
        It follows from  Lemma \ref{lem:structureofRelativeDerivations} that a pointwise lift $\tilde{\D}_m$ of $\D$ is completely determined by a pointwise lift $\tilde{\rho}_m\: B_m\to T_m M$ of the symbol $\rho$. In terms of the bracket $[\cdot,\cdot]:\Gamma_{(B,\onabla)}\times\Gamma_{(B,\onabla)}\to \Gamma(B)$, this means that such a lift determines a unique extension to a bracket $[\cdot,\cdot]_{\tilde{\rho}_m}:\Gamma(B)\times\Gamma(B)\to B_m$, defined on all sections by requiring
        \[ [b_1,fb_2]_{\tilde{\rho}_m}=f(m)[b_1,b_2](m)+\langle\tilde{\rho}(b_1),\d_m f\rangle b_2(m),\]
        for any $b_1,b_2\in \Gamma_{(B,\onabla)}$ and $f\in C^\infty(M)$.  
\end{remark}

For the next definition we recall that, according to Example \ref{ex:PointwiseDerivations}, an element $\tilde{\D}_m\in L$ can be regarded as a derivation $\tilde{\D}\: \Omega^\bullet(B) \to  \wedge^{\bullet + 1} B^*_m$. 

\begin{definition}
    The \textbf{torsion} of $(B, \overline{\nabla},\D)$ is the 2-derivation $T\in \Gamma(\DD^2_{(p_1^*B, p_1^*\overline{\nabla})})$ relative to $p_1^*\FF$ defined by
    \[
        T(p^*_1\alpha)\big\vert_{\tilde{\D}_m} :=  \tilde{\D}_m\left(\D\alpha\right),\quad \text{for $p_1^*\alpha\in \Omega^\bullet_{(p_1^*B, \overline{\nabla})} \cong p_1^*\Omega^\bullet_{(B, \overline{\nabla})}$.}
    \]
    The symbol $\sigma(T)$ is called the \textbf{symbol torsion} of $\D$.
\end{definition}

It follows from Lemma \ref{lem:square} that the torsion $T$ has associated 2-bracket given by
\begin{equation}
    \label{eq:torsion:bracket}
    [p^*_1b_1,p^*_1b_2,p^*_1b_3]|_{\tilde{\D}_m}=[[b_1,b_2],b_3]_{\tilde{\rho}_m}+[[b_2,b_3],b_1]_{\tilde{\rho}_m}+[[b_3,b_1],b_2]_{\tilde{\rho}_m},
\end{equation} 
for $b_1, b_2,b_3\in\Gamma_{(B,\onabla)}$, where $\tilde{\rho}_m$ and $[\cdot,\cdot]_{\tilde{\rho}_m}$ denote the symbol and the bracket of $\tilde{\D}_m$, as in Remark \ref{rem:pointwise:lift:bracket}. The symbol torsion, in turn, is given by
\begin{equation}
    \label{eq:torsion:symbol}
        \sigma(T)(p^*b_1, p^*b_2)|_{\tilde{\D}_m}(f) = \tilde{\rho}_m(b_1)\left(\rho(b_2)(f)\right) - \tilde{\rho}_m(b_2)\left(\rho(b_1)(f)\right) - \rho_m([b_1,b_2])(f),
\end{equation}
for each $b_1, b_2\in \Gamma_{(B, \overline{\nabla})}$ and all $f\in C^\infty_{\bas}$.

Let $(P, r, \theta)$ be a realization of $(B, \overline{\nabla},\D)$. For each $p\in P$, it determines the lift of $\rho$ at $r(p)$ given by
\begin{equation}
    \label{eq:symbol:derivation:realization}
    \tilde{\rho}_{r(p)}:=\d_p r\circ \theta_p\:B_{r(p)}\to T_{r(p)} M.
\end{equation}
Therefore, it determines also a lift $\tilde{\D}_{r(p)}$ of $\D$.

\begin{proposition}[Existence of Realizations: first order necessary condition]\label{prop:NecessaryTorsion}
    Let $(P, r, \theta)$ be a realization of $(B, \overline{\nabla}, \D)$. Then, for each $p\in P$ we have
    \[
       T\big\vert_{\tilde{\D}_{r(p)}} = 0.  
    \]
\end{proposition}

\begin{proof}
    By definition, $\d \circ\theta^* = \theta^*\circ \D$ on $\Omega^\bullet _{(B,\overline{\nabla})}$, so the derivation
\[
    \left(\theta^{-1}_p\right)^* \circ\left(\d\circ \theta^*\right)|_p \: \Omega^\bullet(B)\to \wedge^{\bullet + 1} T^*_p P \to \wedge^{\bullet  + 1} B^*_{r(p)}
\]
is a lift of $\D$ at $r(p)$. Since its symbol is \eqref{eq:symbol:derivation:realization}, it must coincide with $\tilde{\D}_{r(p)}$. The proposition now follows from $\d^2=0$.
\end{proof}

Recall that any two pointwise lifts of $\D_m$ differ by an element $\xi\in \Hom(B_m, \FF_m)$, where $\xi$ acts as a derivation as described in the Lemma \ref{lem:structureofRelativeDerivations}. {See Remark \ref{rk:TableauxMapSpencerDifferential} for the definition of the Spencer differential $\delta_\tau$ of the tableau map $\tau$.}

\begin{proposition}\label{prop:torsionDifferenceSpencerDifferential}
    Let $\tilde{\D}_m$ be a pointwise lift of $\D$ at $m$ and $\xi\in \Hom(B_m, \FF_m)$. Then
    \[
        T\big\vert_{\tilde{\D}_m + \xi} - T\big\vert_{\tilde{\D}_m} = \delta_\tau \xi,\quad \text{ in }\Omega^\bullet_{(p_1^*B, \overline{\nabla})} \cong p_1^* \Omega^\bullet_{(B, \overline{\nabla})}.
    \]
\end{proposition}

\begin{proof}
    Note that $\overline{\nabla}_X (\D\alpha) = (\overline{\nabla}_X \D)\alpha = \tau(X)\alpha$ for $\alpha\in \Omega^\bullet_{(B, \overline{\nabla})}$. By Lemma \ref{lem:structureofRelativeDerivations}, we find that for $\beta\in \Omega^1_{(B, \overline{\nabla})}$ we have
        \begin{align*}
        \left(T(\beta)\big\vert_{\tilde{\D} + \xi}- T(\beta)\big\vert_{\tilde{\D}}\right)(b_1, b_2, b_3) &= \iota(\xi)(\D\beta)(b_1, b_2, b_3)\\
        &=\overline{\nabla}_{\xi(b_1)}(\D\beta)(b_2, b_3) + \mbox{c.p.}\\
        &= \tau(\xi(b_1))\beta(b_2, b_3) + \mbox{c.p.}\\
        &= \left(\delta_\tau\xi\right)(\beta)(b_1, b_2, b_3).
    \end{align*}
    Similarly, for $f\in C^\infty_{\bas}$, we find
    \begin{align*}
        \left(T(f)\big\vert_{\tilde{\D} + \xi} - T(f)\big\vert_{\tilde{\D}}\right)(b_1, b_2) &= \overline{\nabla}_{\xi(b_1)}(\D f)(b_2) - \overline{\nabla}_{\xi(b_2)} (\D f)(b_1)\\
        &=\left(\delta_\tau \xi\right)(f)(b_1, b_2), 
    \end{align*}
    which completes the proof.
\end{proof}

This justifies the following definition.

\begin{definition}
    The \textbf{intrinsic torsion} of a relative algebroid $(B, \overline{\nabla}, \D)$ is the section $\Tor\in \Gamma(H^{-1, 2}(\tau))$ defined by
    \[
        \Tor_m := [T\big\vert_{\tilde{\D}_m}],
    \]
     for any $\tilde{\D}_m\in L_m$.
\end{definition} 

{
\begin{remark}
    In the context of Section \ref{sec:TableauxOfDerivations}, the torsion class lives in the Spencer cohomology of the Spencer complex (Equations \eqref{eq:SpencerComplex}) of a tableau map (Remark \ref{rk:TableauxMapSpencerDifferential}) of degree-1 derivations relative to a foliation (Remark \ref{rk:TableauxOfDerivationsRelativeToFoliations}). Thus, in Definition \ref{def:SpencerCohomology}, we have $m=-1, k=1$, and $l=2$, and we are looking at the cohomology of the complex
    \[
        \xymatrix{
            0 \ar[r] & \tau^{(1)} \ar@{^{(}->}[r] & \Hom(B, \FF) \ar[r]^-{\delta_{\tau}} & \DD^{2}_{(B, \onabla)} \ar[r] & 0.
        }
    \]

\end{remark}
}

\begin{corollary}
    If a relative algebroid $(B, \overline{\nabla}, \D)$ admits a realization through $m$ then $\Tor_m=0$.
\end{corollary}

\subsection{Curvature}

Suppose we are given a section $\tilde{\D}\in \Gamma(L)$. In this case, $\tilde{\D}$  is a completion of the derivation $\D$ which is an actual 1-derivation on $B$, so $\tilde{\D}^2$ defines a 2-derivation on $B$. Recalling again from Lemma \ref{lem:structureofRelativeDerivations} that there is a short exact sequence
\[  
    \xymatrix{0 \ar[r] & \Hom(\wedge^2 B, \FF)\ar[r] & \DD^2_B \ar[r]^---{\Pi} & \DD^2_{(B, \overline{\nabla})} \ar[r] & 0}.
\]
We conclude that:

\begin{proposition}\label{prop:torsionOfSection}
    Given a section $\tilde{\D}\in \Gamma(L)$, the torsion of $\D$ along $\tilde{\D}$ is 
    \[
        T|_{\tilde{\D}} = \Pi(\tilde{\D}^2).
    \]
\end{proposition}

If $T$ vanishes identically on the image of $\tilde{\D}$, we call $\tilde{\D}$ a \textbf{torsionless lift} of $\D$. In that case, according to the previous short exact sequence, the derivation $\tilde{\D}^2$ is a section of $\Hom(\wedge^2B, \FF)$.

\begin{definition}
    Let $\tilde{\D}\in \Gamma(M^{(1)})$ be a torsionless lift of $\D$. The \textbf{curvature of $\tilde{\D}$} is the section
    \[
        \kappa_{\tilde{\D}} = \tilde{\D}^2\in \Gamma(\Hom(\wedge^2B, \FF)) \subset \Gamma(\DD^2_B).
    \]
\end{definition}    

\begin{remark}
    \[
    \kappa_{{\tilde{\D}}}(b_1, b_2) = [\tilde{\rho}(b_1), \tilde{\rho}(b_2)] - \tilde{\rho}([b_1, b_2]_{\tilde{\D}}),
    \]
    for $b_1, b_2\in \Gamma(B)$. 
    Note that the curvature depends pointwise on the first jet of the section $\tilde{\D}$ and that for any local sections $b_1,b_2\in\Gamma_B$ and $f\in C^\infty_\bas$ one has
    \[ \kappa_{{\tilde{\D}}}(b_1, b_2)(f)=0.\]
    Moreover, the fact that $\tilde{\D}$ is a torsionless lift of $\D$ implies that for any local \emph{flat} sections $b_1,b_2,b_3\in\Gamma_{(B,\onabla)}$ one also has
    \begin{equation}
    \label{eq:curvature:flat:sections}[[b_1,b_2]_{\tilde{\D}},b_3]_{\tilde{\D}}+[[b_2,b_3]_{\tilde{\D}},b_1]_{\tilde{\D}}+[[b_3,b_1]_{\tilde{\D}},b_2]_{\tilde{\D}}=0.
    \end{equation}
\end{remark}

\begin{lemma}
    The curvature $\kappa_{\tilde{\D}}$ is closed in the Spencer complex of $\tau$:
    \[
        \delta_\tau \kappa_{\tilde{\D}} = 0.
    \]
\end{lemma}

\begin{proof}
    Since the curvature is uniquely determined by its symbol, it is enough to show that the Spencer differential of the symbol tableau vanishes (Section \ref{sec:TableauxRelativeAlgebroids}). Since the result is a $C^\infty(M)$-linear, it is enough to show that it vanishes  on parallel sections. For $b_1, b_2, b_3\in \Gamma_{(B, \overline{\nabla})}$ we find, combining \eqref{eq:SpencerDifferentialFunctions} and \eqref{eq:symbol:tableau},
    \begin{align*}
        \delta_{\sigma(\tau)} \sigma\left(\kappa_{\tilde{\D}}\right) (b_1, b_2, b_3) &= \overline{\nabla}^{\mathrm{Bott}}_{\kappa_{\tilde{\D}}(b_1, b_2)} \rho(b_3) + \mathrm{c.p.}\\
        &= \Pi\left(\left[ \kappa_{\tilde{\D}}(b_1, b_2), \tilde{\rho}(b_3)\right] \right) + \mathrm{c.p.}\\
        &= - \Pi\left(\left[\tilde{\rho}\left([b_1, b_2]_{\tilde{\D}}\right), \tilde{\rho}(b_3) \right] \right) + \mathrm{c.p.}\\
        &= - \Pi\left(\kappa_{\tilde{\D}}\left([b_1, b_2]_{\tilde{\D}},b_3\right)\right)-\Pi\circ \tilde{\rho}\left([[ b_1, b_2]_{\tilde{\D}}, b_3]_{\tilde{\D}} \right) + \mathrm{c.p.} \\
        &=-\rho([[b_1,b_2]_{\tilde{\D}},b_3]_{\tilde{\D}}+[[b_2,b_3]_{\tilde{\D}},b_1]_{\tilde{\D}}+[[b_3,b_1]_{\tilde{\D}},b_2]_{\tilde{\D}})=0
    \end{align*}
    where $\Pi\: TM\to \nu(\FF)$ is the projection and the last identity follows from \eqref{eq:curvature:flat:sections}.
\end{proof}

\begin{definition}
    The \textbf{intrinsic curvature} of the relative algebroid $(B, \overline{\nabla}, \D)$ is the section $\Curv\in \Gamma\left( p_1^* H^{0, 2}(\tau)\right)$ defined along a torsionless lift $\tilde{\D}\in\Gamma(M^{(1)})$ as
    \[
        \Curv\big\vert_{\tilde{\D}}:= [\kappa_{\tilde{\D}}].
    \]
\end{definition}

\begin{remark}
    The intrinsic curvature $\Curv$ depends affinely on the values of $\tilde{\D}$ in $M^{(1)}$, and not on its first derivatives, in contrast to $\kappa_{\tilde{\D}}$, which  depends on the first jet of $\tilde{\D}$.     
\end{remark}

\section{Prolongation and integrability}\label{sec:prolongation}
We shall now attempt to complete a relative algebroid to a true Lie algebroid.

\subsection{Prolongations}

Let $(B, \overline{\nabla},\D)$ be an algebroid relative to a foliation $\FF$ on $M$. There is no direct way to make sense of ``$\D^{2}=0$", because $\D$ is only defined on flat forms in $\Omega^\bullet_{(B, \overline{\nabla})}$. To remedy this issue we consider extensions of $\D$ to $\Omega^\bullet_B$ as in Section \ref{sec:morphisms:extensions}.

\begin{definition}\label{def:prolongation}
    A \textbf{prolongation} of $(B, \overline{\nabla},\D)$ is an algebroid $(B, p_1, \D_1)$ relative to a submersion $p_1\: M_1 \to M$ satisfying
    \begin{enumerate}[(i)]
        \item (Extension) $\D_1 \alpha = (p_1)^* \D\alpha$ for $\alpha\in \Omega^\bullet_{(B, \overline{\nabla})}$, and
        \item (Completion) $\D_1 \circ \D = 0$.
    \end{enumerate}
    In this context, we write $B_1 :=p_1^*B$ for the pullback bundle.
\end{definition}    

If $(A, p, D)$ is an algebroid relative to a submersion $p\: M\to N$, we like to graphically depict a prolongation as
\[
\xymatrix@C=40pt{
    B_1 \ar[r] \ar[d] & B \ar[d] \ar[r] \ar@{-->}@/^1pc/[l]^{\D_1} & A \ar[d] \ar@{-->}@/^1pc/[l]^{\D} \\
    M_1 \ar[r]_{p_1} & M  \ar[r]_{p} & N.}
\]

\begin{remark}
    Note that for a relative algebroid the anchor and the bracket do not satisfy integrability conditions. For a prolongation, condition (ii) imposes a set of integrability conditions. By Lemma \ref{lem:square}, in terms of the brackets of $\D$ and $\D_1$, this condition amounts to the Jacobi type identity
    \[ 
    [b_1,[b_2,b_3]]_1+[b_2,[b_3,b_1]]_1+[b_3,[b_1,b_2]]_1=0,
    \]
    together with the fact that the anchor almost preserves brackets:
    \[ 
    \rho_1(p^*[b_1,b_2])(f)=\rho_1(p^*b_1)\big(\rho_0(b_2)(f)\big)-\rho_1(p^*b_2)\big(\rho_0(b_1)(f)\big),\]
    for any sections $b_1,b_2\in\Gamma_{(B,\onabla)}$ and $f\in C^\infty_\bas$.
\end{remark}

Prolongations may or may not exist, and their existence is contingent on vanishing of the intrinsic torsion. Thus, the space of torsionless lifts
\[
    p_1\: M^{(1)}\to M, \quad \mbox{where } M^{(1)}= \{ \tilde{\D}_m\in L\: T\big\vert_{\tilde{\D}_m} = 0\},
\]
plays a special role. Actually, if this space is smooth then it yields a canonical prolongation!

\begin{proposition}
\label{prop:1st:prolongation}
    Let $(B, \overline{\nabla},\D)$ be an algebroid relative to a foliation $(M,\FF)$ and assume $M^{(1)}$ has a smooth structure such that $p_1\: M^{(1)}\to M$ is a submersion. Then the relative algebroid $(B,p_1,\D^{(1)})$, where $\D^{(1)}$ is defined by
    \[
    (\D^{(1)} \alpha)\big\vert_{\tilde{\D}_m} := \tilde{\D}_m\alpha, \quad \mbox{for $\alpha\in \Omega^\bullet_B$ and $\tilde{\D}_m\in M^{(1)}$,}
    \]
    is a prolongation of $(B, \overline{\nabla},\D)$.
\end{proposition}

\begin{remark}
    We will see later, in Example \ref{example:UniversalProlongation}, that the prolongation $M^{(1)}$ is \emph{universal}: any other prolongation must factor through $M^{(1)}$. 
\end{remark}

\begin{proof}
    Since $p_1$ is a submersion, one checks immediately that for any $\al\in\Omega^\bullet_B$, $\D^{(1)}\al$ is a smooth form. The derivation property then follows from the fact that each $\tilde{\D}_m\in M^{(1)}$ is a derivation relative to the inclusion. From the definition of $\D^{(1)}$ it follows that 
    \[ \D^{(1)} \alpha = (p_1)^* \D\alpha,\quad \text{for $\alpha\in \Omega^\bullet_{(B, \overline{\nabla})}$}. \]
    Finally, since every $\tilde{\D}\in M^{(1)}$ is a torsionless lift of $\D$, it follows that $\D^{(1)}\circ \D\big\vert_{\tilde{\D}}=T_{\tilde{\D}}=0$. Hence, $(B,p_1,\D^{(1)})$ is a prolongation of $(B, \overline{\nabla},\D)$. 
\end{proof}

\begin{example}[1st prolongation in coordinates]\label{ex:prolongationInCoordinates}
Let us assume that we have fixed local coordinates $(U,(x^\mu, y^\varrho))$ and a frame $\{e_i\}$ with dual coframe $\{\theta^i\}$, as in Examples \ref {ex:relativeAlgebroidsInCoordinates} and \ref{ex:relativeAlgebroidsInCoordinates:2}. The derivation $\D$ is determined by
\begin{equation}
    \label{eq:derivation:D0}
            \begin{cases}
            \D \theta^i = -\frac{1}{2}c^{i}_{jk}(x, y)\, \theta^j\wedge\theta^k,\\
            \D x^\mu = F^\mu_i(x, y)\, \theta^i,
        \end{cases}
\end{equation}
for some functions $c^{i}_{jk}, F^\mu_i\in C^\infty(M)$. The completions of the anchor take the form
\[ \tilde{\rho}(e_i)=F^\mu_i(x, y)\partial_{x^\mu}+u_i^\varrho\partial_{y^\varrho}, \]
where $u_i^\varrho$ should be thought of as coordinates on the fibers of $L\to M$. The corresponding completion of the derivation $\tilde{\D}$ is determined by $\tilde{\D} y^\varrho = u^\varrho_i \theta^i$.

The first prolongation space $M^{(1)}$ consists of points $(x^\mu,y^\varrho,u_j^\varsigma)$ for which $\tilde{\D}\circ \D =0$. This yields a system of equations
\[
    \tilde{\D}\D x^\mu = 0, \qquad \tilde{\D}\D \theta^i =0.
\]
The first set of equations corresponds to the vanishing of the symbol torsion associated with the extension -- see \eqref{eq:torsion:symbol}:
\[ \tilde{\rho}(e_i)(\rho(e_j)(x^\mu)-\tilde{\rho}(e_j)(\rho(e_i)(x^\mu)-\rho([e_i,e_j])(x^\mu)=0,\]
while the second set of equations amounts to the vanishing of the torsion on the frame $\{e_i\}$ -- see \eqref{eq:torsion:bracket}:
\[ [[e_1,e_2],e_3]_{\tilde{\rho}_m}+[[e_2,e_3],e_1]_{\tilde{\rho}_m}+[[e_3,e_1],e_2]_{\tilde{\rho}_m}=0.
\]
Together they express the vanishing of the total torsion, and they give rise to the system of equations
\begin{align*}
    \frac{\partial F^\mu_i}{\partial y^\varrho}u_j^\varrho -\frac{\partial F^\mu_j}{\partial y^\varrho}u_i^\varrho &= F_i^\nu \frac{\partial F^\mu_j}{\partial x^\nu}-F_j^\nu \frac{\partial F^\mu_i}{\partial x^\nu}-
    c^k_{ij}F_k^\mu,\\
    \frac{\del c^i_{jk}}{\del y^\varrho}u^\varrho_l + \frac{\del c^i_{kl}}{\del y^\varrho} u^\varrho_j + \frac{\del c^i_{lj}}{\del y^\varrho} u^\varrho_k &=
    c^i_{mj}c^m_{kl} +c^i_{mk}c^m_{lj} + c^i_{ml}c^m_{jk} \\
    &\qquad - F^\mu_j\frac{\del c^i_{kl}}{\del x^\mu} - F^\mu_k \frac{\del c^i_{lj}}{\del x^\mu} - F^\mu_{l} \frac{\del c^i_{jk}}{\del x^\mu}.
\end{align*}
Assuming that $M^{(1)}$ is smooth, these equations will determine a subset of the variables $u_i^\alpha$ as functions of $x^\mu$, $y^\varrho$ and the remaining variables $u_i^\varrho$, call them $v^\varphi$. Hence, we have a set of local coordinates $(x^\mu,y^\varrho,v^\varphi)$ for $M^{(1)}$ and the derivation $\D^{(1)}$ is then defined by \eqref{eq:derivation:D0} together with
\[ \D^{(1)} y^\varrho=u^\varrho_i(x^\mu,y^\varrho,v^\varphi)\, \theta^i. \]
\end{example}

In what follows, given a relative algebroid it will be convenient to identify the tableau of its 1st prolongation with the 1st prolongation of its tableau. The precise identification is as follows.

\begin{lemma}\label{lem:tableauOfProlongation}
    Let $(B, \overline{\nabla},\D)$ be a relative algebroid. The tableau of its first prolongation is canonically isomorphic to $p_1^*\tau^{(1)}\subset p_1^*\Hom(B, \FF)$ (as a tableau), where $\tau^{(1)}$ is the first prolongation of the tableau $\tau$ of $(B, \overline{\nabla},\D)$. 
\end{lemma}

\begin{proof}
Notice that we have the following:
\begin{enumerate}[(a)]
    \item The 1st prolongation $(B, p_1, \D^{(1)})$ has a fiberwise injective tableau map
    \[
        \ker\d p_1\to p^*_1\DD^1_B, \quad X\mapsto \onabla_X \D^{(1)};
    \]
    \item The 1st prolongation of $\tau$ is the classical tableau
    \[ \tau^{(1)}:=\ker(\delta\: \Hom(B,\FF)\to \DD^2_{(B,\onabla)});\]
\end{enumerate}
Now, the bundle $p_1\:L\to M$ is an affine bundle modeled on $\Hom(B,\FF)$, and by Proposition \ref{prop:torsionDifferenceSpencerDifferential}, the restriction $p_1\:M^{(1)}\to M$ is an affine bundle modeled on $\tau^{(1)}$. It follows that we have an isomorphism of tableaux
\[ \ker\d p_1\simeq p^*_1\tau^{(1)}. \qedhere \]
\end{proof}

\subsection{Integrability} We now consider the problem of existence of prolongations and integrability. 

\begin{definition}
    Given a relative algebroid $(B, \overline{\nabla},\D)$:
    \begin{itemize}
        \item its \textbf{first prolongation} is the relative algebroid $(B, p_1, \D^{(1)})$ given by Proposition \ref{prop:1st:prolongation}, provided it exists;
        \item its \textbf{$k$-th prolongation} is defined iteratively through
        \[
        (B^{(k-1)}, p_k, \D^{(k)}):= (B^{(k-2)}, p_{k-1}, \D^{(k-1)} )^{(1)},
        \]
        provided it exists.
    \end{itemize}
    The relative algebroid $(B, \overline{\nabla},\D)$ is called \textbf{$k$-integrable} if all the prolongations up to and including $k$ exist, and \textbf{formally integrable} if it is $k$-integrable for all $k\in \NN$. If $M^{(k)}=M^{(k+1)}$ for some $k$, then we say that the relative algebroid is of \textbf{finite type}, otherwise we say that it is of \textbf{infinite type}.
\end{definition}

Our next result shows that the curvature is precisely the second order obstruction to integrability of a relative algebroid. In order to state it, note that -- see also Example \ref{ex:normalTableauAsGeneralizedTableau}:
    \begin{enumerate}[(a)]
        \item If we view $\tau^{(1)}$ as a classical tableau, then the cohomology group $H^{-1, 2}(\tau^{(1)})$ is isomorphic to $\Hom(\wedge^2B, \FF)/ \im \delta$;
        \item If we view $\tau^{(1)}$ as a tableau of derivations, then the cohomology group $H^{-1, 2}(\tau^{(1)})$ is isomorphic to $\DD^2_{B}/ \im \delta$. This is where the torsion class of the first prolongation lives.
        \item The curvature class $\Curv$ of $(B, \overline{\nabla},\D)$ lives in $H^{0, 2}(\tau) = \ker \delta/ \im \delta$.
    \end{enumerate}
    By Lemma \ref{lem:structureofRelativeDerivations}, we have  an inclusion 
    \[
        \xymatrix{
    \Hom(\wedge^2B, \FF)\, \ar@{^{(}->}[r]& \DD^2_{B}}
    \]
    which sends $\im \delta$ isomorphically onto $\im \delta$, and hence induces an inclusion in Spencer cohomology
        \begin{equation}
        \label{eq:comparing:torsion:curvature}
                \xymatrix{
    H^{0, 2}(\tau)\, \ar@{^{(}->}[r]& H^{-1, 2}(\tau^{(1)})}.
    \end{equation}
We can state the result about 2nd order obstructions as follows.

\begin{theorem}[Fundamental Theorem of Prolongation]\label{thm:fundamentalTheoremOfProlongation}
    Let $(B, \overline{\nabla},\D)$ be a 1-integrable relative algebroid. Under the inclusion map \eqref{eq:comparing:torsion:curvature} the torsion class of the first prolongation coincides with the curvature class of $(B, \overline{\nabla},\D)$. 
\end{theorem}

\begin{remark}\label{rk:extensionProblem}
    The \textbf{extension problem} for relative algebroids asks: given a relative algebroid $(B, \overline{\nabla},\D)$, is there an extension $\tilde{\D}$ of $\D$ to $\Omega^\bullet_B$ such that $\tilde{\D}^2 =0$? In other words, can one complete the relative algebroid to a Lie algebroid? The curvature serves as an obstruction to this extension problem.

    On the other hand, we will see later (cf.~Proposition \ref{prop:realizationsOfProlongations}) that realizations of a relative algebroid are in a 1:1 correspondence with realizations of its first prolongation. Thus, by Proposition \ref{prop:NecessaryTorsion}, the torsion of the first prolongation is an obstruction to the existence of realizations. 
    
    Therefore, the Fundamental Theorem of Prolongation implies that, at the formal level, the extension problem is equivalent to the realization problem.
\end{remark}

\begin{proof}[Proof of Theorem \ref{thm:fundamentalTheoremOfProlongation}] 
    The statement depends only pointwise on elements in $M^{(1)}$. So, in a neighborhood of a point in $M^{(1)}$, we can assume that we have a flat Ehresmann connection, which we view as bundle map $h\: p_1^*TM\to TM^{(1)}$ satisfying
    \[ \d p_1 (h(\tilde{\D}_m,v))=v, \quad \text{for all }v\in T_m M.\]
    The connection $h$ induces a splitting of the map $\DD^1_{B^{(1)}}\to \DD^1_{p_1^*}\cong p_1^*\DD^{1}_B$ compatible with the symbol exact sequences
    \[
        \xymatrix{
            0 \ar[r]  & p_1^*\Hom(\wedge^2B, B) \ar[d]_{ \cong} \ar[r]  & p_1^*\DD^1_B \ar[d]^{h_*} \ar[r]  & p^*_1\Hom(B, TM) \ar[d]^{h_*} \ar[r]  & 0 \\
            0 \ar[r]  & \Hom(\wedge^2B^{(1)}, B^{(1)}) \ar[r] & \DD^1_{B^{(1)}} \ar[r]                & \Hom(B^{(1)}, TM^{(1)}) \ar[r] & 0.}
    \]
    For a point in $M^{(1)}$, let $\sigma\: M\to M^{(1)}$ be the unique local flat section through that point, so that
    \begin{equation}
        \label{eq:flat:section}
        h(\tilde{\D}_m,v)=\d_m\sigma(v).
    \end{equation} 
    We denote by $\tilde{\D}_\sigma$ the derivation of $\Omega^\bullet_{B}$ determined by $\sigma$, so
    \[ \tilde{\D}_\sigma|_m:=\sigma(m). \]
    (it is convenient to keep the distinction between $\sigma$ as a section and $\sigma$ as a derivation). The theorem will follow from the identity:
    \begin{equation}\label{eq:fundamentalTheoremOfProlongation}
        \sigma^*\circ (h_*\D^{(1)}\circ \D^{(1)}) = \tilde{\D}_\sigma^2.
    \end{equation}
    Note that $h_*\D^{(1)}$ is a derivation in $\DD^1_{B^{(1)}}$ that extends $\D^{(1)}$, so $h_*\D^{(1)}\circ \D^{(1)}$ is a 2-derivation in $\DD^2_{p^*_1}\cong p_1^*\DD^2_B$ (see Section \ref{sec:morphisms:extensions}). Precomposing with $\sigma^*$ is the same as pulling back this 2-derivation to $\sigma^*p_1^*\DD^2_B \cong \DD^2_B$, so the equation makes sense.

    Assume \eqref{eq:fundamentalTheoremOfProlongation} holds. Passing to  Spencer cohomology, the left hand side of this equation gives
    \[ [\sigma^*\circ(h_*\D^{(1)} \circ \D^{(1)})]=\Tor^{(1)}, \] 
    i.e., the torsion class of $M^{(1)}$. On the other hand, since $\tilde{\D}_\sigma$ is torsionless, the right hand side of \eqref{eq:fundamentalTheoremOfProlongation} gives the curvature class
    \[ [\tilde{\D}^2_\sigma] = \Curv, \] 
    so the theorem follows.
    
    It remains to prove that equation \eqref{eq:fundamentalTheoremOfProlongation} holds. For that, we only need to observe that both sides act in the same way on $\Omega^1_{(B, \overline{\nabla})}$, and that both derivations have the same symbol precisely because $h$ is flat, i.e, because of \eqref{eq:flat:section}. 
\end{proof}

For PDEs, Goldschmidt formulated a criterion for when a PDE is formally integrable (\cite[Thm 8.1]{Goldschmidt1967}). The following result is an analogue (or rather, extension) of that result for relative algebroids.

\begin{theorem}[Formal integrability criterion]\label{thm:GoldschmidtsFormalIntegrabilityCriterion}
    Let $(B, \overline{\nabla},\D)$ be a relative algebroid. Suppose that 
    \begin{enumerate}[(i)]
        \item $(B, \overline{\nabla},\D)$ is 1-integrable,
        \item $H^{k, 2}(\tau) = 0$ for all $k\geq 0$,
    \end{enumerate}
    then $(B, \overline{\nabla},\D)$ is formally integrable.
\end{theorem}

For the proof we need the following two lemmas.

\begin{lemma}\label{lem:1-integrableCriterion}
    Let $(B, \overline{\nabla},\D)$ be a relative algebroid with tableau map $\tau$. Suppose that $p_1\: M^{(1)}\to M$ is surjective and that $\tau^{(1)}\to M$ has constant rank. Then $(B, \overline{\nabla},\D)$ is 1-integrable.
\end{lemma}
\begin{proof}
    Follows from Proposition 3.5 in \cite{Goldschmidt1967}.
\end{proof}

\begin{lemma}\label{lem:prolongationConstantRank}
    Let $\tau$ be a tableau bundle such that $\tau^{(1)}$ has constant rank and $H^{k, 2}(\tau) = 0$ for all $k\geq 0$. Then $\tau^{(k)}$ has constant rank for all $k\geq 0$. 
\end{lemma}
\begin{proof}
    The proof of Lemma 1.5.6 in \cite{Yudilevich2016} holds in this setting.
\end{proof} 

\begin{proof}[Proof of Theorem \ref{thm:GoldschmidtsFormalIntegrabilityCriterion}] 
    We use induction to show that the $k$-th prolongation of $(B, \overline{\nabla},\D)$ exists. By assumption, the 1st prolongation exists. So assume that $k\ge 1$ and that we already know that the $k$-th prolongation exists. We claim that
    \begin{enumerate}[(a)]
        \item $\tau^{(k+1)}=(\tau^{(k)})^{(1)}$ has constant rank, and
        \item $p_{k+1}\:M^{(k+1)}=(M^{(k)})^{(1)}\to M^{(k)}$ is surjective.
    \end{enumerate}
    Then Lemma \ref{lem:1-integrableCriterion} shows that the $k$-th prolongation is 1-integrable, so we are done.

    Item (a) follows immediately from Lemma \ref{lem:prolongationConstantRank} and the assumptions in the statement.  To prove item (b), note that by Theorem \ref{thm:fundamentalTheoremOfProlongation}, under the map
    \[ H^{0,2}(\tau^{(k-1)})\hookrightarrow H^{-1, 2} (\tau^{(k)}), \]
    the torsion class of the $k$-th prolongation takes values in $H^{0,2}(\tau^{(k-1)})$. But,
    \[ H^{0,2}(\tau^{(k-1)})\simeq H^{k, 2}(\tau), \]
    which vanishes under our assumptions. The vanishing of this class implies that the projection $p_k$ is surjective.
\end{proof}

\begin{corollary}\label{cor:involutiveIsFormallyIntegrable}
    If $(B, \overline{\nabla},\D)$ has an involutive tableau $\tau$ and vanishing torsion class, then it is formally integrable.
\end{corollary}

\begin{definition}
    A \textbf{relative Lie algebroid} is a formally integrable relative algebroid.
\end{definition}

The appearance of the term ``Lie" in the terminology is motivated by the full prolongation tower of a formally integrable relative algebroid:
\[
    \xymatrix{
    {\left(B^{(\infty)}, \D^{(\infty)}\right)\:\ldots} \ar[r] \ar@<-12pt>[d] & B^{(k)} \ar[d] \ar[r] & B^{(k-1)} \ar[r] \ar[d] \ar@{-->}@/^1pc/[l]^{\D^{(k)}} & \ldots \ar[r] & B^{(1)} \ar[r] \ar[d] & B \ar[d] \ar@{-->}@/^1pc/[l]^{\D^{(1)}}\\
    M^{(\infty)}\:\ldots \ar[r]                                            & M^{(k)} \ar[r]_{p_k}    & M^{(k-1)} \ar[r]                                                               & \ldots \ar[r] & M^{(1)} \ar[r]_{p_1}    & M}               
\]
In particular, although $\D^2$ does not make sense,
the derivation $\D^{(\infty)}$, defined on profinite sections of $B^{(\infty)}\to M^{(\infty)}$ (i.e., on those that locally factor through some $M^{(k)}$) squares to zero: 
\[ \left(\D^{(\infty)}\right)^2=0, \] 
so $(B^{(\infty)},\D^{(\infty)})$ is a ``profinite Lie algebroid of finite rank". We refer to Section \ref{sec:postlude} for a concrete example of a profinite Lie algebroid arising in this way.

\subsection{Some examples}
In this section, we discuss several simple examples that illustrate the various issues that can arise with prolongations of relative algebroids. We provide examples of relative algebroids that:
(i) are of infinite type,
(ii) are of finite type,
(iii) do not admit a prolongation, and
(iv) have a first prolongation but not a second prolongation.

While such examples already exist in the context of PDEs, the ones presented here are independent of PDE theory. These examples should also help the reader develop intuition for the general framework.

\begin{example}[Relative vector fields and control systems]\label{ex:RelativeVectorFields}
If a relative algebroid is 1-integrable and has an involutive tableau with non-zero Cartan characters, it is always of infinite type. Such examples, with the smallest rank, arise from \emph{relative vector fields}.

A \textbf{vector field relative to a submersion} $p\: M\to N$ is a section $X\in \Gamma(p^*TN)$. It gives rise to a relative algebroid
    \[
        \xymatrix{
            \underline{\RR} \ar[d] \ar[r] & \underline{\RR} \ar[d] \ar@{-->}@/^1pc/[l]^{\D_X}\\
            M \ar[r]^{p} & N
        }
    \]
where the derivation is given by
\[
\D_X\: C^\infty(N) \to  \Omega^1(\underline{\RR}),\quad \D_X (f)_m = X_m(f)\d t, \]
where $\d t$ is a basis of $\RR^*$. Conversely, any relative algebroid with vector bundle $\underline{\RR}$ gives rise to a relative vector field. In particular, the prolongation of a relative vector field is again a relative vector field.
    
For a relative vector field $X$, the tableau map of the corresponding algebroid is
    \begin{align*}
        \tau=\onabla X\:& \ker \d p \to \Hom(\RR, p^*TN)\cong p^*TN,\\
        \tau(v)(\lambda):&=\lambda\left.\frac{\d}{\d t}X_{\gamma(t)}\right|_{t=0}
    \end{align*}
    where $\gamma(t)$ is a curve in a fiber of $p$ with $\dot{\gamma}(0)=v$.
    The Spencer differential $\delta_\tau\: \Hom(\RR, \ker \d p)\to \Hom(\wedge^2\RR, p^*TN)$ vanishes identically, and therefore the tableau is involutive, with Cartan character 
    \[ s_1 = \rk (\ker \d p) = \dim M - \dim N.\] 
    A lift of $\D_X$ to $\underline{\RR}\to M$ is a vector field $\tilde{X}$ on $M$ such that $\d p(\tilde{X}) = X$. The corresponding derivation $\tilde{\D}_X$ always squares to zero, so the relative algebroid has vanishing torsion and curvature, and is therefore formally integrable.

    A local realization can be described explicitly as a curve $\gamma^M\: I\to M$ such that $\gamma^N:= p\circ \gamma^M$ satisfies the ODE
    \begin{equation}
    \label{eq:ODERelativeVectorField}
        \dot{\gamma}^N(t) = X_{\gamma^M(t)}.
    \end{equation}
    In local coordinates, if $(x^\mu, y^\varrho)$ are submersion coordinates on $M$, then the vector field looks like
    \[
        X = X^\mu(x, y)\frac{\del}{\del x^\mu},
    \]
    so the ODE \eqref{eq:ODERelativeVectorField} takes the form
    \begin{equation}
    \label{eq:ODERelativeVectorFieldsCoordinates}
        \dot{x}^\mu(t) = X^\mu(x(t), y(t)).
    \end{equation}
    This shows that locally, realizations (integral curves) are completely determined by the choice of the $\dim M-\dim N$ family of functions $y^\varrho(t)$ and the choice of an initial point $(x_0, y_0)$.
    Equation \eqref{eq:ODERelativeVectorFieldsCoordinates} is a control system (with unspecified observables), where the choice of the functions $y^\varrho(t)$ \emph{is} the control input. In particular, according to the results in \cite{HermannKerner1977}, when the vector field is real analytic, then $N$ is partitioned into a singular foliation with the property that for each point $n_0\in N$, the integral curves of $X$ through $n_0$ saturate a neighborhood of $n_0$ inside the leaf through $n_0$. 
    
    As a very particular case, consider $p\:\RR\to *$ with the zero relative vector field. A realization of the corresponding relative algebroid is just an $\RR$-valued function $x^1(t)$ defined on an interval. The prolongation tower consists of a sequence of vector fields $X_k$ relative to the projection $p_k\:\RR^k\to \RR^{k-1}$,
    \[
    \xymatrix{
    {\left(\underline{\RR}, \D_{X_\infty}\right)\:\ldots} \ar[r] \ar@<-12pt>[d] & \underline{\RR} \ar[d] \ar[r] & \underline{\RR} \ar[r] \ar[d] \ar@{-->}@/^1pc/[l]^{\D_{X_{k-1}}} & \ldots \ar[r] & \underline{\RR} \ar[r] \ar[d] & \underline{\RR} \ar[d] \ar@{-->}@/^1pc/[l]^{\D_{X_0}}\\
    \RR^{\infty}\:\ldots \ar[r]                                            & \RR^k \ar[r]_{p_k}    & \RR^{k-1} \ar[r]                                                               & \ldots \ar[r] & \RR \ar[r]_{p}    & {*}}               
    \]
    Denoting by $(x^1, \dots, x^{k})$ the coordinates on $\RR^k$, the relative vector fields $X_k$ are given by
    \[
        X_k = \sum_{i=1}^k x^{i+1}\frac{\del}{\del x^{i}}, \quad X_\infty = \sum_{i=1}^\infty x^{i+1}\frac{\del}{\del x^i}.
    \]
    The profinite vector field $X_\infty$ appears later in Section \ref{sec:postlude}.
\end{example}

\begin{example}[Finite-type relative algebroid]
Consider the submersion $p\: \RR^3\to \RR^2, (x, y, z)\mapsto (x,y)$ and the relative algebroid $(\underline{\RR}^2, p, \D)$ determined by
\[
        \begin{cases}
            \D \theta^1= \theta^1\wedge \theta^2,\\
            \D \theta^2= 0,\\
            \D x= z\theta^1,\\
            \D y= z \theta^2.
        \end{cases}
\]
To compute the prolongation, we find an extension $\tilde{\D}$ of $\D$ satisfying
    \begin{align*}
        0 &= \tilde{\D}(\D x) = \tilde{\D}z \wedge \theta^1 + z \theta^1 \wedge \theta^2\\
        0 &= \tilde{\D}(\D y) = \tilde{\D}z \wedge \theta^2.
    \end{align*}
Writing $\tilde{\D}z = \alpha \theta^1 + \beta \theta^2$, we see that we must have
\[
    \tilde{\D} z = z \theta^2.
\]
This defines a prolongation where $M^{(1)}\cong M = \RR^3$, so $\tilde{\D}$ is a derivation relative to the identity. At this point, we only know that $\tilde{\D} \circ \D = 0$, but one checks easily that $\tilde{\D}^2=0$. Hence, the first prolongation is actually a Lie algebroid and so $M^{(k)}\simeq M^{(1)}$ for all $k\ge 1$. The resulting algebroid $(\underline{\RR^2}, \tilde{\D}) \to \RR^3$ is isomorphic to an action algebroid $\mf{g}\ltimes \RR^3$, where $\mf{g}$ is the non-abelian two-dimensional Lie algebra.
\end{example}

\begin{example}[1-Integrable relative algebroid that is not 2-integrable]
Let us modify the derivation in the previous example by setting
\[
        \begin{cases}
            \D \theta^1= \theta^1\wedge \theta^2,\\
            \D \theta^2= 0,\\
            \D x= z\theta^1,\\
            \D y= z \theta^2-\theta^1.
        \end{cases}
\]
Proceeding as before, the first prolongation $\tilde{\D}$ is now given by
\[
    \tilde{\D}z = \theta^1 + z \theta^2.
\]
However, in this case $\tilde{\D}^2 z = 2\theta^1\wedge \theta^2 \neq 0$, so the second prolongation does not exist.
\end{example}

\begin{example}[Torsion from tableau of derivations]
In the previous example, the equations determining the prolongation and obstructions to integrability arose solely from the symbol tableau of the relative algebroid. We now illustrate how the full tableau of derivations can give rise to torsion. Of course, this can only happen when the rank is at least three, otherwise the symbol torsion completely determines the torsion. 
One way to achieve this  is to consider relative algebroids with zero anchor. So, consider the relative algebroid $(\underline{\RR}^3, p, \D)$, with $p\: \RR\to *$ and $\D$ determined by
\begin{equation}\label{eq:exampleDerivationTorsion}
        \begin{cases}
        \D \theta^1 =  x \theta^1\wedge \theta^2,\\
        \D \theta^2 = x \theta^2 \wedge \theta^3,\\
        \D \theta^3 = x \theta^3 \wedge \theta^1,
        \end{cases}
\end{equation}
where $x$ is the coordinate on $\RR$. A lift $\tilde{\D}$ of $\D$ is determined by $\tilde{\D}x = x_1\theta^1 + x_2\theta^2 + x_3\theta^3$ and the torsion is the 2-derivation $(\tilde{\D}\circ \D)$ relative to $\RR^3\times \RR\to *$ defined by
\[
\begin{cases}
        \tilde{\D}\circ \D\theta^1 = (x_3 - x^2) \theta^1\wedge\theta^2\wedge\theta^3,\\
        \tilde{\D}\circ \D \theta^2 = (x_1 - x^2) \theta^1\wedge\theta^2\wedge\theta^3,\\
        \tilde{\D}\circ \D \theta^3 = (x_2 - x^2)\theta^1\wedge\theta^2\wedge\theta^3.
        \end{cases}
\]
We find that the first prolongation space is
\[ M^{(1)}=\{x_1=x_2=x_3=x^2\}, \]
so the first prolongation is given by
\[ \D^{(1)}x=x^2(\theta^1 + \theta^2 + \theta^3). \]
Note that the curvature is only zero at $x=0$:
\[
    \left(\D^{(1)}\right)^2x = x^3(\theta^1\wedge\theta^2 + \theta^2\wedge\theta^3 + \theta^3\wedge\theta^1).
\]
Hence, to obtain an actual Lie algebroid one has to restrict to the space $x=0$, in which case the relative algebroid is just the abelian Lie algebra $\RR^3$. 
\end{example}

\subsection{Realizations and prolongations} Prolongations arose as a tool for constructing and computing obstructions to the existence of realizations. The realizations of the prolongations are related to realizations of the original relative algebroid in the following manner.

\begin{proposition}\label{prop:realizationsOfProlongations}
    Let $(B, \overline{\nabla},\D)$ be a relative algebroid. Then:
    \begin{enumerate}[(i)] 
        \item If $(B_1, p_1, \D_1)$ is a prolongation, then any realization of $(B_1, p_1, \D_1)$ induces a realization of $(B, \overline{\nabla},\D)$.
        \item If $(B, \overline{\nabla},\D)$ is 1-integrable, then realizations of the canonical prolongation $(B^{(1)}, p_1, \D^{(1)})$ are in 1-1 correspondence with realizations of $(B, \overline{\nabla},\D)$.
    \end{enumerate}
\end{proposition}

\begin{proof}
To prove (i), let $(B_1, p_1, \D_1)$ be a prolongation and let $(P, r_1, \theta_1)$ be a realization of $(B_1, p_1, \D_1)$. The fact that $(P, p_1 \circ r_1, (p_1)_* \circ \theta_1)$ is a realization of $(B, \overline{\nabla},\D)$ follows by restricting to $\Omega^1_{(B, \overline{\nabla})}$ the identity valid on $\Omega^1_B$:
\[ \d \circ \left(\theta_1\right)^* = \left( \theta_1\right)^*\circ \D_1. \]

To prove (ii), let $(P, r, \theta)$ be a realization of $(B, \overline{\nabla}, \D)$. By Proposition \ref{prop:NecessaryTorsion}, we can define the map
\[
    r^{(1)}\: P\to M^{(1)}, \quad r^{(1)}(p) = \tilde{\D}_{\theta_p} :=\left( \theta^{-1}_p\right)^* \circ \left( \d \theta^*\big\vert_p \right).
\]
Clearly, $p_1\circ r^{(1)} = r$, and since $B^{(1)} = p_1^* B$ there is a natural bundle map 
\[ (\theta^{(1)}, r^{(1)})\: TP \to B^{(1)},\text{ with } (p_1)_* \circ \theta^{(1)} = \theta.\] 
Using that $\theta$ is a realization, we find that
\[
    \d \circ \left(\theta^{(1)}\right)^* \circ p_1^* = \d \circ \theta^* = \theta^* \circ \D = \left(\theta^{(1)}\right)^* \circ \D^{(1)} \circ p_1^*, \mbox{ on $\Omega^\bullet_{(B, \overline{\nabla})}$.}
\]
From our formula for $r^{(1)}$ and the tautological nature of $\D^{(1)}$ it also follows that
\[
    \d \circ \left(r^{(1)}\right)^*(f) = \left(\theta^{(1)}\right)^* \circ \D(f), \mbox{ for $f\in C^\infty(M)$,}
\]
so $(r^{(1)}, \theta^{(1)})$ is a realization of $B^{(1)}$. The construction in (i) when applied to the 1st prolongation is an inverse to this construction, so we have a 1:1 correspondence.
\end{proof}

In the analytic setting, the existence of realizations follows from the classical Cartan–K\"ahler theory; see the discussion in Section 3 of Bryant's unpublished lecture notes \cite{Bryant2014}.
\begin{theorem}
\label{thm:CartanBryantExistence}
    Let $(B, \overline{\nabla}, \DD)$ be an analytic relative Lie algebroid. Then, for each $k$ and each $x\in M^{(k)}$, there exists a realization through $x$.
\end{theorem}
\begin{proof}
    Since, by assumption, $(B, \overline{\nabla}, \DD)$ is formally integrable, induction and Proposition \ref{prop:realizationsOfProlongations} show that it is enough to prove the result for $k = 1$. Moreover, since every tableau is involutive after finitely many prolongations, we can assume that the tableau $\tau$ is involutive. But then, the realization problem for $B^{(1)}$ in local analytic coordinates (see Remark \ref{rem:realizations:MC} and Example \ref{ex:prolongationInCoordinates}) satisfies the assumptions of  \cite[Thm.~3]{Bryant2014}. The latter result shows that there exists a realization through every point in $M^{(1)}$.
\end{proof}

\subsection{Naturality}

Let $(\varphi, p)\: (B_1, \onabla^1) \to (B_2, \onabla^2)$ be a morphism of flat foliated vector bundles covering a map $p\:(M_1, \FF_1)\to (M_2, \FF_2)$, such that $\varphi$ is a fiberwise isomorphism. As we saw in Section \ref{sec:morphisms:extensions}, $\varphi$ induces a bundle map
\[
   \varphi_*\: \DD^k_{(B_1, \onabla^1)}\to \DD^k_{(B_2, \onabla^2)}.
\]
More generally, in this situation, there is a bundle map
\[
    \varphi_*\: \Hom\left(\wedge^lB_1, \DD^k_{(B_1, \onabla^1)}\right) \to \Hom\left(\wedge^lB_2, \DD^k_{(B_2, \onabla^2)}\right)
\]
that intertwines the Spencer differentials: $\delta\circ \varphi_* = \varphi_* \circ \delta$. All constructions with derivations that we discussed before behave naturally relative to these induced maps.

\begin{proposition}
\label{prop:naturalityOfRelativeAlgebroidConstructions}
    Let $(\varphi, p)\: (B_1, \onabla^1, \D_1)\to (B_2, \onabla^2, \D_2)$ be a morphism of relative algebroids, which is a fiberwise isomorphism. Then
    \begin{enumerate}[(i)]
        \item The tableau maps are $(\varphi_*,p_*)$-related: $ \varphi_*\circ\tau_1=\tau_2 \circ p_*$. In particular, $\varphi$ induces a morphism in Spencer cohomology
        \[
            \varphi_*\: H^{k,l}(\tau_1)\to H^{k,l}(\tau_2),
        \]
        that maps the torsion class of $B_1$ to the one of $B_2$;
        \item If $(P, r, \theta)$ is a realization of $(B_1, \onabla^1, \D_1)$, then $(P, p \circ r, \varphi \circ \theta)$ is a realization of $(B_2, \onabla^2, \D_2)$;
        \item If $(B_1, \onabla^1, \D_1)$ and $(B_2, \onabla^2, \D_2)$ are $k$-integrable, then $\varphi$ induces a morphism of relative algebroids $(\varphi^{(k)}, p^{(k)})\: (B_1^{(k)}, (p_1)_k, \D^{(k)}_1) \to (B_2^{(k)}, (p_2)_k, \D^{(k)}_2)$ that is fiberwise an isomorphism;
        \item If $(B_1, \onabla^1, \D_1)$ and $(B_2, \onabla^2, \D_2)$ are formally integrable, then $\varphi$ induces a morphism $(\varphi^{(\infty)}, p^{(\infty)})\: (B_1^{(\infty)}, \D^{(\infty)}_1)\to (B_2^{(\infty)}, \D^{(\infty)}_2)$ of profinite Lie algebroids. 
    \end{enumerate}
\end{proposition}

Notice that if $(B, p_1, \D_1)$ is a prolongation of a relative algebroid $(B, \onabla, \D)$, then the torsion class of $(B, p_1, \D_1)$ is in the kernel of $(p_1)_*$. Hence, the proposition implies:

\begin{corollary}
    If a relative algebroid $(B, \onabla, \D)$ admits a prolongation $(B, p_1, \D_1)$, then its torsion class must vanish.
\end{corollary}

\begin{proof}
    First, we show that $\varphi_*$ and $p_*$ intertwine the tableau maps. For that, notice that if $\alpha\in \Omega_{(B_1,\onabla^1)}$ then, for any $X\in\FF$, we have
    \[ (\onabla^1_X\D_1)\alpha = \onabla^1_X(\D_1\alpha).\]
    Using this and the definition of the tableaux, we find for $\alpha\in \Omega^\bullet_{(B_2, \onabla^2)}$:
    \begin{align*}
        (\varphi_*\tau_1(X))\alpha &= \left(\varphi_* \left(\onabla^1_X \D_1\right)\right) \alpha = \varphi_* \left( \onabla^1_X \left( \D_1 \varphi^*\alpha\right)\right)\\
        &=\varphi_*\left(\onabla^1_X \varphi^*(\D_2\alpha)\right) = \varphi_* \left( \varphi^* \onabla^2_{p_*(X)}(\D_2\alpha)\right)\\
        &= \onabla^2_{p_*(X)} (\D_2\alpha) = (\onabla^2_{p_*(X)}\D_2)(\alpha) = \tau_2(p_*(X))(\alpha).
    \end{align*}

    Since $\varphi_*$ intertwines the Spencer differentials, it induces maps on the Spencer complexes of $\tau_1$ and $\tau_2$, commuting with the differentials, and therefore $\varphi_*$ descends to the level of cohomology. To check that it relates the torsion classes, let $\tilde{\D}_1$ be a pointwise lift of $\D_1$ above $x\in M_1$. Its torsion is the 2-derivation $T^1\big\vert_{\tilde{\D}_1} = \tilde{\D}_1\circ \D_1$. Because $\D_1$ and $\D_2$ are $\varphi$-related, the derivation $\tilde{\D}_2:=\varphi_*(\tilde{\D}_1)$ is a pointwise lift of $\D_2$ above $p(x)\in M_2$. The 2-derivations $T^1\big\vert_{\tilde{\D}_1}$ and $T^2\big\vert_{\tilde{\D}_2}$ are also $\varphi_*$-related since
    \begin{align*}
        \varphi_*\left( T^1\big\vert_{\tilde{\D}_1}\right)(\alpha) &=\varphi_*(\tilde{\D}_1\circ \D_1)(\alpha) = \varphi_* \circ \tilde{\D}_1\circ \D_1\circ \varphi^*(\alpha)\\
        &=\varphi_* \circ \tilde{\D}_1 \circ \varphi^* \circ \D_2 (\alpha) = \varphi_*(\tilde{\D}_1)\circ \D_2 (\alpha)\\
        &= \tilde{\D}_2 \circ \D_2(\alpha) = T^2\big\vert_{\tilde{\D}_2}(\alpha),
    \end{align*}
    for any $\alpha\in \Omega^\bullet_{(B_2, \onabla^2)}$. Passing down to the Spencer cohomology, item (i) follows.

    Item (ii) follows by observing that, for $\alpha\in \Omega^\bullet_{(B_2, \onabla^2)}$, one has
    \begin{align*}
        (\varphi\circ \theta)^*\D_2 \alpha &= \theta^*\varphi^*\D_2\alpha = \theta^*\D_1\varphi^*\alpha = \d (\varphi\circ \theta)^*\alpha.
    \end{align*}

    In order to prove item (iii) it is enough to prove the case $k=1$. By the previous calculation, if $\tilde{\D}_1$ is a \emph{torsionless} lift of $\D_1$, then $\varphi_*(\tilde{\D}_1)$ is also a torsionless lift of $\D_2$. So $\varphi_*$ restricts to a map between the base spaces of the first prolongation $p^{(1)}:=\varphi_*\: M_1^{(1)}\to M_2^{(1)}$, satisfying
    \[ (p_2)_1\circ p^{(1)}=(p_2)_1\circ \varphi_* = p \circ (p_1)_1. \] 
    This map and $\varphi$ combine into a map between the pullbacks $B^{(1)}_i= (p_i)_1^*B_i$:
    \[ \varphi^{(1)}:=(\varphi, p^{(1)})\: B^{(1)}_1\to B^{(1)}_2. \] 
    Note that $\varphi^{(1)}$ is again a fiberwise isomorphism and also a map of relative algebroids because of the tautological nature of the derivations $\D^{(1)}_1$ and $\D_2^{(1)}$. Indeed, we find
    \begin{align*}
        \D^{(1)}_1(\varphi^*(\alpha)) \big\vert_{\tilde{\D}_1} &= \tilde{\D}_1 (\varphi^*\alpha) = \varphi^*\big(\varphi_*(\tilde{\D}_1)(\alpha)\big) \\
        &= \varphi^*\left( \D^{(1)}_2(\alpha)\big\vert_{\varphi_*(\tilde{\D}_1)} \right) = (\varphi^{(1)})^*\big(\D^{(1)}_2 (\alpha)\big)
    \end{align*}
    for any $\alpha\in \Omega^\bullet_{(B_2, \onabla^2)}$, where we used that $p^{(1)}= \varphi_*$.

    Finally, item (iv) follows from item (iii).
\end{proof}

\section{Constructions}
\label{sec:constructions}
To better understand relative algebroids, we will now develop several important constructions involving them.

\subsection{The universal relative algebroid}\label{sec:tautologicalAlgebroid}

Let $A\to N$ be any vector bundle, and $p_1\: \DD^k_A\to N$ its bundle of $k$-derivations. There is a   \textbf{tautological $k$-derivation} $\UD$ relative to the projection $p_1\: \DD^k_A\to N$, which at each point $\D_x\in \left(\DD^k_A\right)_x$ is the derivation itself:
\[
    \UD(\alpha)(\D_x) = \D_x\alpha \in \wedge^{k+l} A_x^*, \quad \text{for }\alpha\in \Omega^l(A).
\]

\begin{definition}
    The \textbf{universal relative Lie algebroid} of a vector bundle  $A\to N$ is the triple $(A,p_1,\UD)$, where $p_1\: \DD^1_A \to N$ and $\UD$ is the tautological 1-derivation. 
\end{definition}
The bundle over $\DD^1_A$ of the universal relative algebroid will be denoted $\UB:= p_1^*A$. In order to justify the use of the term ``universal'', let us introduce the following notation.

\begin{definition}\label{def:classifyingMap}
    Let $(A, p, \D)$ be an algebroid relative to a submersion $p\: M\to N$. The \textbf{classifying map} $c_{\D}$ of $(A, p, \D)$ is the composition
    \[
    \xymatrix{M \ar[r]_---{\D} \ar@/^1.5pc/[rr]^{c_{\D}}
    & p^*\DD^1_A \ar[r]_{\pr} & \DD^1_A}.
    \]
\end{definition}

\begin{proposition}
\label{prop:classifying:map}
    Every relative algebroid $(A,p,\D)$ is canonically isomorphic to the pullback of the universal relative Lie algebroid $(A,p_1, \UD)$ along its classifying map.
\end{proposition}
\begin{proof}
    Note that $p_1\circ c_D = p$, and that the pullback $c_D^*\UD$ coincides with $\D$ under the canonical identification $c^*_Dp_1^*\DD^1_A \cong p^*\DD^1_A$.
\end{proof}

\begin{example}
    When $N=\{*\}$, so $A= V$ is a vector space, we have 
    \[ \DD^1_V\cong \Hom(\wedge^2V, V)\to *. \] 
    A relative algebroid $(V,p,\D)$, for $p\:M\to \{*\}$, is the same thing as a skew-symmetric bilinear map $[\cdot,\cdot]\:V\times V\to C^\infty(M, V)$, viewed as a bracket on $V$ relative to $p$. This is determined by a map $c_{\D}\:M\to \Hom(\wedge^2V, V)$, $c_{\D}(x)(v, w) = [v, w](x)$, which is precisely the classifying map. This construction was already considered by Bryant in relation to his Theorem 4 in \cite{Bryant2014}.
\end{example}

\begin{example}
    For almost Lie algebroids, the previous proposition says that any such structure $\D$ on a fixed vector bundle $p\:A\to N$ is obtained by pulling back the tautological derivation $\UD$ along a section $c_{\D}\:N\to \DD^1_A$: $c_{\D}^*\UD=\D$.
\end{example}

\begin{example}\label{example:UniversalProlongation}
    The classifying map of the first prolongation of a 1-integrable relative algebroid $(B,\onabla, \D)$ is the inclusion $M^{(1)}\hookrightarrow \DD^1_{B}$. In fact, given an algebroid $(B, p_1, \D_1)$ relative to a submersion $p_1\: M_1\to M$ with classifying map $c_{\D_1}\:M_1\to \DD^1_B$, one has that:
    \begin{enumerate}[(i)]
        \item $(B, p_1, \D_1)$ is an \emph{extension} of $(B,\onabla, \D)$ if and only if $c_{\D_1}$ takes values in the space of pointwise lifts $L\subset \DD^1_B$;
        \item $(B, p_1, \D_1)$ is a \emph{prolongation}  of $(B,\onabla, \D)$ if and only if the image of $c_{\D_1}$ is contained in $M^{(1)}$.
    \end{enumerate}  
    In particular, among all prolongations of $(B, \onabla, \D)$, the canonical prolongation is the universal one. 
\end{example}

Next, we will study properties of the universal relative Lie algebroid. We start by justifying the use of the term ``Lie".

\begin{proposition}
    The universal relative Lie algebroid of a vector bundle $A\to N$ is formally integrable.
\end{proposition}

\begin{proof}
    The tableau of the universal algebroid is the identity map $\DD^1_A\to \DD^1_A$, so it is involutive by Proposition \ref{prop:tautologicalTableauIsInvolutive}. The torsion class vanishes automatically because $H^{-1, 2}(\DD^1_A) = 0$ by Lemma \ref{lemma:symbolAndSpencerDifferential}. Then, by Corollary \ref{cor:involutiveIsFormallyIntegrable} of Goldschmidt's Formal Integrability Criterion, the universal algebroid is formally integrable.
\end{proof}

The realizations of the universal relative Lie algebroid have a nice geometric interpretation. For that, 
given a manifold $P$ and a vector bundle $A\to N$, by an \textbf{$A$-coframe} on $P$ we mean a bundle map $\theta_A\: TP\to A$ covering a map $r_N\: P \to N$ which is fiberwise an isomorphism. These objects have appeared in the literature under the name ``generalized coframes" in the context of Dirac spinors coupled to Einstein's equations -- see \cite{PierardDeMaujouy2024}.

\begin{proposition}\label{prop:TautologicalBryantAlgebroidRealizations}
    The realizations of the universal relative Lie algebroid of $A\to N$ are in one-to-one correspondence with manifolds equipped with $A$-coframes. 
\end{proposition}

\begin{proof}
    In one direction, it is clear that a realization $(\theta,r)\:TP\to p_1^*A$ of the universal relative Lie algebroid $(A,p,\UD)$, gives rise to the $A$-frame $\theta_A\:TP\to A$ covering $r_N:=p_1\circ r$ given by
    \[ \theta_A=\pr_A\circ\, \theta. \]
    In the opposite direction, given an $A$-coframe $\theta_A\: TP\to A$ covering a map $r_N\: P \to N$, we can define a bundle map  $(\theta,r)\:TP\to p_1^*A$ which is a fiberwise isomorphism by
    \[ \begin{cases}
        r(p):=\D_{r_N(p)},\\
        \theta(v_p):=(\D_{r_N(p)},\theta_A(v_p)),
    \end{cases}
    \quad (v_p\in T_p P),\]
    where $\D_{r_N(p)}\in (\DD^1_A)_{r_N(p)}$ is given by
    \[ \D_{r_N(p)}\alpha:=(\theta_A)_*\circ \d_p\circ \theta_A^*(\alpha), \quad (\alpha\in\Omega^1(A)). \]
    This ensures that the condition
    \[ \theta^*\circ\UD=\d\circ \theta^*\circ p_1^*\]
    holds, so $(\theta,r)$ is a realization of $(A,p_1,\UD).$
    
    These constructions are inverse to each other, so the proposition follows.
    \end{proof}

Recall that a relative algebroid $(A, p, \D)$ is called standard when its tableau map $\tau\: \ker \d p\to p^*\DD^1_A$ is fiberwise injective.

\begin{proposition}\label{prop:standardClassifyingMapImmersion}
    An algebroid $(A, p, \D)$ relative to a submersion $p\: M\to N$ is standard if and only if its classifying map $c_{\D}\: M\to \DD^1_A$ is an immersion.
\end{proposition}

\begin{proof}
    Note that $p = p_{1}\circ c_{\D}$. This means that the classifying map $c_{\D}\: M\to \DD^1_A$ is an immersion if and only if its vertical derivative $\d c_{\D}\big\vert_{\ker \d p}\: \ker \d p \to \ker \d p_1$ is fiberwise injective. We claim that the tableau map $\tau\: \ker \d p\to p^*\DD^1_A$ is given by
    \begin{equation}
        \label{eq:tableau:classifying_map}
        \tau(X_m)=(m,\d c_{\D}(X_m)),\quad (X_m\in\ker \d_m p),
    \end{equation}
    so the result follows. To prove the above formula, just observe that we have, by Proposition \ref{prop:classifying:map},
    \[
    \tau(X)_m=\onabla_{X_m}\D=\onabla_{X_m}(c_{\D}^*\UD)
    =\onabla_{\d c_{\D}(X_m)}\UD=\d c_\D(X_m). \qedhere
    \]

\end{proof}

\subsection{Restriction}

Almost relative algebroids are so flexible that they can easily be restricted to subspaces. From the point of view of Cartan's realization problem, imposing restrictions on a relative algebroid structure is equivalent to adding extra conditions to the realization problem of the original relative algebroid. This will be clear from examples at the end of this section. For now, let us be precise about what we mean by ``restriction". 

In this section we fix an ambient relative algebroid $(B, \overline{\nabla},\D)$. 

\begin{definition}
    A map $q\: Q\to M$ is \textbf{invariant} for $(B, \overline{\nabla},\D)$ if the image of $\d q$ contains the image of the anchor of $\D$, i.e., if for every $x\in Q$,
    \[ \sigma(\D_{q(x)})\in\Hom(B_{q(x)}, \im (\d_x q) / \FF_{q(x)}) \subseteq \Hom(B_{q(x)}, \nu(\FF_{q(x)})).\]
\end{definition}
\begin{remark}
    In the special case of an algebroid $(A, p, \D)$ relative to a submersion $p\: M\to N$, the invariance condition on a map $q\: Q\to M$ says that
    \[
        \im \rho_{p(x)} \subseteq \im T_x(p\circ q), \quad \mbox{ for all $x\in Q$.}
    \]  
    If $(A, \D)$ is an almost Lie algebroid and $q$ is an immersion, we recover the usual notion of invariant submanifold for such an algebroid.
\end{remark}

\begin{proposition}\label{prop:restriction}
    Suppose that a map $q\: Q\to M$ satisfies the following conditions:
    \begin{enumerate}[(a)]
        \item $q$ is invariant for $\D$;
        \item $(\d_x q)^{-1}(\FF_{q(x)})$ has constant rank for all $x\in Q$.
    \end{enumerate}
    Then there exists a unique relative algebroid $(q^*B, q^*\overline{\nabla}, \D_Q)$ for which the map $q_*\: q^*B\to B$ is a morphism of relative algebroids. 
\end{proposition}

\begin{remark}    
    The morphism $q_*\:(q^*B,q^*\onabla,\D_Q)\to (B,\onabla,\D)$ is a fiberwise isomorphism, so it follows from Proposition \ref{prop:naturalityOfRelativeAlgebroidConstructions} that:
    \begin{enumerate}[(i)]
    \item realizations of $(q^*B, q^*\overline{\nabla}, \D_Q)$ yield realizations of $(B,\onabla,\D)$;
    \item there is an induced morphism at level of the Spencer cohomologies of the tableaux $q_*\:H^{k,l}(\tau_Q)\to H^{k,l}(\tau)$, which relates the torsion classes; however, this map is neither injective nor surjective in general;
    \item if both algebroids are $k$-integrable, there is an induced algebroid morphism $q_*\:((q^*B)^{(k)}, (p_Q)_k, \D_Q^{(k)})\to (B^{(k)},p_k,\D^{(k)})$;
    \item if both algebroids are formally integrable, there is an induced algebroid morphism $q_*\:((q^*B)^{\infty}, \D_Q^{\infty})\to (B^{\infty},\D^{\infty})$.
\end{enumerate}
    Note also that if $q\: Q\to M$ is \emph{transverse} to $\FF$, then both conditions (a) and (b) in the proposition are automatically satisfied.
\end{remark}

\begin{proof}
    The regularity condition (b) means that the pullback foliation $q^!\FF$ exists. It follows that one has a flat foliated vector bundle $(q^*B, q^*\overline{\nabla})$ over  $(Q, q^!\FF)$, as well as a map 
    of foliated vector bundles
    \[ q_*:=(q,\pr)\:(q^*B, q^*\overline{\nabla})\to (B,\onabla). \]
    Since this is a fiberwise isomorphism, as we saw in Section \ref{sec:morphisms:extensions}, it induces a bundle map
    \[ q_*\:\DD^1_{(q^*B, q^*\overline{\nabla})}\to \DD^1_{(B,\onabla)} \]
    which covers $q$. 

    It is easy to check that the short exact sequence in Lemma \ref{lem:sessymbolconnection} is natural in $q_*$. This means that there is a commutative diagram of short exact sequences
    \[
        \xymatrix{
            0 \ar[r] & \Hom(\wedge^2q^*B, q^*B) \ar[r] \ar[d]_{q_*} & \DD^1_{(q^*B, q^*\overline{\nabla})} \ar[r] \ar[d]_{q_*} & \Hom(q^*B, \nu(q^!\FF)) \ar[r] \ar[d]_{q_*} & 0 \\ 
            0 \ar[r] & q^*\Hom(\wedge^2B, B)\ar[r] & q^*\DD^1_{(B, \overline{\nabla})} \ar[r] & q^*\Hom(B, \nu(\FF)) \ar[r] & 0
        }
    \]
    Both vertical arrows on the sides are injective, and therefore so is the middle vertical arrow. Since $q\: Q\to M$ is an invariant map, the section $q^*\D$ is in the image of $q_*$, so there is a unique relative derivation $\D_Q$ on $(Q, q^!\FF, q^*B, q^*\overline{\nabla})$ that satisfies $q_*(\D_Q) = q^*\D$. This completes the proof.
\end{proof}

\begin{definition}
    If $q$ is an injective immersion that satisfies the conditions of Proposition \ref{prop:restriction}, we call $(q^*B, q^*\overline{\nabla}, \D_Q)$ the \textbf{restriction} of $(B, \overline{\nabla},\D)$ to $Q$. 
\end{definition}

\begin{example}
    If $(A, p,\D)$ is a 1-integrable relative algebroid, then its first prolongation $(B, p_1, \D^{(1)})$ is the restriction of the tautological relative algebroid $(A, p_1, \UD)$ to $M^{(1)}\hookrightarrow \DD^1_A$.
\end{example}

It should be noted that, in general, restriction does not preserve any kind of integrability of the relative algebroid. From the point of view of the realization problem, restriction amounts to adding extra equations to the problem. So a problem that originally had realizations may stop having them.


The vanishing locus of the anchor is always invariant.

\begin{proposition}
    Let $(B, \onabla, \D)$ be a relative algebroid and let $Q\subset M$ be a submanifold along which $\rho$ vanishes. Then $Q$ is an invariant submanifold.
\end{proposition}

\begin{proof}
    In this case, $\im \rho_x =0$ for all $x\in Q$, so $Q$ is clearly invariant.
\end{proof}

If $(B^{(\infty)}, \D^{(\infty)})$ is the prolongation tower of a relative algebroid $(B, \overline{\nabla}, \D)$, then the kernel of the anchor $\ker \rho^{(\infty)}_{x_0}$ at each point $x_0\in M^{(\infty)}$ is a Lie algebra. This is the Lie algebra of a group of symmetries of a realization whose image contains $p_\infty(x_0)\in M^{(\infty)}$. So, in general, the larger the group of symmetries of a realization of a geometric problem is, the larger the kernel of the anchor of the corresponding relative algebroid must be. 

\begin{example}[Submanifolds where anchor has constant rank]
    In general, submanifolds along which the rank of the anchor is constant are not invariant, unless involutivity conditions are imposed. 

    For example, take the algebroid $\underline{\RR}^2\to \RR^2$ (relative to the identity), with derivation defined by
\[
    \begin{cases}
        \D \theta^1 = 
        \D \theta^2 = 0, \\
        \D x = -y \theta^1,\\
        \D y = \theta^2.
    \end{cases}
\]
The corresponding anchor satisfies $\rho(e_1) =-y \del_x$, $\rho(e_2) = \del_y$. Therefore, the submanifold $Q = \{\rk \rho = 1\} = \{ y = 0\}$ is not invariant. In this example, there are no realizations anywhere, as $\D^2x = \theta^1\wedge\theta^2\neq 0$.

\end{example}

\begin{example}[Universal bundle of Lie algebras]\label{ex:UniversalBundleOfLieAlgebras}
The universal bundle of Lie algebra structures on a vector space $V$ can be constructed, in an \emph{ad hoc} manner, as the bundle of Lie algebras -- a Lie algebroid with zero anchor -- $\underline{V}\to \hat{\mf{g}}$ where
\[
    \hat{\mf{g}} = \{c\in \Hom(\wedge^2V, V) \ | \ \Jac_c = 0\}
\]
with $\Jac_c(v_1, v_2, v_3) = c(v_1, c(v_2, v_3)) + \mathrm{c.p.}$ for $v_1, v_2, v_3\in V$. The bracket on $\underline{V}\to \hat{\mf{g}}$ is the tautological bracket: $[v_1, v_2](c) = c(v_1, v_2)$. A realization of this algebroid above a point $c\in \hat{\mf{g}}$ is a local Lie group integrating the Lie algebra $(V, c)$. 

This Lie algebroid arises by imposing maximal symmetry on the classifying algebroid for all $V$-coframes. Namely, let $p_1\:\DD^1_V\cong \Hom(\wedge^2V, V)\to *$ be the tautological relative algebroid corresponding to a vector space $V$. Consider its prolongation $(p_1^*V, p_2, \D^{(1)})$, relative to the projection $p_2\:(\DD^1_V)^{(1)}\to \DD^1_V$, where
\[
    (\DD^1_V)^{(1)} = \big\{ \xi_c\in p_1^*\Hom(V, \Hom(\wedge^2V, V)) \ | \ \Jac_c = \delta \xi_c \big\},\quad p_2(\xi_c) = c.
\]
The anchor of this prolongation at $\xi_c\in (\DD^1_V)^{(1)}$ vanishes if and only if $\xi_c = 0$. The restriction of $p_2$ to this invariant subvariety gives a canonical identification $\{\rho^{(1)} = 0 \} \cong \hat{\mf{g}}$. The restricted relative algebroid is precisely $\underline{V}\to \hat{\mf{g}}$, the  universal bundle of Lie algebra structures on $V$.

Realizations of the relative algebroid  $\underline{V}\to \hat{\mf{g}}$ correspond to $V$-coframes with maximal symmetry: these are just local Lie groups with their Maurer-Cartan forms!
\end{example}

\begin{example}[Jacobi manifolds]
Fixing a vector space $V$, Bryant in \cite{Bryant2014} defines a \emph{Jacobi manifold} as a submanifold $M\subset \DD^1_V \cong \Hom(\wedge^2V, V)$ such that 
    \[
        \Jac_c \in \delta(\Hom(\underline{V}, T_cM)), \quad \mbox{for all $c\in M$},
    \]
    where $\delta\: \Hom(V, \DD^1_V)\to \DD^2_V$ is the Spencer differential. In our language, this is the same as saying that the restriction of the tautological algebroid to $M$ has vanishing torsion class. Theorem 4 in \cite{Bryant2014} states that, when $M$ is real analytic, if $T_cM\subset \DD^1_V$ is an involutive tableau of derivations, then there exists a realization of the restricted algebroid through every point in $M$. 
\end{example}

\begin{example}[Riemannian metrics]
    As we recalled in the introductory section, the orthonormal frame bundle of a Riemannian manifold has a canonical coframe $(\theta, \omega)$ with values in $\RR^n\oplus \mf{o}(n)$ satisfying Cartan's structure equations \eqref{eq:IntroductionCartanStructureEquationsRiemannian}. This suggests that the realization problem for (locally orthonormal frame bundles of) Riemannian metrics can be obtained by restricting $\DD^1_{\RR^n\oplus \mf{o}(n)}$ to derivations of the form
    \begin{equation}
    \label{eq:Cartan:Riemannian}
        \begin{cases}
            \D \theta = -\omega \wedge \theta,\\
            \D \omega = R(\theta\wedge\theta) - \omega\wedge\omega
        \end{cases}
    \end{equation}  
    The space of such derivations is an affine subspace of $\DD^1_{\RR^n\oplus \mf{o}(n)}$ parametrized by $R\in\Hom(\wedge^2\RR^n, \mf{o}(n))$.

    The resulting restricted algebroid has a non-zero torsion class, and therefore the algebroid needs to be restricted further to the subspace where the torsion class vanishes. The torsion class can be computed by taking an extension $\tilde{\D}$ of $\D$. Such extension is completely determined by $\tilde{\D}R$, and to find this value one applies $\tilde{\D}$ to \eqref{eq:Cartan:Riemannian}:
    \begin{align*}
        0 &= \D^2\theta = -\D\omega \wedge \theta + \omega \wedge \D \theta = R(\theta\wedge\theta)\wedge\theta\\
        0 & = \tilde{\D}\D\omega = \tilde{\D}R \wedge (\theta\wedge\theta) + R(\D \theta \wedge \theta) - R(\theta \wedge \D \theta) - \D \omega \wedge \omega + \omega \wedge \D \omega\\
         &=\tilde{\D}R\wedge (\theta\wedge\theta)
    \end{align*}
    
    The torsion class can be identified with $R(\theta\wedge\theta)\wedge\theta$, and its vanishing is precisely the first (algebraic) Bianchi identity.
    
    One may also wonder about the second Bianchi identity. This will be a consequence of the vanishing of the curvature class, when imposing ``$\mf{o}(n)$-invariance of the prolongation". In order to implement this one needs to include in the notion of a relative algebroid a structure group. Such a concept, which may be called a \emph{relative $G$-structure algebroid}, would extend the notion of $G$-structure Lie algebroid studied in \cite{FernandesStruchiner2019,FernandesStruchiner2021} (see also Section \ref{sec:outlook}).
\end{example}

\subsection{Systatic foliation and reduction}

Recall that a relative algebroid is \emph{standard} when its tableau map $\tau$ is fiberwise injective. As we saw in the previous section, this happens precisely when the classifying map is an immersion. We will now show that the directions in which the tableau map is zero -- i.e., the directions in which the classifying map is constant -- are essentially ``redundant" from the perspective of the realization problem. The notion of systatic space and inessential invariants goes back to Cartan -- a modern formulation and discussion can be found in \cite{CrainicYudilevich2024}.

Again, in this section we fix a relative algebroid $(B, \overline{\nabla},\D)$ over $(M,\FF)$. We also assume that the kernel of its tableau map $\tau\:\FF\to \DD^1_{(B,\onabla)}$ has constant rank.

\begin{definition}
    The \textbf{systatic foliation} of $(B, \overline{\nabla},\D)$ is the foliation
    \[ \FF_{\text{sys}} := \ker \tau \subset \FF,\] 
    where $\tau$ is the tableau map of $(B, \overline{\nabla},\D)$. 
\end{definition}

Note that $\FF_{\text{sys}}$ is involutive because the $\FF$-connection $\overline{\nabla}$ induced on $\DD^1_{(B, \overline{\nabla})}$ is flat. Next, we present two useful characterizations of the systatic foliation.

\begin{proposition}
    For a relative algebroid $(B, \overline{\nabla},\D)$ the systatic foliation is given by
    \[
        \FF_{\sys} = \{ X\in \FF : \overline{\nabla}_X [b_1, b_2] = 0, \mbox{ for all $b_1, b_2\in \Gamma_{(B, \overline{\nabla})}$ } \},
    \]
    where $[\cdot, \cdot]$ denotes the bracket associated to $\D$.  
\end{proposition}

\begin{proof}
    By \eqref{eq:bracket:tableau}, for each $X\in\FF$, we have 
    \[ \onabla_X\D=0 \quad \Longleftrightarrow\quad  \overline{\nabla}_X [b_1, b_2] = 0, \text{ for all }b_1, b_2\in \Gamma_{(B, \overline{\nabla})}. \qedhere \]
\end{proof}

\begin{proposition}
    The systatic foliation of an algebroid $(A, p, \D)$ relative to a submersion coincides with the connected components of the fibers of the classifying map $c_{\D}$.
\end{proposition}

\begin{proof}
    By \eqref{eq:tableau:classifying_map}, we have $\ker\tau=\ker\d c_{\D}$,
    so the result follows.
\end{proof}

Let us now discuss how to get rid of directions along the systatic foliation. 

\begin{definition}
    Let $(B, \overline{\nabla},\D)$ be a relative algebroid.
    A foliation $\FF_0\subset\FF_\sys$ is called \textbf{inessential} if
    it is a simple foliation and $\overline{\nabla}$ has no holonomy along its leaves.
\end{definition}

Notice that an inessential foliation of $(B, \overline{\nabla},\D)$ is given by the fibers of a submersion $q\:M\to M_\red$. The foliation $\FF$ descends to a foliation $\FF_\red$ of the leaf space $M_\red$ characterized by
\[ \FF=(\d q)^{-1}(\FF_\red). \]
Moreover, by Proposition \ref{prop:BundlesWithFlatPartialConnectionsOfSimpleFoliations}, there is a unique flat foliated bundle $(B_\red,\onabla)$ over $M_\red$ whose pullback under $q\:M\to M_\red$ is $(B,\onabla)$. 

We use these notations in the statement of the following theorem showing that the derivation $\D$ can also be reduced, preserving all its essential properties.

\begin{theorem}[Reduction]
    Let $(B, \overline{\nabla},\D)$ be a relative algebroid and let $\FF_0\subset\FF_\sys$ be an inessential foliation with leaf space $q\: M\to M_{\red}$. Then:
    \begin{enumerate}[(i)]
        \item There exists a unique structure of a relative algebroid $(B_{\red}, \overline{\nabla},\D_{\red})$ such that $q_*\:B\to B_{\red}$ is a morphism of relative algebroids;
        \item The tableaux $\tau$ and $\tau_{\red}$ of $(B, \overline{\nabla},\D)$ and $(B_{\red}, \overline{\nabla},\D_{\red})$, have naturally isomorphic
        Spencer cohomologies;
         \item The relative algebroid $(B, \overline{\nabla},\D)$ is $k$-integrable if and only if $(B_{\red}, \overline{\nabla},\D_{\red})$ is $k$-integrable;
        \item There is a one-to-one correspondence
        \[
            \left\{
                \substack{\mbox{realizations $(P, r, \theta)$}\\
                \\
                \mbox{of $(B, \overline{\nabla},\D)$}}
            \right\}
            \overset{1:1}{\longleftrightarrow}
            \left\{ 
            \substack{
                    \mbox{realizations $(P, \tilde{r}, \tilde{\theta})$ of $(B_{\red}, \overline{\nabla},\D_{\red})$}\\
                    \mbox{together with a lift $r$ of $\tilde{r}$:}\\
                    \mbox{  $
                    \xymatrix@R=15 pt{ & M \ar[d]^q \\
                    P \ar[ru]^r\ar[r]_{\tilde{r}} & M_{\red}}$
                    }
                }
            \right\}. 
        \]
    \end{enumerate}
Moreover, if $\FF_0=\FF_\sys$ then the reduced algebroid $(B_{\red}, \overline{\nabla},\D_{\red})$ is standard.
\end{theorem}

\begin{remark}
    The reduced algebroid $(B_\red,\onabla,\D_\red)$ can be seen as a \emph{quotient} of the relative algebroid $(B,\onabla,\D)$ by the pseudogroup generated by the flows of vector fields tangent to the inessential foliation $\FF_0$. 
\end{remark}

\begin{remark}
    The prolongation of the reduction cannot be a (systatic) reduction of the prolongation. The reason is that the prolongation is always standard (see Lemma \ref{lem:tableauOfProlongation}). Intuitively, the prolongation before reduction adds extra equations that (locally) encode a map from a realization to the fibers of $q$, while this is lost in the prolongation of the reduction. However, since the morphism $q_*\:(B, \overline{\nabla},\D)\to (B_{\red}, \overline{\nabla},\D_{\red})$ is a fiberwise isomorphism, it follows from Proposition \ref{prop:naturalityOfRelativeAlgebroidConstructions} that, assuming $k$-integrability, there is an induced morphism between the $k$-prolongations.
\end{remark}

\begin{proof}
    To prove item (i), note that the map of flat foliated vector bundles 
    \[ q_*\: (B, \overline{\nabla}) \to (B_{\red}, \overline{\nabla})\] 
    is a fiberwise isomorphism. It follows from Section \ref{sec:morphisms:extensions} that we have a bundle map
    \[ q_*\: \DD^1_{(B, \overline{\nabla})} \to \DD^1_{(B_{\red}, \overline{\nabla})}. \]
    We claim that there is a unique $\D_\red$ which is $(q,q_*)$-related to $\D$, so (i) holds.

    To prove this claim, note that we have a commutative diagram of short exact sequences
    \[
        \xymatrix{
        0 \ar[r] & \Hom(\wedge^{k+1}B, B) \ar[d]_{q_*} \ar[r] & \DD^k_{(B, \overline{\nabla})} \ar[d]_{q_*} \ar[r] & \Hom(\wedge^k B, \nu(\FF)) \ar[d]_{q_*} \ar[r] & 0 \\
        0 \ar[r] & \Hom(\wedge^{k+1}B_{\red}, B_{\red}) \ar[r] & \DD^k_{(B_{\red}, \overline{\nabla})} \ar[r] & \Hom(\wedge^k B_{\red}, \nu(\FF_{\red})) \ar[r] & 0
        }
    \]
    Since the sides are fiberwise isomorphisms, so is the map in the middle. Now, because $\ker(\d q) \subseteq \ker \tau$, we have that $\onabla_X\D=0$ for any $X\in\ker(\d q)$, and since the holonomy of $\onabla$ vanishes along the directions of $\ker (\d q)$, we conclude that the section $\D\in \Gamma(\DD^1_{(B,\onabla)})$ is the pullback of a section $\D_{\red}\in \Gamma(\DD^1_{(B_\red, \onabla)})$. This also means that $\D$ and $\D_{\red}$ are $(q,q_*)$-related, so the claim follows.
    
    In order to prove item (ii), we look at the relationship between the tableau maps of $B$ and $B_{\red}$. According to Proposition \ref{prop:naturalityOfRelativeAlgebroidConstructions}, the map $q_*$ intertwines the tableau maps and the Spencer differentials. For $k\ge 1$ this amounts to the commutativity of the following diagram where the rows are short exact sequences: \small 
    \[
    \xymatrix@C=13pt{
        {\Hom(\wedge^lB, \Hom(S^kB, \ker (\d q)))} \ar[r] \ar[d]_{\delta} & {\Hom(\wedge^lB, \tau^{(k)})} \ar[r]^---{q_*} \ar[d]_{\delta} & {q^*\Hom(\wedge^lB_{\red}, \tau^{(k)}_{\red})} \ar[d]_{\delta}\\
        {\Hom(\wedge^{l+1}B, \Hom(S^{k-1}B, \ker(\d q)))} \ar[r]             & {\Hom(\wedge^{l+1}B, \tau^{(k-1)})} \ar[r]^-{q_*}               & {q^*\Hom(\wedge^{l+1}B_{\red}, \tau^{(k-1)}_{\red})}
    }
    \]
    \normalsize
    while for $k=0$ one has the commutative diagram
    \small 
    \[
    \xymatrix{
       {\Hom(\wedge^lB, \ker (\d q))} \ar[r] \ar[d] & {\Hom(\wedge^lB, \FF)} \ar[r]^---{q_*} \ar[d]_{\delta_{\tau}} & {q^*\Hom(\wedge^lB_{\red}, \FF_{\red})} \ar[d]_{\delta_{\tau_{\red}}} \\
       0 \ar[r]                            & {\DD^{l+1}_{(B, \onabla)}} \ar[r]^{q_*}                    & {q^*\DD^{l+1}_{(B_{\red}, \onabla)}}
    }
    \]
    \normalsize
    Setting $l=1$, we conclude that $\tau^{(k)}$ has constant rank if and only if $\tau_{\red}^{(k)}$ has constant rank. Moreover, since the restriction
    $\tau\big\vert_{\ker(\d q)}$
    is the zero tableau map, which is involutive, the map $q_*$ descends to an isomorphism in cohomology:
    \[
        q_*\: H^{k, l}(\tau) \xrightarrow{\sim} q^*H^{k, l}(\tau_{\red}).
    \]

    Now, using this isomorphism, item (iii) also follows: by Proposition \ref{prop:naturalityOfRelativeAlgebroidConstructions} (i), torsion classes of the $k$-th prolongations are $q_*$-related; hence, if $B$ and $B_{\red}$ are $(k-1)$-integrable, then $B$ is $k$-integrable if and only if $B_{\red}$ is $k$-integrable.

    Finally, to prove item (iv), observe that by Proposition \ref{prop:naturalityOfRelativeAlgebroidConstructions} (ii) every realization $(P,r,\theta)$ of $(B,\onabla,\D)$ induces a realization $(P,\tilde{r},\tilde{\theta})$ of $(B_\red,\onabla,\D_\red)$ with $\tilde{r}=q\circ r$. Conversely,
    given a realization $(P,\tilde{r},\tilde{\theta})$ of $(B_\red,\onabla,\D_\red)$ and a lift $r\:P\to M$ of $\tilde{r}$, using the fact that $q_*\: \DD^1_{(B, \overline{\nabla})}\to q^*\DD^1_{(B_{\red}, \overline{\nabla})}$ is an isomorphism of vector bundles, there is a unique vector bundle map $(\theta, r)\: TP\to B$ such that $\tilde{\theta}=q_* \circ \theta$. Moreover, since $q_*$ is a morphism of relative algebroids, we have
    \begin{align*}
        q_*(\d\circ \theta^* - \theta^*\circ \D)
        &= \d\circ\theta^*\circ q^* - \theta^*\circ \D\circ q^*\\
        &= \d \circ (q_*\circ \theta)^* - (q_*\circ \theta)^* \circ \D_{\red}=\d\circ\tilde{\theta}^*-\tilde{\theta}^*\circ\D_{\red}=0.
    \end{align*}
    Since $q_*$ is a fiberwise isomorphism we must have $\d\circ \theta^* = \theta^*\circ \D$,  so $(\theta, r)$ is a morphism of almost relative algebroids. 
    
    These constructions are inverse to each other, so this completes the proof.
\end{proof}

\section{Relative connections and PDEs}\label{sec:PDE}

We will now show that any partial differential equation can be recast as a relative algebroid in such a way that the formal theory of prolongations \cite{Goldschmidt1967} coincides with the prolongation theory for relative algebroids. The relative algebroid of a PDE arises from a relative connection, so we start by discussing this notion.

\subsection{Relative connections and the Cartan distribution}
Let $N\xrightarrow{q} X$ be a submersion. An Ehresmann connection for $q$ is a splitting $TN\cong \ker (\d q) \oplus q^*TX$. Sometimes, however, the splitting does not arise on the level of $N$, but rather depends on additional coordinates. Let us illustrate this with an important example.

\begin{example}\label{ex:CartanDistributionOnJetSpace}
   Let $q\: N\to X$ be any submersion. The bundle $p_1\:J^1q\to N$ of first jets of local sections can be identified with the bundle of horizontal complements of $\ker (\d q)$ in $TN$, that is
   \[
        (J^1q)_n \cong\{ C_n \subset T_nN: C_n\oplus \ker (\d_n q) = T_n N\}.
   \]
   There is no canonical splitting of $TN$ into the components $\ker (\d q)$ and $q^*TX$, but there is a tautological splitting of $p_1^*TN$, given by the identifications
   \[
        p_1^*TN \cong p_1^*\ker (\d q)\oplus \mathcal{C}, \quad \left(p_1^*TN\right)_{C_n} = \ker T_nq \oplus C_n.
   \]
   The subbundle $\mathcal{C}\subset p_1^*TN$ is called the \textbf{Cartan distribution}. We recall that its relevance arises from the fact that it detects which sections $\tau$ of $q_1\:J^1q\to X$ are holonomic, i.e., of the form $\tau=j^1\sigma$, with $\sigma:X\to N$ a section of $q\:N\to X$. In fact, one has:
   \begin{itemize}
       \item A local section $\tau\:X\to J^1q$ is holonomic if and only if it is tangent to the Cartan distribution, i.e., if
       \[
            \im \d_x(p_1\circ\tau) \subset \mathcal{C}_{\tau(x)} \mbox{ for all $x\in \dom \tau$}.
       \]
   \end{itemize}
\end{example}

The notion of a relative connection formalizes the type of behavior found in the previous example and generalizes the notion of Ehresmann connection.

\begin{definition}
    Let $M\xrightarrow{p} N \xrightarrow{q} X$ be two submersions. A \textbf{connection on $q$ relative to $p$} is a vector bundle $C\subset p^*TN$ complementary to $p^*\ker (\d q)$:
    \[
        p^*TN \cong C\oplus p^*\ker (\d q).
    \]
    Equivalently, it is a map $c\:M\to J^1q$ such that $p_1 \circ c = p$, where $p_1\: J^1q\to N$ is the projection.
\end{definition}

A relative connection $C$ on $M\xrightarrow{p}N\xrightarrow{q} X$ gives rise to an algebroid $(A,p,\D)$ relative to $p$. For the vector bundle, we take $A= q^*TX$ and, under the identification $p^*A\xrightarrow{\sim} C$, the anchor map corresponds to the inclusion 
\[ \rho\:p^*A \xrightarrow{\sim} C\subset p^*TN.\] 
If $X_1,X_2\in \Gamma(TX)$ are vector fields on $X$, then the relative bracket is determined by
\[
    [q^*X_1, q^*X_2]_{\D} = p^*q^*[X_1,X_2],
\]
and extended to any sections of $A$ through the Leibniz rule using the anchor. The resulting relative algebroid is an example of a relative algebroid with injective anchor. 

To identify the derivation $\D$ of this relative algebroid one proceeds as follows. Recall from Section \ref{sec:morphisms:extensions} that there is a canonical map $q_*\: \DD^1_{q^*TX}\to q^*\DD^1_{TX}$, which pulls back to a map 
\[
        p^*q_*\: p^*\DD^1_{q^*TX}\to p^*q^*\DD^1_{TX}
\]
The derivation $\D$ is a lift of the de Rham differential $\d$ in the sense that one has $p^*q_*(D) = p^*q^*\d$, where $\d$ is interpreted as a section $\d \in \Gamma(\DD^1_{TX})$. 
\[
\xymatrix{
p^*\DD^1_{q^*TX} \ar[rr]^{p^*q_*}\ar[dr] && p^*q^*\DD^1_{TX}\ar[dl] \\
 & M \ar@/^8pt/[ul]^{\D} \ar@/_8pt/[ur]_{p^*q^*\d} 
}
\]
This is the defining feature of the relative derivation associated to a relative connection.

\begin{proposition}\label{prop:relativeConnectionAndDerivations}
    Let $M\xrightarrow{p}N \xrightarrow{q} X$ be two submersions. There is a one-to-one correspondence between $p$-relative connections $C$ on $q$ and sections $D\in \Gamma(p^*\DD^1_{q^*TX})$ with $p^*q_* \D = p^*q^*\d$. In particular, given $C$ the corresponding derivation $\D_C$ is determined by
    \[
    \begin{cases}
        \D_C f=(p^*\d f)|_C, & \text{if }f\in C^\infty(N),\\
        \D_C (q^*\alpha)=p^*q^*\d \alpha, & \text{if }\alpha\in\Omega^\bullet(X).
    \end{cases}
    \]
\end{proposition}  

\begin{proof}
Let $\D\in \Gamma(p^*\DD^1_{q^*TX})$. The equation 
\[ (p^*q_* \D)(q^*f)= p^*q^*\d (q^*f),\quad (f\in C^\infty(X)),\]
holds if and only if for any vector field $X\in\Gamma(TX)$ one has
\[ (p^*q_* \D)(q^*f)(p^*q^*X)=(p^*q^*\d (q^*f))(p^*q^*X)\]
and this is equivalent to:
\[ \rho_{\D}(p^*q^*X)(p^*q^*f)=p^*q^*X(f). \]
This last equation holds if and only if $\rho_{\D}\:p^*A\to p^*TN$ is injective with image a subbundle $C$ complementary to $p^*\ker(\d q)$. 

On the other hand, the equation 
\[ (p^*q_* \D)(q^*\alpha)= p^*q^*\d (q^*\alpha),\quad (\alpha\in \Omega^1(X)),\]
holds if and only if for any vector fields $X_1,X_2\in\Gamma(TX)$ one has
\[ (p^*q_* \D)(q^*\alpha)(p^*q^*X_1,p^*q^*X_2)=(p^*q^*\d (q^*\alpha))(p^*q^*X_1,p^*q^*X_2),\]
and in terms of the bracket of $\D$ this is equivalent to
\[ [q^*X_1,q^*X_2]_{\D}=p^*q^*[X_1,X_2], \]
so the result follows.
\end{proof}

The Cartan distribution $\Cc\subset p_1^*TN$ in Example \ref{ex:CartanDistributionOnJetSpace}, being a relative distribution for the submersions $J^1q\xrightarrow{p_1}N \xrightarrow{q} X$, has an associated relative algebroid with derivation
\[ \D_\Cc\:\Omega^\bullet(q^*TX)\to\Omega^{\bullet+1}(q^*_1TX), \]
where $q_1:=q\circ p_1\:J^1q\to X$.  Notice that, by the previous proposition, one has
\begin{equation}
    \label{eq:Cartan:derivation}
    \begin{cases}
        \D_{\Cc} f=(p_1^*\d f)|_\Cc, & \text{if }f\in C^\infty(N),\\
        \D_\Cc (q^*\alpha)=q_1^*\d \alpha, & \text{if }\alpha\in\Omega^\bullet(X).
    \end{cases}
\end{equation}

\begin{definition}
    The derivation $\D_{\mathcal{C}}$ is called the \textbf{Cartan derivation}.
\end{definition}

The previous proposition leads to the following description of $\D_\Cc$.

\begin{corollary}\label{cor:CartanDerivation}
    Let $q\: N\to X$ be a submersion. The bundle of first jets is naturally isomorphic to 
    \[
        J^1q\cong q_*^{-1}(\d):=\left\{ \D_n\in \DD^1_{q^*TX} \ | \ q_*(\D_n) = \d_{q(n)}  \right\},
    \]
    where the natural isomorphism is given by restriction of the symbol map $\sigma\: \DD^1_{q^*TX}\to \Hom(q^*TX, TN)$.
    The relative algebroid $(q^*TX,p_1,\D_\Cc)$ has classifying map the resulting inclusion
    \[ \xymatrix{c_{\D_\Cc}\:J^1q\, \ar@<-2pt>@{^{(}->}[r] & \DD^1_{q^*TX}}. \]
\end{corollary}

\begin{proof}
    The previous result applied to $p=\id$ shows that for a fixed $n\in N$, we have
    \[ \left\{ \D_n\in (\DD^1_{q^*TX})_n \ | \ q_*(\D_n) = \d_{q(n)}  \right\}=\{ C_n \subset T_nN: C_n\oplus \ker (\d_n q) = T_n N\}.\]
    The left side is canonically isomorphic to $(J^1q)_n$. 
\end{proof}

\begin{example}[Cartan derivation in local coordinates]
    Let us assume that we have fixed local charts $(V,x^i)$ for $X$ and $(U=q^{-1}(V),x^i,u^a)$ for $N$, so that 
    \[ q\:N\to X, \quad q(x^i,u^a)=x^i. \] 
    Then we have an induced chart $(p_1^{-1}(U),x^i,u^a,u^a_i)$ on the total space of the first jet bundle so that
    \[ p_1\:J^1q\to N, \quad p_1(x^i,u^a,u^a_i)=(x^i,u^a). \] 
    Also, let $e_i=\partial_{x^i}$ be the corresponding local frame for $TX$ with dual coframe $\theta^i=\d x^i$. Then a form $\al\in\Omega^k(q^*TV)$ can be expressed as
    \[ \al=\sum_{i_1<\cdots< i_k} \al_{i_1,\dots,i_k}(x,u)\, \d x^{i_1}\wedge\cdots\wedge \d x^{i_k}, \]
    and it follows from \eqref{eq:Cartan:derivation} that the Cartan derivation acts on such a form as the total exterior derivative
    \[ \D_{\Cc}\al=\sum_j \sum_{i_1<\cdots< i_k} (D_j\al_{i_1,\dots,i_k})\, \d x^j\wedge \d x^{i_1}\wedge\cdots\wedge \d x^{i_k}, \]
    where 
    \[ D_jf:=\frac{\partial f}{\partial x^j}+\sum_a \frac{\partial f}{\partial u^a}u^a_j. \]
\end{example}

One obtains higher order Cartan derivations by considering the higher order jet spaces of the submersion $q\: N\to X$. For any integer $k\geq 1$, one constructs a connection on $J^{k-1}q \to X$ relative to $p_k\: J^kq \to J^{k-1}q$ by considering first the inclusion $J^kq\hookrightarrow J^1 (J^{k-1}q)$ and then restricting the Cartan distribution $\mathcal{C}\subset p_k^*TJ^{k-1}q\to J^1(J^{k-1}q)$ to $J^kq$. The resulting relative connection will also be called the (higher order) \textbf{Cartan distribution}. 

The higher-order Cartan distribution, being a relative distribution for the submersions $J^kq\xrightarrow{p_k}J^{k-1}q \xrightarrow{q_{k-1}} X$, has an associated relative algebroid with derivation
\[ \D_\Cc\:\Omega^\bullet(q^*_{k-1} TX)\to\Omega^{\bullet+1}(q^*_kTX). \]
We also call $\D_\Cc$ the (higher order) \textbf{Cartan derivation}. Again, applying Proposition \ref{prop:relativeConnectionAndDerivations}, one has
\begin{equation}
    \label{eq:higher:Cartan:derivation}
    \begin{cases}
        \D_{\Cc} f=(p_k^*\d f)|_\Cc, & \text{if }f\in C^\infty(J^{k-1}q),\\
        \D_\Cc (q_{k-1}^*\alpha)=q_k^*\d \alpha, & \text{if }\alpha\in\Omega^\bullet(X).
    \end{cases}
\end{equation}

\begin{example}[Higher order Cartan derivation in local coordinates]\label{ex:HigherCartanDerivationsInLocalCoordinates}
     Similar to the case of first order, we can describe the higher order Cartan derivation in local coordinates as follows. We let $n=\dim X$ (number of independent variables) and $d=\dim N-\dim X$ (number of dependent variables) and we fix local charts $(V,x^i)$ for $X$ and $(U=q_{k-1}^{-1}(V),x^i,u^a,u^a_J)$ for $J^{k-1}q$, with $J=(j_1,\dots,j_r)$ all unordered $r$-tuples of integers with $1\le j_l\le n$, and $\#J=r\le k-1$. Then a form $\al\in\Omega^k(q_{k-1}^*TV)$ can be expressed as
    \[ \al=\sum_{i_1<\cdots< i_k} \al_{i_1,\dots,i_k}(x,u,u_J)\, \d x^{i_1}\wedge\cdots\wedge \d x^{i_k}, \]
    and the Cartan derivation \eqref{eq:higher:Cartan:derivation} acts on such a form as the total exterior derivative
    \[ \D_{\Cc}\al=\sum_j \sum_{i_1<\cdots< i_k} (D_j\al_{i_1,\dots,i_k})\, \d x^j\wedge \d x^{i_1}\wedge\cdots\wedge \d x^{i_k}, \]
    where now
    \[ D_jf:=\frac{\partial f}{\partial x^j}+\sum_{a=1}^n \sum_{\#J<k}  \frac{\partial^{\#J} f}{\partial u_J^a}u^a_{j,J}. \]
\end{example}

\subsection{The relative algebroid of a PDE}

We now wish to associate a relative algebroid to a PDE. By the latter we mean:

\begin{definition}
        A \textbf{partial differential equation (PDE)} of order $k$ on a submersion $q\: N\to X$ is a submanifold $E\subseteq J^kq$. A \textbf{solution} to $E$ is a (local) section $\sigma$ of $q\: N\to X$ such that $\im j^k\sigma \subset E$.
\end{definition}

If we assume that the image $p_k(E)\subseteq J^{k-1}q$ is a manifold and that the map $q_k = q_{k-1}\circ p_k\: E\to X$ is a submersion, then the Cartan distribution restricts to a connection of $q_{k-1}\: p_k(E)\to X$ relative to the submersions $E\xrightarrow{p_k}p_k(E)\xrightarrow{q_{k-1}} X$. The corresponding relative algebroid of the PDE is an algebroid $(q_{k-1}^*TX,p_k,\D_E)$ relative to the submersion $p_k\:E\to p_k(E)$, and the derivation $\D_E$ is determined by
\begin{equation}
    \label{eq:PDE:derivation}
    \begin{cases}
        \D_E f=(p_k^*\d f)|_\Cc, & \text{if }f\in C^\infty(p_k(E)),\\
        \D_E (q_{k-1}^*\alpha)=i^*_Eq_k^*\d \alpha, & \text{if }\alpha\in\Omega^\bullet(X),
    \end{cases}
\end{equation}
where $i_E\:E\hookrightarrow J^kq$ is the inclusion.

\begin{definition}
    We call $(q_{k-1}^*TX, p_k, \D_E)$ the \textbf{relative algebroid of the PDE} $E\subset J^kq$.
\end{definition}

\begin{remark}
If we make the weaker assumption that the PDE $E\subseteq J^k q$ intersects the fibers of $p_k\: J^kq \to J^{k-1}q$ in submanifolds of a fixed dimension, then, by Proposition \ref{prop:restriction}, the relative algebroid $(q_{k-1}^*TX,p_k,\D_\Cc)$ associated to the Cartan distribution can be restricted to $E$. Hence, we still have a relative algebroid $(q_{k-1}^*TX, \overline{\nabla}, \D_E)$ associated to $E$. The results that follow are valid in this more general setting, replacing  $(q_{k-1}^*TX, p_k,\D_E)$ by this algebroid relative to a foliation. To simplify the exposition, we choose to stay within the framework of algebroids relative to a submersion.
\end{remark}

Our next result shows that the relative algebroid of the PDE encodes its solutions.

\begin{theorem}
    Let $E\subset J^kq$ be a PDE with relative algebroid $(q_{k-1}^*TX, p_k, \D_E)$. Then germs of solutions to $E$ are in 1-1 correspondence with germs of realizations of $(q_{k-1}^*TX, p_k, \D_E)$ modulo diffeomorphisms.
\end{theorem}

\begin{proof}
Let $\sigma\:U\to N$ be a local solution of $E$. Then we construct a realization $(\theta,r)\:TP\to q_k^*TX$ of $(q_{k-1}^*TX, p_k, \D_E)$ by setting:
\[ P:=U,\quad r(x):=j^k_x\sigma,\quad \theta(v_x)=(r(x),v_x). \]
That $\theta$ preserves anchors is clear. Using \eqref{eq:PDE:derivation} one finds that for any $\al\in\Omega^1(X)$
\[ \theta^*\D_E(q^*_{k-1}\alpha)=\theta^*i^*_E q_k^*\d \alpha=\d\alpha=\d \theta^*\al. \]

Conversely, assume that $(P, r, \theta)$ is a realization around $p\in P$ of $(q_{k-1}^*TX, p_k, \D_E)$ such that $r(p) = e\in E$. Then the map $q_k\circ r$ is a local diffeomorphism, so in a neighborhood of $p\in P$ it factors through a (local) section $\tau\: X\to E$. 
\[
    \xymatrix{
            P \ar[rd] \ar[r]^{r} & E \ar[d]_{q_k}\\
                            & X \ar@{-->}@/_5pt/[u]_{\tau}
    }
\]
The compatibility of $(r,\theta)$ with the anchor, gives
\[
    \im \d_x(p_k\circ\tau) \subset \mathcal{C}_{\tau(x)} \mbox{ for all $x\in \dom \tau$}.
\]
So $\tau$ is tangent to the Cartan distribution, and we can conclude that it is holonomic. Hence, $\tau=j^k\sigma$ for a local section $\sigma\:X\to N$, which is the desired local solution of $E$.

\end{proof}

\subsection{Prolongation and integrability of PDEs}
We will now show that the formal theory of prolongations \cite{Goldschmidt1967} for PDEs coincides with the prolongation theory for the associated relative algebroids.

\begin{theorem}\label{thm:PDEsAndRelativeAlgebroids}
Let $E\subset J^kq$ be a PDE with relative algebroid $(q_{k-1}^*TX, p_k, \D_E)$. Then:
    \begin{enumerate}[(i)]
        \item $E$ is a 1-integrable PDE if and only if the relative algebroid $(q_{k-1}^*TX, p_k, \D_E)$ is 1-integrable;
        \item If $E$ is a 1-integrable PDE, then the relative algebroid corresponding to the prolongation $E^{(1)}\subset J^{k+1}q$ is the prolongation of $(q_{k-1}^*TX, p_k, \D_E)$. 
    \end{enumerate}
In particular, a PDE is formally integrable if and only if its associated relative algebroid is.
\end{theorem}

\begin{remark}[Variational bicomplex]\label{rk:variationalBicomplex}
    When $E= J^1\pi$ Theorem \ref{thm:PDEsAndRelativeAlgebroids} shows that the prolongations of $E$ correspond to the higher order Cartan derivations. From their expression in local coordinates (see Example \ref{ex:HigherCartanDerivationsInLocalCoordinates}), one sees that they assemble together into the horizontal differential of the first row of the variational bicomplex (see, e.g., \cite{Anderson1992}). In other words, the profinite Lie algebroid corresponding to the prolongation tower has as derivation the horizontal differential of the variational bicomplex:
    \[
    \xymatrix{
    {\left(\pi_\infty^*TX, \d_H\right)\:\ldots} \ar[r] \ar@<-12pt>[d] & \pi_k^*TX \ar[d] \ar[r] & \pi_{k-1}^*TX \ar[r] \ar[d] \ar@{-->}@/^1pc/[l]^{\D_{\Cc}} & \ldots \ar[r] & \pi_1^*TX \ar[r] \ar[d] & \pi^*TX \ar[d] \ar@{-->}@/^1pc/[l]^{\D_{\Cc}}\\
    J^{\infty}\pi\:\ldots \ar[r]                                            & J^k\pi \ar[r]_{p_k}    & J^{k-1}\pi \ar[r]                                                               & \ldots \ar[r] & J^1\pi \ar[r]_{p}    & N.}               
    \]
\end{remark}

\begin{proof}[Proof of Theorem \ref{thm:PDEsAndRelativeAlgebroids}]
    Let us consider first the case $E=J^kq$, so $\D_E=\D_\Cc$. On the one hand, $E$ as a PDE, has first prolongation  
    \[ E^{(1)} = J^{k+1}q\cong \left\{j^1_y\tau\: \tau\in\Gamma(J^kq)\text{ holonomic}\right\}\subset J^1(J^kq). \]
    On the other hand, by Corollary \ref{cor:CartanDerivation}, we have the following description of the first jet bundle of $J^kq$:
    \[ J^1(J^kq)=\left\{ \D_y\in\DD^1_{q_k^*TX}\: (q_k)_*\D_y=\d_{q_k(y)}\right\}. \]
    This gives a description of $E^{(1)}$ in terms of derivations as
    \begin{align*}
        E^{(1)}\cong & \left\{ \D_y\in\DD^1_{q_k^*TX}\: (q_k)_*\D_y=\d_{q_k(y)},\ \D_y\circ \D_{\Cc}=0\right\}\\
        &=\left\{ \D_y\in\DD^1_{q_k^*TX}\: (q_k)_*\D_y=(\D_\Cc)_{p_k(y)},\ \D_y\circ \D_{\Cc}=0\right\}
    \end{align*}
    which is precisely the first prolongation space of the relative algebroid $(q_{k-1}^*TX, p_k, \D_\Cc)$. It follows that the theorem holds in this case.
    
    For general $E\subset J^kq$, the result follows because the first prolongation of $E$ can be described as $E^{(1)} = J^1E\cap J^{k+1}q$ in $J^1(J^kq)$, which corresponds to the first prolongation of the relative algebroid using the description for $J^{k+1}q$ in terms of derivations. This proves both (i) and (ii).
\end{proof} 

\begin{example}
    We illustrate the theorem with the simple PDE $u_x = y$, where $u=u(x,y)$. As a manifold, this PDE has coordinates $(u_y, u, x, y)$ and sits inside 
    \[ \{(y, u_y, u, x, y)\}\subset J^1\pi= \{(u_x, u_y, u, x, y)\}, \]
    where $\pi\: \RR^3\to \RR^2$ is the projection $\pi(u, x, y) = (x, y)$.
    The corresponding relative algebroid is obtained by restricting the Cartan derivation to $E$, so it can be described by the trivial vector bundle $\underline{\RR^2}\to\RR^3$, with derivation $\D$ determined by
    \[
    \begin{cases}
        \D \theta^1 =\D \theta^2= 0,\\
        \D x = \theta^1,\\
        \D y = \theta^2,\\
        \D u = y \theta^1 + u_y \theta^2.
    \end{cases}
    \]
    The free variable is $u_y$, so to compute the prolongation we start with an extension of $\D$, which is determined by $\tilde{\D}u_y = u_{xy} \theta^1 + u_{yy} \theta^2$. We find that $\tilde{\D}$ must satisfy
    \[
        0 = \tilde{\D}\D u = \theta^2\wedge\theta^1 + (u_{xy}\theta^1 + u_{yy}\theta^2)\wedge\theta^2 = (u_{xy} - 1) \theta^1\wedge \theta^2.
    \]
    We find that the 1st prolongation of the relative algebroid is characterized by $u_{xy}=1$ and that $u_{yy}$ is the new variable. This corresponds exactly to the relative algebroid underlying the prolongation $E^{(1)}\subset J^2\pi$ of the PDE: the latter is given by the equations $\{u_x =y, u_{xx} = 0, u_{xy} = 1\}$, so $E^{(1)}$ is parametrized by $\{(u_{yy}, u_y, u, x, y)\}$.
\end{example}
\begin{remark}[PDEs with symmetries and Pfaffian fibrations]
    Symmetries of PDEs are described by pseudogroups of diffeomorphisms (see, e.g., \cite{Olver93}). One may wonder how symmetries and reduction of PDEs can be dealt with in the context of relative algebroids. In \cite{Accornero2021, AccorneroCattafiCrainicSalazarBook, Cattafi2020, CattafiCrainicSalazar2020, Salazar2013}, a framework for studying PDEs with symmetries was developed using so-called Pfaffian fibrations and Pfaffian actions of Pfaffian groupoids. In \cite{smilde2025}, we show that relative algebroids underlie Pfaffian fibrations and that Pfaffian actions induce symmetries of the associated relative algebroids.
\end{remark}

\section{Postlude: an example}\label{sec:postlude}

In this final section, we will revisit Example \ref{ex:introductionSurfaces|nablaK|=1} from the Introduction and we discuss it using the framework developed in the paper. This example, considered by Bryant in \cite[\S 5.1]{Bryant2014}, is simple enough that it can be solved directly, but it is extremely insightful to study it from the perspective of relative derivations.

As discussed in the Introduction, the existence and classification problem of surfaces with a metric whose Gauss curvature satisfies $|\nabla K| =1$ is governed by the equations
\begin{equation}
\label{eq:SurfacesStructureEquations}
        \begin{cases}
            \d \theta^1 = - \theta^3 \wedge \theta^2,\\
            \d \theta^2 = \theta^3 \wedge\theta^1, \\
            \d \theta^3 = K \theta^1 \wedge\theta^2,\\
            \d K = \cos (\varphi) \theta^1 + \sin (\varphi) \theta^2.
        \end{cases}
    \end{equation}
These equations define a derivation $\D_A$ on the trivial vector bundle $A:= \underline{\RR^3}\to \RR$, relative to the projection $p\: \SS^1\times \RR\to \RR$, where $\RR$ has coordinate $K$ and $\SS^1$ has coordinate $\varphi$. Here $\{\theta^1,\theta^2,\theta^3\}$ is a basis of sections of $A^*$ and if we let $\{e_1,e_2,e_3\}$ be the dual basis of sections of $A$, the anchor of this relative algebroid is given by
\begin{align*}
    \rho(e_1) &= \cos (\varphi) \del_K, &
    \rho(e_2) &= \sin (\varphi) \del_K, & 
    \rho(e_3) &= 0,
\end{align*}
while the bracket takes the form
\begin{align*}
    [e_1, e_2] &= - K e_3, & [e_2, e_3] &= - e_1, & [e_3, e_1] &= - e_2.
\end{align*}
Using this expression for the anchor, a straightforward computation shows that this relative algebroid has tableau map $\tau\: \ker(\d p) \to p^*\Hom(\underline{\RR^3}, T\RR)$  given by
\[
    \tau(\del_\varphi) = \left( -\sin (\varphi) \theta^1  + \cos(\varphi)\theta^2\right) \otimes \del_K.
\]
In the sequel, it will be convenient to use the following notation for the expression appearing in the tableau:
\[ \del_\varphi( \d K):= -\sin (\varphi) \theta^1  + \cos(\varphi)\theta^2. \]

All higher prolongations can be explicitly computed, giving the full prolongation tower
\[
    \xymatrix{
        \underline{\RR^3} \ar[r] \ar[d]  & \cdots \ar[r] & \underline{\RR^3} \ar[d] \ar[r]      & \underline{\RR^3} \ar[r] \ar[d]   & \cdots \ar[r]     & \underline{\RR^3} \ar[d] \\
        \RR^\infty\times \SS^1\times \RR \ar[r] & \cdots \ar[r]         & \RR^k\times \SS^1\times \RR \ar[r]_{p_k} & \RR^{k-1}\times \SS^1\times \RR \ar[r] & \cdots \ar[r]_{p_1}               & \RR, 
    }
\]
where $\RR^k\times \SS^1\times \RR$ has coordinates $(c_{k-1}, \ldots, c_0, \varphi, K)$, and the relative derivation is determined by \eqref{eq:SurfacesStructureEquations} together with
\begin{equation}\label{eq:postludeEquationdck}
    \D c_k = P_k[c_0, \dots, c_k] \d K + c_{k+1} \del_\varphi(\d K)
\end{equation}
where $P_k[c_0, \dots, c_k]$ are polynomials given by
\begin{align*}
    P_0[c_0] &= - c_0^2 - K\\
    P_{2m-1}[c_0, \dots, c_{2m-1}] &= -\sum_{i=0}^{m-1} \binom{2m+1}{i+1} c_i c_{2m-1-i},\\
    P_{2m}[c_0, \dots, c_{2m}] &= -\binom{2m+1}{m} c_{m}^2 - \sum_{i=0}^{m-1} \binom{2m+2}{i+1} c_i c_{2m-i},
\end{align*}
for $m\geq 1$. Note that we can also interpret \eqref{eq:postludeEquationdck} as defining a profinite derivation on $\underline{\RR^3}\to\RR^\infty\times\SS^1\times\RR$

Equation (\ref{eq:postludeEquationdck}) suggests using the new global coframe $\{ \d K, \del_\varphi(\d K), \theta^3\}$, for which the anchor is decoupled such that the profinite part is concentrated in one basis vector only. The frame $\{b_1, b_2, b_3\}$ dual to this coframe is given by
\begin{align}
    \begin{split}\label{eq:exdKAlternativeBasis}
        b_1 &= (\cos\varphi)e_1 + (\sin \varphi) e_2\\
        b_2 &= - (\sin \varphi) e_1 + (\cos\varphi)e_2 - c_0 e_3\\
        b_3 &= e_3.
    \end{split}
\end{align}
The profinite Lie algebroid $\underline{\RR^3}\to \RR^\infty \times \SS^1\times \RR$, with respect to the new frame, has bracket
\begin{equation}\label{eq:postludeNewBasisCommutators}
    [b_1, b_2] = -c_0 b_2,\qquad [b_1, b_3] =[b_2, b_3] = 0. 
\end{equation}
and anchor
\begin{align}\label{eq:postludeAnchorAlpha}
        \rho_\infty(b_1) &= \del_K + \sum_{k=0}^\infty P_k[c_0, \dots, c_k] \del_{c_k},    &
        \rho_\infty(b_2) &= \sum_{k=0}^\infty c_{k+1}\del_{c_k},   &
        \rho_\infty(b_3) &= \del_{\varphi}.
\end{align}
The algebroid decouples as the product of $T\SS^1\to \SS^1$ and $\underline{\RR^2} \to \RR^\infty\times \RR$, where the latter has global frame $\{b_1, b_2\}$. It will be convenient to  set $\Sigma^k:= \RR^k \times \RR$, and denote by 
\[ \mathcal{A}:= \underline{\RR^2} \to \Sigma^\infty:=\RR^\infty\times\RR \] 
the algebroid with global frame $\{b_1, b_2\}$.

The vector field $\rho_\infty(b_1)$ is levelwise profinite in nature, and its flow can be explicitly computed in terms of the solution of a Riccati equation. However, by \cite[Lemma 3.3]{Chetverikov1991}, the vector field $\rho_{\infty}(b_2)$ has no flow defined on profinite open subsets. For this reason, there cannot be a smooth groupoid, whose source fibers are 2-dimensional manifolds; if there was one, the right-invariant vector field corresponding to $b_2$ would have a flow restricted to each source fiber, which would descend to a flow of $\rho_\infty(b_2)$ on the base $\Sigma^\infty$.

The vector field $\rho_\infty(b_2)$ does have a flow on resolutions of the sequence $\Sigma^{k+1}\to \Sigma^k$ other than the profinite limit $\Sigma^\infty$. For this, note that this vector field does not depend on the coordinate $K$ on $\Sigma^\infty$, and lives entirely on $\RR^\infty$, with coordinates $c_\infty = (c_k)_{k\in \NN}$. So let $C^\infty(\RR)$ and $C^\omega(\RR)$ be the spaces of smooth, respectively analytic, functions on $\RR$. Also, we denote by $C^\omega_0(\RR)$ and $C^\infty_0(\RR)$ the corresponding spaces of germs around 0. The jet map $j^\infty_0$ relates the vector field $\del_t$ to the vector field $\rho_\infty(b_2)$, viewed as derivations of these spaces:
\[
    \xymatrix{
         (C^\omega_0(\RR), \del_t) \ar@{_{(}->}[rd]_{j^\infty_0} \ar@{^{(}->}[r] & (C^\infty_0(\RR), \del_t) \ar@{->>}[d]^{j^\infty_0} & (C^\infty(\RR), \del_t) \ar@{->>}[ld]^{j^\infty_0} \ar@{->>}[l] \\
         & (\RR^\infty, \sum c_{k+1} \partial_{c_k}) & 
    }
\]
In other words, these spaces all provide different resolutions with topologies vastly different from the profinite topology on $\RR^\infty$. Also, observe that
\begin{itemize}
    \item On the space of analytic germs $C^\omega_0(\RR)$, the vector field $\del_t$ has a flow $\Phi_t$, given by
    \[
        \Phi_t(\germ_0(f(\cdot))) = \germ_0(f(\cdot + t))
    \]
    whenever $t$ is in the maximal domain to which the germ of $f$ can be extended;
    \item On $C^\infty_0(\RR)$, the vector field $\del_t$ has no well-defined flow, since there are many distinct integral curves through a point, due to the existence of flat functions; 
    \item On $C^\infty(\RR)$, the vector field $\del_t$ is complete with smooth flow $\Phi_t(f) = f(\cdot + t)$. 
    \item On $\RR^\infty$, as we already mentioned, the vector field $\sum_k c_{k+1} \del_{c_k}$ has no flow. 
\end{itemize}
It is natural to pull back the algebroid $\mathcal{A}\to \Sigma^\infty$ to the resolutions on which $\del_t$ has a flow.

Let us focus now on the resolution by analytic germs. The image of the infinite jet map $j^\infty_0\: C^\omega_0(\RR)\to \RR^\infty$ can be identified with the space of convergent power series:
\[
    \RR^\omega := \left\{ c_\infty \in \RR^\infty\: \limsup_k \left( \frac{|c_k|}{k!}\right)^{1/k} <\infty \right\},
\]
and it seems natural to consider on $\RR^\omega$ the smooth structure of $C^\omega_0(\RR)$ rather than the profinite smooth structure on $\RR^\infty$. In a similar vein, we let 
\[
    \Sigma^\omega := \RR^\omega \times \RR \subset \Sigma^\infty. 
\]
It turns out that $\rho_\infty(b_1)$ is also tangent to $\Sigma^\omega$. This is because when $c_\infty = (c_k)_k$ is a convergent power series, then the sequence $P_\infty(c_\infty) := \left(P_k(c_0, \dots, c_k)\right)_k$ also defines a convergent power series (with the same radius of convergence!). Therefore, the algebroid $\mathcal{A}\to \Sigma^\infty$ can be restricted to
\[
    \mathcal{A}^\omega\: \mathcal{A}\big\vert_{\Sigma^\omega} \to \Sigma^\omega.
\]
Due to the existence of flows of both $\rho_\infty(b_1)$ and $\rho_\infty(b_2)$, the algebroid partitions $\Sigma^\omega$ into well-defined (finite-dimensional) leaves. 

In this way, the original classification problem can be solved on this space: it is governed by the $\SS^1$-structure algebroid $T\SS^1\times A^\omega\to \SS^1\times \Sigma^\omega$ whose canonical $\SS^1$-integration is the $\SS^1$-structure groupoid $(\SS^1\times \SS^1)\times \GG^\omega\toto \SS^1 \times \Sigma^\omega$. The source fibers of this groupoid, which are of the form $\SS^1\times s_{\GG^\omega}^{-1}(c_\infty)$, are the coframe bundles of maximal \textit{analytic} simply connected solutions to the realization problem! 

The appropriate smooth structure on the total space of this groupoid has yet to be studied, but we believe it is a type of diffeological groupoid that differentiates to the given algebroid, in the sense of Aintablian and Blohmann \cite{AintablianBlohmann2024}.

\subsection{Solutions with additional symmetry} Even without delving into the complicated theory of infinite-dimensional geometry, a glance at the profinite algebroid $T\SS^1\times \mathcal{A}\to \SS^1\times \Sigma^\infty$ can already lead to interesting insights and solutions. For example, it is possible to find solutions by looking at the locus where the anchor drops rank. This amounts to imposing extra symmetry on realizations.

The expression for the anchor shows that the algebroid $\mathcal{A}\to \Sigma^\infty$ drops rank on the subspace
\[
    \Sigma^\infty_{0} = \{ (c_\infty, K)\in \Sigma^\infty : c_k = 0 \mbox{ for $k\geq 1$} \}.
\]
This locus is \emph{finite-dimensional} and one obtains a restricted algebroid with anchor and bracket given by
\[
    \rho(b_1)= \del_K - (K + c_0^2)\del_{c_0},\qquad \rho(b_2)= 0,\qquad [b_1, b_2] = - c_0b_2. 
\]
Geometrically, we are imposing an extra symmetry, namely translation in the direction orthogonal to $\nabla K$. 

The restricted algebroid $\mathcal{A}\big\vert_{\Sigma^\infty_0}$ is integrable and its source-simply connected integration, by the work of Fernandes and Struchiner \cite{FernandesStruchiner2019}, represents the stack of complete, simply connected, solutions of the realization problem with translational symmetry. The integrating groupoid can be described in terms of solutions to a Riccati equation: for $(c_0, K_0)\in \Sigma^\infty_0$, let $R_{(c_0,K_0)}$ be the unique solution to the Riccati equation
\[
\begin{cases}
\frac{\d R}{\d K} = - K - R^2,\\
            R(K_0)= c_0,
        \end{cases}
\]
which is defined on an open interval $I_{(c_0, K_0)}$. The space of arrows is the manifold
\[
    \GG\big\vert_{\Sigma^\infty_0} = \left\{ (t_1, t_2, c_0, K)\in \RR^2\times \RR^2 \ | \ K+t_1\in I_{(c_0, K)}\right\}.
\]
The source and target are
\[
        \sigma(t_1, t_2, c_0, K)= (c_0, K), \qquad \tau(t_1, t_2, c_0, K)= (R_{(c_0, K)}(K+t_1),K+t_1),
\]
and multiplication is given by
\begin{align*}
        &\left(t_1, t_2, K+t_1, R_{(c_0, K)}(K+t_1)\right)\cdot (s_1, s_2, c_0, K)\\ 
        &= \left(t_1+s_1, \exp\left(- \int_{0}^{s_1} R_{(c_0, K)}(K+ s) \d s\right) t_2 + s_2, c_0, K\right).
\end{align*}
The groupoid $\GG\big\vert_{\Sigma^\infty_0}$ has isotropy $\RR$ and carries a groupoid involution 
\[
    (t_1, t_2, c_0, K)\mapsto (t_1, -t_2, c_0, K),
\]
so the isometry groups of the solutions corresponding to the source fibers are $\ZZ_2\ltimes \RR$.

\subsection{Other solutions}
In general, any finite-rank profinite Lie algebroid can be pulled back to a sufficiently regular integral manifold yielding a finite-dimensional Lie algebroid, which can then be integrated to obtain realizations of the original profinite Lie algebroid. Such integral manifolds can be constructed using flows.

To apply this principle to our example, let us describe explicitly the flow $\Phi^t(c_\infty, K)$ of the vector field $\rho_\infty(b_1)$. Due to the specific shape of this vector field, the flow has components
\[
    \Phi^t(c_\infty, K) = (K + t, (\Phi^t_k(c_k, \dots, c_0, K))_k).
\]
As before, the zeroth component of the flow of $\rho_\infty(b_1)$ is given by
\[
    \Phi^t_0(c_0, K) = R_{(c_0, K)}(K + t), \quad t + K\in I_{(c_0, K)}.
\]
All other components are also defined on the same time interval. This is because for $k\geq 1$ the integral curves obey a first order ODE of the form
\begin{equation*}
        \dot{c}_k = - (k+2) c_k c_0 + F_k(c_1(t), \dots c_{k-1}(t)),
\end{equation*}
which has a solution wherever it is defined.

Integral manifolds of $\mathcal{A}^\omega$ can be constructed from this flow as follows. Given $c_{\infty}\in \RR^\omega$, let $c_\infty(t_2)$ be the maximal integral curve of $\rho_\infty(b_2)$, i.e. $c_\infty(t_2) = j^\infty_{t_2} f$ for some maximally extended analytic function $f$. Then we obtain (in general, non-injective) parameterizations of integral curves
\[
\{ (t_1, t_2)\in \RR^2\: t_2\in I_f, K_0+ t_1 \in I_{(c_0(t_2), K_0)}\}\ni (t_1,t_2) \mapsto \Phi^{t_1}(c_\infty(t_2), K_0)\in \Sigma^\omega.
\]
In particular, we obtain integral manifolds of $\mathcal{A}^\omega$ which are diffeomorphic to:
\begin{itemize}
    \item an interval, when $c_\infty\in \Sigma^\infty_0$;
    \item a cylinder, when $c_\infty$ is the jet of a periodic analytic function;
    \item a contractible open set in $\RR^2$ in all other cases.    
\end{itemize}
Although we currently do not have the techniques to compute the entire leaves of the algebroid $\mathcal{A}^\omega$, the algebroid \emph{can} be pulled back to each of these integral manifolds. The integration of the resulting algebroid (which is just the homotopy groupoids of the integral manifolds) gives solutions to the realization problem. The isometry groups of solutions obtained this way come entirely from the symmetry of the function $f$ associated to the jet $c_\infty$: it can be $\ZZ_2\ltimes \ZZ$, $\ZZ$, $\ZZ_2$ or 1, depending on whether $f$ is periodic and/or has reflectional symmetry.

\section{Outlook}\label{sec:outlook}

In this paper, we have established the foundational framework of the theory of relative algebroids. There are numerous directions to further explore. Below, we outline a few of those and their relationships with existing literature.

\subsection{PDEs with symmetries}\label{subsec:PDEsWithSymmetries}
In the context of the formal theory of PDEs, one of the key advantages of the framework of relative algebroids is its stability under quotients by symmetries. In Section \ref{sec:PDE}, we saw that every PDE has an associated canonical relative algebroid whose realizations, up to diffeomorphism, correspond to solutions of the PDE. As shown in \cite{smilde2025}, symmetries of PDEs lead to symmetries of the underlying relative algebroid. Classification problems in geometry are often governed by PDEs with large symmetry groups. Our ultimate goal is to provide a rigorous explanation of how Bryant's equations arise from a geometric problem formulated as a PDE with symmetries.

We expect that the quotient algebroid associated with a PDE with symmetries is related to various existing approaches to symmetries in the literature. For example, the space of differential invariants of a PDE with (pseudogroup) symmetries (see, e.g., \cite{KoganOlver2001,OlverPohjanpelto2009}) should be related to the exterior algebra of the quotient relative algebroid. 

Another connection to the existing literature arises through Pfaffian fibrations and Pfaffian actions \cite{Accornero2021, AccorneroCattafiCrainicSalazarBook, Cattafi2020, CattafiCrainicSalazar2020, Salazar2013}, which has been established in \cite{smilde2025}.

\subsection{Relative $G$-structure algebroids} The Bryant-Cartan existence result for realizations in the analytic setting (Theorem \ref{thm:CartanBryantExistence}) provides local manifolds with coframes that solve the realization problem. However, in many realization problems, such as those arising from Riemannian manifolds or more general $G$-structures, the realization problem is naturally formulated in terms of coframes on the orthonormal frame bundle or a principal $G$-bundle. In general, the local solutions obtained from the Bryant-Cartan theorem do not yield such principal bundles, as the presence of a structure group is not accounted for. Therefore, it is desirable to incorporate structure groups into the theory of relative algebroids.

In the case of finite-dimensional Lie algebroids, a theory of $G$-structure algebroids and $G$-structure groupoids has been developed in \cite{FernandesStruchiner2019}. One important open question is to find the right framework that incorporates the structure group in realization problems defined by relative algebroids as well.

\subsection{Profinite Lie algebroids} Profinite-dimensional manifolds and bundles appear extensively in the theory of formal PDEs (see, e.g., \cite{Accornero2021, AccorneroCrainic2024, GuneysuPflaum2017, OlverPohjanpelto2005}). Hence, their emergence in the theory of relative algebroids is not surprising: the base space of the prolongation tower of a relative algebroid is the space of formal realizations modulo symmetries. 

There are several challenges regarding the existence of smooth groupoids integrating the prolongation tower of a relative algebroid. In fact, we suspect that such groupoids may not exist for any prolongation tower. Intuitively, this stems from the fact that there is no ``continuous" way to assign a smooth function to each jet. Nevertheless, there are two interesting alternatives.

\begin{itemize}
    \item For real analytic relative Lie algebroids, it is possible to restrict to the space of ``convergent power series". Over this subspace, the profinite algebroid has well-defined leaves and can be leafwise integrated into a groupoid. Understanding the global realizations requires a combination of finite-dimensional Lie integration theory and tools from analytic continuation. The smooth structure of the integrating groupoid determines the moduli stack of global real-analytic realizations. We don't understand the nature of this object yet, but recent work of Aintablian and Blohmann \cite{AintablianBlohmann2024} could shed some light.

    \item Another possible approach, suggested to us by Ivan Contreras, is to consider \emph{formal Lie groupoids} (see, e.g., \cite{CDF05,Contreras13}) integrating a profinite Lie groupoid. Since formal Lie groupoids arise from power series, it seems plausible that integrating objects of such nature for profinite Lie algebroids may exist. 
\end{itemize}

\subsection{Control theory}  
Example \ref{ex:RelativeVectorFields} illustrates how a control system arises from a relative algebroid defined by a relative vector field. The interaction between control theory and relative algebroids should also work in the other direction.  

As a particular example, the notion of controllability (the equivalence relation of points connected by realizations -- see \cite{HermannKerner1977}) should also be present for relative algebroids. This notion should induce a partition of the base of a relative algebroid into invariant submanifolds. This bidirectional interaction suggests deeper connections between the two fields.

\bibliographystyle{abbrv}

\end{document}